%% file: article.tex
\documentclass[review,onefignum,onetabnum]{siamart190516}

\input{shared}

\ifpdf
\hypersetup{
  pdftitle={Convergence of a Decoupled Splitting Scheme for the CHNS System},
  pdfauthor={Chen Liu, Rami Masri, and Beatrice Riviere}
}
\fi

\begin{document}

\maketitle

\begin{abstract}
This paper is devoted to the analysis of an energy-stable  discontinuous Galerkin algorithm for solving the Cahn--Hilliard--Navier--Stokes equations within a decoupled splitting framework. 
We show that the proposed scheme is uniquely solvable and mass conservative. The energy dissipation and the $L^\infty$ stability of the order parameter are obtained under a CFL condition. Optimal a priori error estimates in the broken gradient norm and in the $L^2$ norm are derived.
The stability proofs and error analysis are based on induction arguments and do not require any regularization of the potential function.
\end{abstract}


\input{Content/introduction}

\input{Content/model}

\input{Content/scheme}
\input{Content/solvability}
\input{Content/stability}
\input{Content/error_analysis_induction_const}

\input{Content/error_analysis_improved_const}

\input{Content/numerical_experiments}

\section{Conclusion}
The main contribution of this paper is the analysis of a fully decoupled splitting dG algorithm for solving the CHNS equations. 
Existence and uniqueness of the discrete solution are shown. 
We prove mass conservation, energy stability, and $L^\infty$ stability of the discrete order parameter.  
Optimal a priori error estimates in both the broken $H^1$ norm and the $L^2$ norm are derived, under a CFL condition. 
The analysis relies on induction and requires multiple steps and intermediate results. The analysis is novel 
and the results are stronger because they  do not assume any regularization of the potential function.
Numerical experiments verify the theoretical results.
Extending the numerical analysis to a higher order in time stepping method  will be the object of future work.

\section*{Acknowledgments}
This work used the Extreme Science and Engineering Discovery Environment (XSEDE), which is supported by grants TG-DMS  190021. Specifically, it used the Bridges system, which is supported by NSF award number
ACI-1548562, at the Pittsburgh Supercomputing Center.

\appendix 
\input{Content/appendix.tex}
\bibliographystyle{siamplain}
\bibliography{references}
\end{document}

%% file: shared.tex
\usepackage{lipsum}
\usepackage{amsfonts}
\usepackage{graphicx}
\usepackage{epstopdf}
\usepackage{algorithmic}
\ifpdf
  \DeclareGraphicsExtensions{.eps,.pdf,.png,.jpg}
\else
  \DeclareGraphicsExtensions{.eps}
\fi

\newsiamremark{remark}{Remark}
\newsiamremark{hypothesis}{Hypothesis}
\crefname{hypothesis}{Hypothesis}{Hypotheses}
\newsiamthm{claim}{Claim}

\input{commands}                     

\headers{Numerical Analysis of a Splitting Scheme for CHNS System}{Chen Liu, Rami Masri, and Beatrice Riviere}

\title{Convergence of a Decoupled Splitting Scheme for the Cahn--Hilliard--Navier--Stokes System}

\author{Chen Liu\thanks{Department of Mathematics,  Purdue University, 150 North University Street, West Lafayette, Indiana 47907 (\email{liu3373@purdue.edu}).}
\and Rami Masri\thanks{Department of Numerical Analysis and Scientific Computing, Simula Research Laboratory, Oslo 0164, Norway (\email{rami@simula.no}).}
\and Beatrice Riviere\thanks{Department of Computational Applied Mathematics and Operations Research, Rice University, 6100 Main Street, Houston, Texas 77005 (\email{riviere@rice.edu}). Partially supported by NSF-DMS 2111459.}}

\usepackage{amsopn}


%% file: commands.tex
\usepackage{cancel}
\usepackage{mathtools}

\usepackage{xcolor}

\usepackage{bbm}             
\usepackage[cmintegrals]{newpxmath}                              

\usepackage{bm}                                                  

\usepackage{dsfont}                                              
  \newcommand{\IR}{\ensuremath\mathds{R}}                        
  \newcommand{\IP}{\ensuremath\mathds{P}}                        

\newcommand*{\setE}{\ensuremath{\mathcal{T}}}                    
\newcommand*{\Gammah}{\Gamma_h}                                  
\newcommand*{\Nst}{N_\mathrm{T}}                                 
\renewcommand*{\vec}[1]{{\boldsymbol{#1}}}                       
\DeclareMathAlphabet{\mathbfsf}{\encodingdefault}{\sfdefault}{bx}{n}
\newcommand*{\normal}{\vec{n}}                                   
\newcommand*{\dd}{\mathrm{d}}                                    
\newcommand*{\grad}{\vec{\nabla}}                                
\renewcommand*{\div}{\vec{\nabla}\cdot}                          
\newcommand*{\gradh}{\vec{\nabla}_h}                             
\newcommand*{\divh}{\vec{\nabla}_h\cdot}                         

\newcommand*{\laplace}{\upDelta}                                 
\newcommand*{\strain}{\mathcal{D}}                

\newcommand*{\avg}[1]{\{{#1}\}}                                  
\newcommand*{\jump}[1]{[{#1}]}                                   
\newcommand*{\abs}[1]{\ensuremath{|#1|}}                         
\newcommand*{\norm}[2]{\|#1\|_{#2}}                              
\newcommand*{\on}[2]{\left.#1\right\vert_{#2}}                   
\newcommand{\adif}{a_\mathrm{diff}}
\newcommand{\aadv}{a_\mathrm{adv}}
\newcommand{\erru}{\bm{\xi}_{\bm{u}}}
\newcommand{\errv}{\bm{\xi}_{\bm{v}}}

\newcommand{\DG}{\mathrm{DG}}

\newcommand{\vertiii}[1]{{\left\vert\kern-0.25ex\left\vert\kern-0.25ex\left\vert #1 
    \right\vert\kern-0.25ex\right\vert\kern-0.25ex\right\vert}}%
\newsiamremark{assumption}{Assumption}

\usepackage{multirow}
\usepackage{tabularx}                                            
  \newcolumntype{R}{>{\raggedleft\arraybackslash}X}
  \newcolumntype{L}{>{\raggedright\arraybackslash}X}
  \newcolumntype{C}{>{\centering\arraybackslash}X}
\usepackage{booktabs}                                            

\usepackage{cleveref}                       
  \crefformat{equation}{Eq.~(#2#1#3)}
  \Crefformat{equation}{Equation~(#2#1#3)}
  \crefformat{figure}{Fig.~#2#1#3}          \crefname{figure}{Fig.}{Figs.}
  \Crefformat{figure}{Figure~#2#1#3}        \Crefname{figure}{Figure}{Figures}
  \crefformat{assumption}{Asm.~#2#1#3}
  \Crefformat{assumption}{Assumption~#2#1#3}
  \crefformat{lemma}{Lem.~#2#1#3}
  \Crefformat{lemma}{Lemma~#2#1#3}
  \crefformat{remark}{Rem.~#2#1#3}
  \Crefformat{remark}{Remark~#2#1#3}
  \crefformat{section}{Sec.~#2#1#3}         \crefname{section}{Sec.}{Secs.}
  \Crefformat{section}{Section~#2#1#3}      \Crefname{section}{Section}{Sections}
  \crefformat{table}{Tab.~#2#1#3}           \crefname{table}{Tab.}{Tables}
  \Crefformat{table}{Table~#2#1#3}          \Crefname{table}{Table}{Tables}
  \crefformat{theorem}{Thm.~#2#1#3}         \crefname{theorem}{Thm.}{Thms.}
  \Crefformat{theorem}{Theorem~#2#1#3}      \Crefname{theorem}{Theorem}{Theorems}

%% file: Content/introduction.tex
\section{Introduction}
The Cahn--Hilliard--Navier--Stokes (CHNS) system serves as a fundamental phase-field model extensively used  in many fields of science and engineering. 
The simulation of the CHNS equations is a challenging computational task primarily because of: 
(\emph{i}) the coupling of highly nonlinear equations; and
(\emph{ii}) the requirement of preserving certain physical principles, such as conservation of mass and dissipation of energy.
A common approach to overcome these difficulties is to decouple the mass and momentum equations, and to further split the nonlinear convection from the incompressibility constraint \cite{shen2012modeling}. 
The splitting scheme constructed from this strategy only requires the successive solution of several simpler equations at each time step. Thus, such an algorithm is  both convenient for programming and efficient in large-scale simulations.
A non-exhaustive list of several computational papers on the CHNS model include \cite{dong2012time,BaoShi2012,ChenShen2016,liu2020efficient,ZhaoHan2021,liu2022pressure}. 
\par
The analysis of semi-discrete  spatial formulations with continuous and discontinuous Galerkin (dG) methods for solving the CHNS equations has been extensively investigated. 
Without being exhaustive, we refer to the papers  \cite{feng2006fully,styles2008finite,diegel2017convergence,LiuRiviere2018numericalCHNS} for the study of fully coupled schemes. 
For decoupled splitting algorithms based on projection methods, we mention a few papers \cite{HanWang2015,cai2018error,shen2010numerical,shen2010phase}. 
Han and Wang in \cite{HanWang2015} introduce a second order in time scheme and show unique solvability, but this work does not contain any theoretical proof of convergence.
Cai and Shen in \cite{cai2018error} formulate an energy-stable scheme and show convergence based on a compactness argument. In this work, in order to obtain energy dissipation, the authors introduced an additional stabilization term. Similar stabilizing strategies can be found in \cite{shen2010numerical,shen2010phase}. Although this technique enforces a discrete energy law, it also introduces an extra consistency error.
A major difficulty in proving optimal convergence error rates of a numerical scheme for the CHNS system arises from the nonlinear potential function. A widely used regularization technique is to truncate the potential and to extend it  with a quadratic growth \cite{cai2018error,shen2012modeling,li2020discontinuous}. An important objective of our work is to obtain a rigorous convergence analysis of the scheme without modifying and regularizing the potential function.
To the best of our knowledge, the theoretical analysis of a decoupled splitting scheme in conjunction with interior penalty dG discretization without any regularization on the potential function is not available in literature.
\par
The main contribution of this work is the stability and error analysis of a dG discretization of a splitting scheme for the CHNS model. 
We prove the energy stability, the $L^\infty$ stability of the order parameter, and we derive the optimal a priori error bounds in both the broken gradient norm and the $L^2$ norm. 
Our analysis is novel and general  in sense that: 
(\emph{i}) we successfully avoid using any artificial stabilizing terms (which introduce an extra consistency error) when discretizing the CHNS system, and
(\emph{ii}) no regularization (truncation and extension) assumptions on the potential function are needed for the analysis.
The proofs are technical and rely on induction arguments. A priori error bounds are valid for convex domains because the convergence
analysis utilizes dual problems.  Our arguments can be extended to analyze splitting algorithms for other type of phase-field models.
\par
The outline of this paper is as follows. In Section~\ref{sec:CHNS_model}, the CHNS mathematical model is presented. In Section~\ref{sec:scheme}, we introduce the fully discrete numerical scheme. The unique solvability of the scheme is proved in Section~\ref{sec:solvability}. We show that our scheme is energy stable in Section~\ref{sec:stability}, and we derive error estimates in Section~\ref{sec:error_analysis}. Numerical experiments validating our theoretical results are presented in  Section~\ref{sec:numerical_experimants}. Concluding remarks follow.

%% file: Content/model.tex
\section{Model problem}\label{sec:CHNS_model}
Let $\Omega\subset\IR^d$ ($d=2$ or $3$) be an open bounded polyhedral domain and let $\normal$ denote the unit outward normal to the boundary $\partial\Omega$.
In the context of incompressible immiscible two-phase flows, we introduce a scalar field order parameter as a phase indicator, which is defined as the difference between mass fractions.
The unknown variables in the CHNS system are the order parameter $c$, the chemical potential $\mu$, the velocity $\vec{u}$, and the pressure $p$, satisfying:
\begin{subequations}\label{eq:CHNS:model}
  \begin{align}
    \partial_t{c} - \Delta \mu  + \div{(c\vec{u})}&= 0 && \text{in}~(0,\, T]\times\Omega,\label{eq:CHNS:model1}\\
    \mu &= \Phi'(c) - \kappa\, \laplace{c} && \text{in}~(0,\, T]\times\Omega,\label{eq:CHNS:model2}\\
    \partial_t{\vec{u}} + \vec{u}\cdot\grad{\vec{u}} - \mu_\mathrm{s}\laplace{\vec{u}} &= -\grad{p} - c\grad{\mu} && \text{in} ~(0,\, T]\times\Omega,\label{eq:CHNS:model3}\\
    \div{\vec{u}} &= 0 && \text{in} ~(0,\, T]\times\Omega.\label{eq:CHNS:model4}
  \end{align}
\end{subequations}
The parameter $\kappa$ and shear viscosity $\mu_\mathrm{s}$ are positive constants. The Ginzburg--Landau potential function $\Phi$ is
defined by:
\begin{equation}\label{eq:CHNS_GL_potential}
  \Phi(c) = \frac{1}{4}(1-c)^2(1+c)^2.
\end{equation}
This polynomial potential can be decomposed into the sum of a convex part $\Phi_{+}$ and a concave part $\Phi_{-}$. We have:
\begin{align*}
  \Phi = \Phi_{+} + \Phi_{-}, \quad\text{where}~~
  \Phi_{+} = \frac{1}{4}(1+c^4) ~~\text{and}~~
  \Phi_{-} = -\frac{1}{2}c^2.
\end{align*}
We supplement our model problem \eqref{eq:CHNS:model} with the following initial and boundary conditions:
\begin{subequations}\label{eq:proj_CHNS:ini_boundary_conditions}
  \begin{align}
    c &= c^0, & \vec{u} &= \vec{u}^0 & && &\text{on}~\{0\}\times\Omega, \label{eq:proj_CHNS:model:initial_c}\\
    \grad{c}\cdot\normal &= 0, & \grad{\mu}\cdot\normal &= 0, & \vec{u} = \vec{0} && &\text{on}~(0,\, T]\times\partial\Omega. \label{eq:proj_CHNS:model:BC}
  \end{align}
\end{subequations}
%
Let $\overline{c_0}$ denote the average of the initial order parameter. 
The model problem~\eqref{eq:CHNS:model} satisfies the global mass conservation property:
  \begin{equation}\label{eq:model_mass_conservation}
    \frac{1}{\abs{\Omega}} \int_\Omega c = \frac{1}{\abs{\Omega}} \int_\Omega c^0 = \overline{c_0},
  \end{equation}
as well as the energy dissipation property \cite{shenYang2015,FrankLiuAlpakRiviere2018}. Let $F$ denote the total energy of the system.
  \begin{equation}\label{eq:CHNS:energy}
    F(c,\vec{u})=\int_\Omega \frac{1}{2}\abs{\vec{u}}^2
    + \int_\Omega \Big(\Phi(c) + \frac{\kappa}{2}\abs{\grad c}^2\Big), \quad \frac{\dd}{\dd t} F(c,\vec{u}) \leq 0.
  \end{equation}

We end this section by briefly stating the functional setting used throughout the paper.
For a given real number $p\geq1$, on  a domain $\mathcal{O}\in \IR^d$, where $d=2$ or $3$, the standard notation for the $L^p(\mathcal{O})$ spaces is employed. Let $(\cdot,\cdot)_\mathcal{O}$ denote the $L^2$ inner product over $\mathcal{O}$. We also define
\begin{align*}
  L^{2}_0(\mathcal{O}) = \{\omega\in L^2(\mathcal{O}):~(\omega,1)_\mathcal{O}=0\}.
\end{align*}
Let $D^\vec{\alpha}$ denote the weak $\vec{\alpha}$-th partial derivative with multi-index $\vec{\alpha}$. For a given integer $m\geq0$, the Sobolev space $W^{m,p}(\mathcal{O})$ is defined by
\begin{align*}
  W^{m,p}(\mathcal{O}) = \{\omega\in L^p(\mathcal{O}):~D^\vec{\alpha}\omega \in L^p(\mathcal{O}),~\forall\abs{\vec{\alpha}}\leq m\}.
\end{align*}
The usual Sobolev semi-norm $\abs{\cdot}_{W^{m,p}(\mathcal{O})}$ and norm $\norm{\cdot}{W^{m,p}(\mathcal{O})}$ are employed. We introduce the space $H^m(\mathcal{O}) = W^{m,2}(\mathcal{O})$ with the associated semi-norm $\abs{\cdot}_{H^{m}(\mathcal{O})}=\abs{\cdot}_{W^{m,2}(\mathcal{O})}$ and norm $\norm{\cdot}{H^{m}(\mathcal{O})}=\norm{\cdot}{W^{m,2}(\mathcal{O})}$.
For convenience, we use $(\cdot,\cdot)$  and $\|\cdot\|$ to denote the $L^2$ inner product and the $L^2$ norm, when $\mathcal{O}$ is the whole computational domain.

%% file: Content/scheme.tex
\section{Scheme}\label{sec:scheme}
Let $\setE_h = \{E_i\}$ be a family of conforming nondegenerate (regular) quasi-uniform meshes of the computational domain $\Omega$ with the maximum element diameter $h$. 
 The mesh consists of simplices or of parallelepiped (parallelograms for $d=2$). 
Let $\Gammah$ denote the set of interior faces. For each interior face $e \in \Gammah$ shared by elements $E_{i^-}$ and $E_{i^+}$, with $i^- < i^+$, we define a unit normal vector $\normal_e$ that points from $E_{i^-}$ into $E_{i^+}$. For a boundary face $e$, i.\,e., $e = \partial E_{i^-} \cap \partial\Omega$, the normal $\normal_e$ is taken to be the unit outward vector to $\partial\Omega$. We also denote by $\normal_E$ the unit normal vector outward to the element $E$. We introduce the broken Sobolev spaces, $s\geq1$,
\begin{equation*}
H^s(\setE_h) = \big\{\omega\in L^2(\Omega):~\forall E \in \setE_h,\, \on{\omega}{E} \in H^s(E) \big\}. 
\end{equation*}
The average and jump operators of any scalar function $\omega \in H^s(\mathcal{T}_h)$ is defined for each interior face $e\in\Gammah$ by
\begin{align*}
\avg{\omega}|_e = \frac{1}{2}\on{\omega}{E_{i^-}}\! + \frac{1}{2}\on{\omega}{E_{i^+}}, \quad
\jump{\omega}|_e = \on{\omega}{E_{i^-}}\! - \on{\omega}{E_{i^+}}, \quad e = \partial E_{i^-} \cap \partial E_{i^+}.
\end{align*}
If $e$ belongs to the boundary $\partial\Omega$, the jump and average of $\omega$ coincide with its trace on $e$. The related definitions of any vector quantity in $H^s(\setE_h)^d$ are similar~\cite{riviere2008}. 
Fix an integer $k\geq 1$ and denote by $\IP_{k}(E)$ the set of all polynomials of degree at most $k$ on an element $E$.
Define the following discontinuous polynomial spaces for simplicial meshes: 
\begin{align*}
M_h^k &= \big\{\omega_h\in L^2(\Omega):~\forall E \in \setE_h,\, \on{\omega_h}{E} \in \IP_{k}(E) \big\}, \\
\quad M_{h0}^k &= \big\{\omega_h\in M_h^k:~(\omega_h,1) = 0 \big\},\\
\mathbf{X}_h^k &= \big\{\vec{\theta}_h \in L^2(\Omega)^d:~\forall E \in \setE_h,\, \on{\vec{\theta}_h}{E} \in \IP_{k}(E)^d \big\}.
\end{align*}
 For meshes with parallelograms or parallelepipeds, the space $\mathbb{Q}_k(E)$, namely the space of tensor product polynomials of degree at most $k$ on an element $E$, is used instead of $\mathbb{P}_k(E)$ in the above definitions. 
We now present the dG pressure projection algorithm for solving \eqref{eq:CHNS:model} with initial and boundary conditions \eqref{eq:proj_CHNS:ini_boundary_conditions}. Uniformly partition $[0,T]$ into $\Nst$ intervals with length equal to $\tau$ and 
for any $1 \leq n \leq \Nst$ let $t^n = n \tau$ and  let $\delta_\tau$ be the temporal backward finite difference operator $\delta_\tau{c_h^n} = (c_h^n-c_h^{n-1})/\tau$. The scheme consists of four sequential steps. \\ Given $(c_h^{n-1},\vec{u}_h^{n-1}) \in M_h^{k}\times\mathbf{X}_h^{k}$, compute $(c_h^n,\mu_h^n) \in M_h^{k}\times M_h^{k}$, such that for all $\chi_h \in M_h^{k}$ and for all $\varphi_h \in M_h^{k}$
\begin{align}
(\delta_\tau c_h^n,\chi_h) + a_{\mathrm{diff}}(\mu_h^n,\chi_h) + a_\mathrm{adv}(c_h^{n-1},\vec{u}_h^{n-1},\chi_h) &= 0,\label{eq:fully_dis1}\\
\big(\Phi_{+}\,\!'(c_h^n)+\Phi_{-}\,\!'(c_h^{n-1}),\varphi_h\big) + \kappa\,a_{\mathrm{diff}}(c_h^n,\varphi_h) - (\mu_h^n,\varphi_h) &= 0.\label{eq:fully_dis2}
\end{align}
Second, given $(c_h^{n-1},\mu_h^n,\vec{u}_h^{n-1},p_h^{n-1}) \in M_h^{k}\times M_h^{k}\times\mathbf{X}_h^{k}\times M_h^{k-1}$, compute $\vec{v}_h^{n} \in \mathbf{X}_h^{k}$, such that for all $\vec{\theta}_h \in \mathbf{X}_h^{k}$
\begin{multline}
\frac{1}{\tau}(\vec{v}^n_h - \vec{u}^{n-1}_h,\vec{\theta}_h) + a_{\mathcal{C}}(\vec{u}^{n-1}_h,\vec{u}^{n-1}_h,\vec{v}^n_h,\vec{\theta}_h) + \mu_\mathrm{s}a_\mathcal{D}(\vec{v}^n_h,\vec{\theta}_h)  
\\ = b_{\mathcal{P}}(\vec{\theta}_h, p_h^{n-1}) + b_{\mathcal{I}}(c_h^{n-1},\mu_h^n,\vec{\theta}_h).\label{eq:fully_dis3} 
\end{multline}
Next, given $\vec{v}_h^n \in \mathbf{X}_h^{k}$, compute $\phi_h^n \in M_{h0}^{k-1}$, such that for all $\varphi_h \in M_{h0}^{k-1}$
\begin{align}
a_{\mathrm{diff}}(\phi^n_h,\varphi_h) = -\frac{1}{\tau}b_{\mathcal{P}}(\vec{v}^n_h, \varphi_h).\label{eq:fully_dis4}
\end{align}
Finally, given $(\vec{v}_h^{n},p_h^{n-1},\phi_h^n) \in \mathbf{X}_h^{k}\times M_h^{k-1}\times M_{h0}^{k-1}$, compute $(\vec{u}^n_h, p_h^n) \in \mathbf{X}_h^{k}\times M_h^{k-1}$, such that for all $\chi_h \in M_h^{k-1}$ and for all $\vec{\theta}_h \in \mathbf{X}_h^{k}$
\begin{align}
(p^n_h,\chi_h) &= (p^{n-1}_h,\chi_h) + (\phi^n_h,\chi_h) - \sigma_\chi\mu_\mathrm{s}b_{\mathcal{P}}(\vec{v}^n_h, \chi_h),\label{eq:fully_dis5}\\
(\vec{u}^n_h,\vec{\theta}_h) &= (\vec{v}^n_h,\vec{\theta}_h) + \tau b_{\mathcal{P}}(\vec{\theta}_h, \phi^n_h).\label{eq:fully_dis6}
\end{align}
For the approximation of the initial values, let $\vec{u}_h^0$ be the $L^2$ projection of $\vec{u}^0$ and let
$c_h^0$ be the elliptic projection of $c^0$, namely $c_h^0$ satisfies
\begin{equation}
a_\mathrm{diff}(c_h^0-c^0,\chi_h) = 0,~\forall \chi_h\in M_h^{k},~\text{with constraint}~(c_h^0-c^0,1) = 0. \label{eq:ch0ellip}
\end{equation}
In addition, we set $p_h^0 = \phi_h^0 = 0$  and $\bm{v}_h^0 = \bm{u}_h^0$.  The parameter $\sigma_\chi$ is a (user-specified) positive number 
that can be chosen between $0$ and $1/(4d)$. 
\par 
The forms $a_\mathrm{diff}$ and $a_\mathcal{D}$ are the SIPG discretizations of the scalar and vector Laplace operator, $-\laplace{\omega}$ and $-\laplace{\vec{v}}$, respectively.
Let $\tilde{\sigma}\geq 1, \sigma\geq 1$ be given penalty parameters. We define 
\begin{align}
a_{\mathrm{diff}}&(\omega,\chi) =
\sum_{E\in\setE_h} \int_E \grad \omega \cdot \grad \chi
-\sum_{e\in\Gammah} \int_e \avg{\grad \omega \cdot \normal_e} \jump{\chi}\label{eq:CHNS:DG_diffusion}\\
&-\sum_{e\in\Gammah} \int_e \avg{\grad \chi \cdot \normal_e} \jump{\omega}
+ \frac{\tilde{\sigma}}{h} \sum_{e\in\Gammah}\int_e \jump{\omega}\jump{\chi}, \quad \forall \omega,\chi\in H^2(\mathcal{T}_h), \nonumber \\
a_\mathcal{D}&(\vec{v}, \vec{\theta}) = 
\sum_{E\in\setE_h} \int_E \grad{\vec{v}}:\grad{\vec{\theta}} 
- \sum_{e\in\Gammah\cup\partial\Omega} \int_e\avg{\grad{\vec{v}}\,\normal_e}\cdot\jump{\vec{\theta}} \label{eq:CHNS:DG_strain}\\
&-\sum_{e\in\Gammah\cup\partial\Omega} \int_e\avg{\grad{\vec{\theta}}\,\normal_e}\cdot\jump{\vec{v}}
+ \frac{\sigma}{h}\sum_{e\in\Gammah\cup\partial\Omega} \int_e \jump{\vec{v}}\cdot\jump{\vec{\theta}}, \quad\forall \vec{v}, \vec{\theta}\in H^2(\mathcal{T}_h)^d. \nonumber 
\end{align}
The dG form $a_{\mathcal{C}}: H^2(\setE_h)^d\times H^2(\setE_h)^d\times H^2(\setE_h)^d\times H^2(\setE_h)^d\rightarrow \IR$ of the convection term $\vec{v}\cdot\grad{\vec{v}}$ is
\begin{align}\label{eq:CHNS:DG_convection}
a_{\mathcal{C}}(\vec{w},\vec{v},\vec{z},\vec{\theta}) =&
\sum_{E\in\setE_h}\Big(\int_E (\vec{v}\cdot\grad\vec{z})\cdot\vec{\theta} 
+ \int_{\partial E_{-}^\vec{w}} \abs{\avg{\vec{v}}\cdot\normal_E}\,(\vec{z}^\mathrm{int}-\vec{z}^\mathrm{ext})\cdot\vec{\theta}^\mathrm{int}\Big)\\
&+ \frac{1}{2} \sum_{E\in\setE_h}\int_E (\div{\vec{v}})\,\vec{z}\cdot{\vec{\theta}}
- \frac{1}{2} \sum_{e\in\Gammah\cup\partial\Omega}\int_e \jump{\vec{v}\cdot\normal_e}\avg{\vec{z}\cdot\vec{\theta}}. \nonumber 
\end{align}
Here, the set $\partial E_{-}^\vec{w}$ is the inflow part of $\partial E$, defined by 
$\partial E_{-}^\vec{w} = \big\{\vec{x}\in\partial E: \avg{\vec{w}(\vec{x})}\cdot\normal_E <0 \big\}$,
and the superscript $\mathrm{int}$ (resp. $\mathrm{ext}$) refers to the trace of the function on a face of $E$ coming from the interior of $E$ (resp. coming from the exterior of $E$ on that face). In addition, if the face lies on the boundary of the domain, we take the exterior trace to be zero.
The discretization of the linear advection term $\div{(c\vec{v})}$ is done with the dG form $a_{\mathrm{adv}}: H^2(\setE_h)\times H^2(\setE_h)^d\times H^2(\setE_h)\rightarrow \IR$:
\begin{equation}\label{eq:CHNS:DG_advection} 
a_{\mathrm{adv}}(c,\vec{v},\chi) = 
-\sum_{E\in\setE_h} \int_E c\vec{v}\cdot\grad{\chi}
+\sum_{e\in\Gammah} \int_e \avg{c} \avg{\vec{v}\cdot\normal_e} \jump{\chi}.
\end{equation}
The dG form $b_\mathcal{I}: H^2(\setE_h)\times H^2(\setE_h)\times H^2(\setE_h)^d\rightarrow \IR$ of the interface term $-c\grad{\mu}$ is equal to $a_\mathrm{adv}$ with switched arguments:
\begin{align}\label{eq:CHNS:DG_interface}
b_\mathcal{I}(c,\mu,\vec{\theta}) = a_\mathrm{adv}(c,\vec{\theta},\mu).
\end{align}
Finally, for the discretization of the gradient and divergence terms, such as $-\grad{p}$, $-\grad{\phi}$, and $\div{\vec{v}}$, we introduce the dG bilinear form $b_{\mathcal{P}}: H^2(\setE_h)^d \times H^1(\setE_h) \rightarrow \IR$:
\begin{align}\label{eq:CHNS:DG_pressure}
b_{\mathcal{P}}(\vec{\theta},p) =
\sum_{E\in\setE_h} \int_E p\div{\vec{\theta}}
-\sum_{e\in\Gammah\cup\partial\Omega} \int_e \avg{p}\jump{\vec{\theta}\cdot\normal_e}.
\end{align}
With Green's theorem, an equivalent expression for $b_{\mathcal{P}}$ is:
\begin{align}\label{eq:CHNS:DG_pressure2}
b_{\mathcal{P}}(\vec{\theta},p) =
-\sum_{E\in\setE_h} \int_E \vec{\theta} \cdot \grad{p}
+\sum_{e\in\Gammah} \int_e \avg{\vec{\theta}\cdot\normal_e}\jump{p}.
\end{align}
\par
The broken space $H^1(\setE_h)$ and discrete space $M_h^k$  are equiped with the semi-norm $\vert \cdot \vert_{\mathrm{DG}}$.
\begin{align*}
\vert \omega\vert_{\mathrm{DG}}^2 = \sum_{E\in\setE_h} \norm{\grad{\omega}}{L^2(E)}^2 +  \frac{\tilde{\sigma}}{h} \sum_{e\in\Gammah} \norm{\jump{\omega}}{L^2(e)}^2, \quad \forall \omega \in H^1(\setE_h).
\end{align*}
Note, $\vert \cdot\vert_{\mathrm{DG}}$ is a norm on $H^1(\setE_h)\cap L^2_0(\Omega)$. The vector space $H^1(\setE_h)^d$ is equiped with the following norm:
\begin{align*}
\norm{\vec{v}}{\mathrm{DG}}^2 = \sum_{E\in\setE_h} \norm{\grad{\vec{v}}}{L^2(E)}^2 + \frac{\sigma}{h} \sum_{e\in\Gammah\cup\partial\Omega} \norm{\jump{\vec{v}}}{L^2(e)}^2, \quad\forall\vec{v}\in H^1(\setE_h)^d.
\end{align*}
We now recall several properties satisfied by the dG forms. The forms $a_\mathrm{diff}$ and $a_\mathcal{D}$ are coercive.
There exist $\tilde{\sigma}_0$ and $\sigma_0$ such that for all $\tilde{\sigma}\geq \tilde{\sigma}_0$
and $\sigma\geq\sigma_0$, there exist $K_\alpha>0$ and $K_\mathcal{D}>0$ independent of $h$ such that
\begin{align}
K_\alpha \vert\omega_h\vert_{\mathrm{DG}}^2 &\leq a_{\mathrm{diff}}(\omega_h,\omega_h), \quad \forall \omega_h\in M_h^k,  \label{eq:coercivity_adiff}\\
K_\strain \norm{\vec{v}_h}{\mathrm{DG}}^2 &\leq a_\strain(\vec{v}_h,\vec{v}_h), \quad \forall\vec{v}_h\in \mathbf{X}_h^k. \label{eq:coercivity_astrain}
\end{align}
 Without loss of generality, we can set $K_\alpha = K_\strain = 1/2$. Indeed, this is true  if $\tilde{\sigma}_0 = \sigma_0  = 2N_{\mathrm{face}} C^2_\mathrm{tr}$ where  $C_\mathrm{tr}$ is the trace constant in
\eqref{eq:trace_estimate_discrete}  (that depends on $k$)  and $N_\mathrm{face}$ is the maximum number of faces of a mesh element.  Let us denote 
\begin{equation}\label{eq:tildeMk}
\tilde{M}_k = 2N_{\mathrm{face}} C^2_\mathrm{tr}.
\end{equation} 
Continuity also holds for these two forms.  There exist  constants $C_\alpha>0$ and $C_\strain>0$ independent of mesh size $h$, such that 
\begin{align}
\abs{a_{\mathrm{diff}}(\omega_h,\chi_h)} &\leq C_\alpha \vert\omega_h\vert_{\mathrm{DG}} \vert\chi_h\vert_{\mathrm{DG}}, \quad \forall \omega_h, \chi_h\in M_h^k, \label{eq:adiffcont}\\
\abs{a_\strain(\vec{v}_h,\vec{\theta}_h)} &\leq C_\strain\norm{\vec{v}_h}{\mathrm{DG}} \norm{\vec{\theta}_h}{\mathrm{DG}}, \quad
\forall \vec{v}_h, \vec{\theta}_h \in \mathbf{X}_h^k. \label{rq:continuity_a_D}
\end{align}
The form $a_\mathrm{adv}$ satisfies the following bounds (see Lemma~3.2 in \cite{LiuRiviere2018numericalCHNS}). For all $ c_h, \chi_h\in M_h^k,$ and for all $\vec{v}_h\in \mathbf{X}_h^k.$ 
\begin{align}
\vert a_{\mathrm{adv}}(c_h,\vec{v}_h,\chi_h)\vert &\leq C_\gamma \Big( \vert c_h \vert_\mathrm{DG} 
+ \vert\int_\Omega c_h\vert \Big) \Vert \vec{v}_h\Vert_\mathrm{DG} \, \vert \chi_h\vert_{\mathrm{DG}}, \label{eq:bound_aadv_1}  \\ 
|\aadv(c_h, \bm{v}_h, \chi_h )| &\leq C_\gamma \Big(\vert c_h\vert_{\mathrm{DG}} + \vert\int_{\Omega} c_h\vert \Big) \|\bm{v}_h\|^{1/2} \|\bm{v}_h\|_{\DG}^{1/2}\vert\chi_h\vert_{\mathrm{DG}}. \label{eq:bound_aadv_2}  
\end{align}
The form $a_{\mathcal{C}}$ satisfies the following positivity property:
\begin{align}\label{eq:aCformpos}
a_{\mathcal{C}}(\vec{v}_h,\vec{v}_h,\vec{\theta}_h,\vec{\theta}_h)\geq 0, \quad \forall \vec{v}_h,  \vec{\theta}_h\in \mathbf{X}_h^k.
\end{align}
We define lift operators $R_h: \mathbf{X}_h^{k}\rightarrow M_h^{k-1}$ and $\vec{G}_h: M_h^{k-1}\rightarrow\mathbf{X}_h^{k}$ by 
\begin{alignat*}{2}
(R_h(\jump{\vec{\theta}_h}), q_h) &= \sum_{e\in\Gammah\cup\partial{\Omega}} \int_e \{q_h\} [\bm{\theta}_h] \cdot \bm{n}_e, && \quad \forall\vec{\theta}_h\in \mathbf{X}_h^{k},~ \forall q_h \in M_h^{k-1}, \\
(\vec{G}_h(\jump{\chi_h}), \bm{\theta}_h) &= \sum_{e\in\Gammah} \int_e \{\bm{\theta}_h\}\jump{\chi_h}, && \quad \forall\chi_h\in M_h^{k-1},~ \forall \vec{\theta}_h\in \mathbf{X}_h^{k}.
\end{alignat*}
 One can easily obtain the bounds  \cite{inspaper1,di2010discrete},  recalling the definition \eqref{eq:tildeMk} 
\begin{align}
       \|R_h(\jump{\bm{\theta}_h})\|&\leq  \tilde{M}_{k-1}  \Big( h^{-1} \sum_{e \in \Gamma_h \cup \partial \Omega} \|\jump{\bm{\theta}_h}\|_{L^2(e)}^2\Big)^{1/2},&& \quad\forall \bm{\theta}_h \in \mathbf{X}^{k}_h, \label{eq:lift_prop_r} \\ 
       \|\bm{G}_h(\jump{q_h})\| &\leq \tilde{M}_{k} \Big(h^{-1}\sum_{e \in \Gamma_h} \|\jump{q_h}\|_{L^2(e)}^2\Big)^{1/2}, &&\quad\forall q_h \in M^{k-1}_h. \label{eq:lift_prop_g}
       \end{align}
With the lift operators, we can rewrite $b_\mathcal{P}$ into two different ways:
\begin{align}
b_{\mathcal{P}}(\vec{\theta}_h, p_h) &= ( \divh{\vec{\theta}_h}, p_h ) - (R_h(\jump{\vec{\theta}_h}), p_h), && \quad \forall \vec{\theta}_h\in\mathbf{X}_h^{k}, \, \forall p_h \in M_h^{k-1},\label{eq:def_b_lift} \\ 
b_{\mathcal{P}}(\vec{\theta}_h, p_h) &= -(\gradh{p_h}, \vec{\theta}_h) + (\vec{G}_h(\jump{p_h}), \vec{\theta}_h),&&  \quad \forall \vec{\theta}_h\in\mathbf{X}_h^{k}, \, \forall p_h \in M_h^{k-1}. \label{eq:def_b_lift_2}
\end{align}
In the above, $\divh$ and $\gradh$ denote the broken divergence and gradient operators respectively. 
\begin{proposition}\label{thm:dis_mass_conservation}
The dG pressure correction scheme satisfies the discrete global mass conservation property, i.\,e., for any $1\leq n \leq \Nst$, we have
\begin{equation}\label{thm:CHNS:discrete_mass_conservation}
(c_h^n,1) = (c_h^{0},1) = (c^0,1) = (c(t^n),1), \quad \forall 1\leq n\leq \Nst.
\end{equation}
\end{proposition}
\begin{proof}
The first equality  is  obtained by choosing $\chi_h = 1$ in \eqref{eq:fully_dis1} and by using $a_{\mathrm{diff}}(\mu_h^n,1)=0$ and $a_{\mathrm{adv}}(c_h^{n-1},\vec{u}_h^{n-1},1)=0$. The second equality is a constraint on the initial order parameter and the third equality is 
from \eqref{eq:model_mass_conservation}.
\end{proof}
\par
We finish this section by recalling Poincar\'e's, inverse and trace inequalities.
For $p < +\infty$ when $d = 2$  and for $p \leq 6$ when $d = 3$,  we have \cite{girault2016strong,di2010discrete}: 
\begin{align}
\norm{\omega_h - \frac{1}{\abs{\Omega}}\int_\Omega \omega_h}{L^p(\Omega)} &\leq C_\mathrm{P} \vert \omega_h\vert_{\mathrm{DG}}, && \forall \omega_h \in M_h^k,  
\label{eq:PoincareMh} \\ 
 \Vert \phi \Vert_{L^p(\Omega)} & \leq C_\mathrm{P} \left( \vert \phi \vert_{\mathrm{DG}}^2 +  \frac{1}{\vert \Omega\vert}
\Big|\int_\Omega \phi \Big|^2 \right)^{1/2}, &&\forall \phi \in H^1(\setE_h), \label{eq:PoincareBroken}\\ 
\norm{\vec{v}_h}{L^p(\Omega)} 
& \leq C_\mathrm{P} \norm{\vec{v}_h}{\mathrm{DG}}, && \forall \vec{v}_h \in \mathbf{X}^k_h.  \label{eq:poincare_ineq}
\end{align}
Here, $C_\mathrm{P}>0$ denote a constant independent of the mesh size $h$.
The following are two well-known inverse inequalities, for $p,q \in [1,\infty]$ and $p \geq q $ ,
\begin{align}
\norm{\vec{v}_h}{L^p(\Omega)} &\leq C_\mathrm{inv}h^{d/p-d/q}\norm{\vec{v}_h}{L^q(\Omega)}, &&  \forall \vec{v}_h \in \mathbf{X}_h^k, \label{eq:inverse_estmate_lp}\\
\norm{\vec{v}_h}{\mathrm{DG}} &\leq C_\mathrm{inv}h^{-1}\norm{\vec{v}_h}{}, &&   \forall \vec{v}_h \in \mathbf{X}_h^k, \label{eq:inverse_estimate_dg}
\end{align} 
where $C_\mathrm{inv}$ is a constant independent of $h$. We also use the following trace estimates: 
\begin{align}
\|\bm{v}_h\|_{L^r(e)} \leq  C_\mathrm{tr}  h^{-1/r} \|\bm{v}_h\|_{L^r(E)}, \quad \forall \bm{v}_h \in \mathbf{X}_h^k,~ r \geq 1,~ e \in \partial E, \label{eq:trace_estimate_discrete} \\ 
\|\bm{v}\|_{L^2(e)} \leq  \tilde{C}_\mathrm{tr}  h^{-1/2}(\|\bm{v}\|_{L^2(E)} + h\|\grad \bm{v}\|_{L^2(E)}), \quad \forall \bm{v} \in H^1(\mathcal{T}_h),~ e \in \partial E. \label{eq:trace_estimate_continuous}   
\end{align}
We remark that the above inverse and trace estimates also hold for scalar valued functions. For brevity, we may refer to the above bounds for both scalar and vector valued functions.
%
%
For any function $\xi$ we denote by $\overline\xi$ the average of $\xi$:
\[
\overline \xi = \frac{1}{\vert \Omega\vert}\int_\Omega \xi.
\]
We end this section by stating a discrete Sobolev inequality. 
\begin{lemma}
Define $\Theta_{h,d}$ as follows:
\begin{equation} 
       \Theta_{h,d} = \begin{cases}
         (1 + |\ln(h)|)^{1/2}, & d = 2, \\ 
           h^{-1/2}, & d = 3. 
       \end{cases} \label{eq:def_Theta_hd}
\end{equation} 
There exists a constant $\tilde{C}_P$ independent of $h$ such that 
\begin{equation}       
\|\omega_h - \overline{\omega_h} \|_{L^{\infty}(\Omega)} 
\leq \tilde{C}_P \Theta_{h,d} |\omega_h|_{\DG}, \quad \forall \omega_h \in M_h^k. \label{eq:sobolev_ineq} 
\end{equation}
\end{lemma}
\begin{proof}
For $d=2$, bound \eqref{eq:sobolev_ineq} is proved in \cite{brenner2004discrete} (see (19) in Theorem 5). For $d=3$, inequality \eqref{eq:sobolev_ineq} follows from \eqref{eq:inverse_estmate_lp} (with $p = \infty$ and $q=6$) and \eqref{eq:PoincareMh}. 
\[\|\omega_h - \overline{\omega_h}\|_{L^{\infty}(\Omega)} \leq C_\mathrm{inv}h^{-1/2} \|\omega_h - \overline{\omega_h}\|_{L^6(\Omega)} \leq C_\mathrm{inv} C_\mathrm{P} h^{-1/2}|\omega_h|_{\DG}.\]
\end{proof}
Throughout the paper, $C$ denotes a generic constant that takes different values at different places and that is independent of mesh size $h$ and time step size $\tau$.

%% file: Content/solvability.tex
\section{Existence and uniqueness}\label{sec:solvability} We show that the discrete problem \eqref{eq:fully_dis1}-\eqref{eq:fully_dis6} is well posed in three steps. First, we show  equivalence of \eqref{eq:fully_dis1}-\eqref{eq:fully_dis2} to a problem posed in $M_{h0}^k$ in Lemma \ref{lem:extint}. Similar to the arguments found in \cite{HanWang2015,LiuRiviere2018numericalCHNS}, we then use the Browder--Minty theorem to show existence and uniqueness of a solution to \eqref{eq:CHNS:ProblemQ:a}-\eqref{eq:CHNS:ProblemQ:b} in Lemma \ref{lem:CH:solvability:B-M}. 
The existence of solutions to \eqref{eq:fully_dis3}-\eqref{eq:fully_dis6} follow by showing uniqueness in Lemma \ref{lem:exivh}.
\begin{lemma}\label{lem:extint}
The unique solvability of \eqref{eq:fully_dis1}-\eqref{eq:fully_dis2} is equivalent to the unique solvability of the following problem: given $(c_h^{n-1},\vec{u}_h^{n-1}) \in M_h^{k}\times\mathbf{X}_h^{k}$, compute $(y_h^n, w_h^n) \in M_{h0}^{k}\times M_{h0}^{k}$, such that for all $\mathring{\chi}_h \in M_{h0}^{k}$ and for all $\mathring{\varphi}_h \in M_{h0}^{k}$
\begin{align}
(\delta_\tau y_h^n,\mathring{\chi}_h) + a_{\mathrm{diff}}(w_h^n,\mathring{\chi}_h) + a_\mathrm{adv}(c_h^{n-1},\vec{u}_h^{n-1},\mathring{\chi}_h) &= 0, \label{eq:CHNS:ProblemQ:a}\\
(\Phi_{+}\,\!'(y_h^n+\overline{c_0})+\Phi_{-}\,\!'(c_h^{n-1}),\mathring{\varphi}_h) + \kappa a_{\mathrm{diff}}(y_h^n,\mathring{\varphi}_h) - (w_h^n,\mathring{\varphi}_h) &= 0, \label{eq:CHNS:ProblemQ:b}
\end{align}
where $y_h^{n-1} = c_h^{n-1} - \overline{c_0}$.
\end{lemma}
\begin{proof}
Assume that the system \eqref{eq:CHNS:ProblemQ:a}-\eqref{eq:CHNS:ProblemQ:b} has a unique solution $(y_h^n, w_h^n)$. 
Let $f(y_h^n) =  \Phi_+\,\!'(y_h^n + \overline{c_0}) + \Phi_{-}\,\!'(c_h^{n-1})$. Define $c_h^n = y_h^n + \overline{c_0}$ and $\mu_h^n = w_h^n + \overline{f(y_h^n)}$. To see that $(c_h^n, \mu_h^n)$ solves \eqref{eq:fully_dis1}-\eqref{eq:fully_dis2}, take any $\chi_h \in M_h^{k}$ and let $\mathring{\chi}_h   = \chi_h - \overline{\chi_h}\in M_{h0}^{k}$ in \eqref{eq:CHNS:ProblemQ:a}. Then, since $a_{\mathrm{diff}}(z_h, q) = a_{\mathrm{adv}}(c_h^{n-1}, \bm{u}_h^{n-1}, q)= 0$ for any constant $q$ and any $z_h \in M_h^{k}$, we have by \eqref{eq:CHNS:ProblemQ:a} 
\begin{equation}
(\delta_\tau y_h^n, \chi_h - \overline{\chi_h} )+ a_\mathrm{diff}(\mu_h^n, \chi_h) + a_{\mathrm{adv}}(c_h^{n-1}, \bm{u}_h^{n-1}, \chi_h) = 0, \quad \forall \chi_h \in M_{h}^{k}. 
\end{equation} 
By Proposition~\ref{thm:dis_mass_conservation} and the definitions of $c_h^n$ and $y_h^{n-1}$,  we have 
\begin{align}
    (\delta_{\tau} y_h^n, \chi_h - \overline{\chi_h}) = (\delta_{\tau} c_h^n, \chi_h - \overline{\chi_h}) = (\delta_{\tau} c_h^n, \chi_h). 
\end{align}
Hence, \eqref{eq:fully_dis1} is satisfied. Similarly, let $\mathring{\phi}_h = \phi_h - \overline{\phi_h}$ for any $\phi_h \in M_h^{k}$. Since $w_h^n= \mu_h^n - \overline{f(y_h^n)}  \in M_{h0}^{k}$, we have 
\begin{align}
    (f(y_h^n), \phi_h &- \overline{\phi_h} ) - (\mu_h^n - \overline{f(y_h^n)}, \phi_h - \overline{\phi_h} ) 
     = (f(y_h^n),\phi_h)  -(\mu_h^n, \phi_h).
\end{align}
This implies that \eqref{eq:fully_dis2} is satisfied. To see that this solution is unique, assume that there exists a different pair $(c_h^{n,1}, \mu_h^{n,1})$ that satisfies \eqref{eq:fully_dis1}-\eqref{eq:fully_dis2}. Define $(y_h^{n,1}, w_h^{n,1}) = (c_h^{n,1} - \overline{c_0}, \mu_h^{n,1} - \overline{\mu_h^n})$. This pair also solves \eqref{eq:CHNS:ProblemQ:a}-\eqref{eq:CHNS:ProblemQ:b} which is a contradiction. Hence, the solution to \eqref{eq:fully_dis1}-\eqref{eq:fully_dis2} is unique. 

Conversely, assume that \eqref{eq:fully_dis1}-\eqref{eq:fully_dis2} has a unique solution $(c_h^n, \mu_h^n)$. Then, it is easy to see that $(y_h^n, w_h^n ) = (c_h^n - \overline{c_0}, \mu_h^n - \overline{\mu_h^n})$ solves \eqref{eq:CHNS:ProblemQ:a}-\eqref{eq:CHNS:ProblemQ:b}. To show uniqueness, assume there is a different pair $(y_h^{n,1}, w_h^{n,1})$ which solves \eqref{eq:CHNS:ProblemQ:a}-\eqref{eq:CHNS:ProblemQ:b}. Then, the pair $(c_h^{n,1}, \mu_h^{n,1}) = \big(y_h^{n,1} + \overline{c_0}, w_h^{n,1} + \overline{f(y_h^{n,1})}\big)$ also solves \eqref{eq:fully_dis1}-\eqref{eq:fully_dis2}. This provides a contradiction. Hence, the solution to \eqref{eq:fully_dis1}-\eqref{eq:fully_dis2} is unique. 
\end{proof}
\begin{lemma}\label{lem:CH:solvability:B-M}
The system \eqref{eq:CHNS:ProblemQ:a}-\eqref{eq:CHNS:ProblemQ:b} is uniquely solvable for any fixed time step size $\tau$ and mesh size $h$.
\end{lemma}
\begin{proof}
For any $w_h \in M_{h0}^{k}$, let $y_h = y_h(w_h)\in M_{h0}^k$ be the unique function such that
\begin{align}\label{eq:CH:solvability:uniq_c}
\big(\Phi_{+}\,\!'(y_h+\overline{c_0})+\Phi_{-}\,\!'(c_h^{n-1}),\mathring{\varphi}\big) &+ \kappa a_{\mathrm{diff}}(y_h,\mathring{\varphi}) - (w_h,\mathring{\varphi}) = 0, && \forall \mathring{\varphi} \in M_{h0}^{k}.
\end{align}
Existence and uniqueness of such a $y_h(w_h)$ is proved in Lemma~3.14 of \cite{LiuRiviere2018numericalCHNS}.
If we choose $\mathring{\varphi}_h=y_h$ in \cref{eq:CH:solvability:uniq_c} and use the coercivity property \eqref{eq:coercivity_adiff},  we obtain
\begin{equation}\label{eq:CH:uniq_G_boundedness2}
\big(\Phi_{+}\,\!'(y_h+\overline{c_0})+\Phi_{-}\,\!'(c_h^{n-1}),y_h\big) + K_\alpha \kappa \vert y_h\vert_{\mathrm{DG}}^2 \leq (w_h,y_h).
\end{equation}
By convexity of $\Phi_+$ and the fact that $(y_h,1) = 0$, we have
\begin{equation}\label{eq:CH:Phi_Taylor2}
\big(\Phi_{+}\,\!'(y_h+\overline{c_0}),y_h\big) = \big(\Phi_{+}\,\!'(\overline{c_0}),y_h\big) + \big(\Phi_{+}\,\!''(\xi_h),y_h^2\big) = \big(\Phi_{+}\,\!''(\xi_h),y_h^2\big) \geq 0.
\end{equation}
Therefore,  \eqref{eq:PoincareMh} and Cauchy--Schwarz's inequality yield
\begin{equation}\label{eq:CH:uniq_G_boundedness3}
\vert y_h(w_h) \vert_{\mathrm{DG}} \leq \frac{C_\mathrm{P}^2}{K_\alpha\kappa}\vert w_h \vert_{\mathrm{DG}}
+ \frac{C_\mathrm{P}}{K_\alpha\kappa}\norm{\Phi_{-}\,\!'(c_h^{n-1})}{}.
\end{equation}
Let $M_{h0}^{k\ast}$ denote the dual space of $M_{h0}^k$ and let 
$\mathcal{F}:~M_{h0}^{k}\rightarrow M_{h0}^{k\ast}$ be a mapping defined as follows: for all $\mathring{\chi}_h$ in $M_{h0}^{k}$,
\begin{align*}
\langle\mathcal{F}(w_h),\mathring{\chi}_h\rangle = \big(y_h(w_h)-y_h^{n-1},\mathring{\chi}_h\big) + \tau a_{\mathrm{diff}}(w_h,\mathring{\chi}_h) + \tau a_\mathrm{adv}(c_h^{n-1},\vec{u}_h^{n-1},\mathring{\chi}_h).
\end{align*}
First, let us check the boundedness of $\mathcal{F}$. For any $\mathring{\chi}_h \in M_{h0}^{k}$ with $\vert \mathring{\chi}_h \vert_{\mathrm{DG}}=1$, by Cauchy--Schwarz's inequality, \eqref{eq:poincare_ineq} with $p=2$, \eqref{eq:adiffcont}, \eqref{eq:bound_aadv_1}, we have
\begin{align}\label{eq:CH:uniq_G_boundedness1}
\abs{\langle\mathcal{F}(w_h),\mathring{\chi}_h\rangle} 
\leq C_\alpha\tau \vert w_h \vert_{\mathrm{DG}} + C_\mathrm{P}^2(\vert y_h \vert_{\mathrm{DG}} + \vert y_h^{n-1}\vert_{\mathrm{DG}}) \\
+ C_\gamma \tau (\vert c_h^{n-1} \vert_{\mathrm{DG}} + \abs{\Omega}\overline{c_0} \, )\norm{\vec{u}_h^{n-1}}{\mathrm{DG}}. \nonumber
\end{align}
Using \eqref{eq:CH:uniq_G_boundedness1} and \eqref{eq:CH:uniq_G_boundedness3}, we have
\begin{multline}
\norm{\mathcal{F}(w_h)}{M_{h0}^{k\ast}} 
\leq \Big(C_\alpha\tau + \frac{C_\mathrm{P}^4}{K_\alpha\kappa}\Big) \vert w_h \vert_{\mathrm{DG}} 
+ \frac{C_\mathrm{P}^3}{K_\alpha\kappa}\norm{\Phi_{-}\,\!'(c_h^{n-1})}{} \\
+ C_\mathrm{P}^2\vert y_h^{n-1} \vert_{\mathrm{DG}} + C_\gamma \tau (\vert c_h^{n-1} \vert_{\mathrm{DG}} + \abs{\Omega}\overline{c_0} \, )\norm{\vec{u}_h^{n-1}}{\mathrm{DG}}.
\label{eq:boundG}
\end{multline}
Thanks to \eqref{eq:boundG}, we have shown that 
the operator $\mathcal{F}$ maps bounded sets in $M_{h0}^{k}$ to bounded sets in $M_{h0}^{k\ast}$\,, i.e., we have proved boundedness of the operator.
Second, we show the coercivity of $\mathcal{F}$. 
With \eqref{eq:coercivity_adiff}, \eqref{eq:poincare_ineq} and \eqref{eq:bound_aadv_1}, we have
\begin{align*}
\langle\mathcal{F}(w_h),w_h\rangle 
&\geq (y_h, w_h) +  \tau \vert w_h \vert_{\mathrm{DG}}^2 - C_\mathrm{P} \Vert y_h^{n-1}\Vert \, \vert w_h \vert_{\mathrm{DG}} \\
&- C_\gamma \tau (\vert c_h^{n-1} \vert_{\mathrm{DG}} + \abs{\Omega}\overline{c_0} \, )\norm{\vec{u}_h^{n-1}}{\mathrm{DG}}\vert w_h\vert_{\mathrm{DG}}. 
\end{align*}
With \eqref{eq:CH:uniq_G_boundedness2}, \eqref{eq:CH:Phi_Taylor2}, \eqref{eq:poincare_ineq} and Young's inequality, we have
\begin{equation}\label{eq:CH:uniq_G_coercivity2}
(y_h,w_h) \geq K_\alpha \kappa \vert y_h\vert_{\mathrm{DG}}^2 + (\Phi_{-}'(c_h^{n-1}),y_h)
 \geq -\frac{C_\mathrm{P}^2}{4K_\alpha\kappa}\norm{\Phi_{-}\,\!'(c_h^{n-1})}{}^2.
\end{equation}
Therefore we obtain 
%
\begin{align*}
\langle\mathcal{F}(w_h),w_h\rangle 
\geq&\, \tau \vert w_h \vert_{\mathrm{DG}}^2 - C_\mathrm{P} \norm{y_h^{n-1}}{}\vert w_h \vert_{\mathrm{DG}} - \frac{C_\mathrm{P}^2}{4K_\alpha\kappa}\norm{\Phi_{-}\,\!'(c_h^{n-1})}{}^2\\
&- C_\gamma \tau (\vert c_h^{n-1} \vert_{\mathrm{DG}} + \abs{\Omega}\overline{c_0} \, )\norm{\vec{u}_h^{n-1}}{\mathrm{DG}}\vert w_h\vert_{\mathrm{DG}}. 
\end{align*}
This implies coercivity of $\mathcal{F}$: 
\begin{equation*}
\lim_{\vert w_h\vert_{\mathrm{DG}}\rightarrow +\infty} \frac{\langle\mathcal{F}(w_h),w_h\rangle}{\vert w_h\vert_{\mathrm{DG}}}= +\infty.
\end{equation*}
%
Third, let us check the monotonicity of $\mathcal{F}$. For any $w_h$ and $s_h$ in $M_{h0}^{k}$, we have:
\begin{equation}\label{eq:CH:monotonicity_1}
\langle\mathcal{F}(w_h)-\mathcal{F}(s_h),w_h-s_h\rangle = (y_h(w_h)-y_h(s_h),w_h-s_h) + \tau a_{\mathrm{diff}}(w_h-s_h,w_h-s_h).
\end{equation}
Due to the coercivity of $a_{\mathrm{diff}}$, the second term in the right-hand side is nonnegative, which means we only need to check the sign of the first term. From \eqref{eq:CH:solvability:uniq_c}, we have, for any $\mathring{\varphi}_h \in M_{h0}^{k}$:
\begin{multline*}
(w_h-s_h,\mathring{\varphi}_h) = \big(\Phi_{+}\,\!'(y_h(w_h)+\overline{c_0} \, )-\Phi_{+}\,\!'(y_h(s_h)+\overline{c_0} \, ),\mathring{\varphi}_h\big) \\
+ \kappa a_{\mathrm{diff}}\big(y_h(w_h)-y_h(s_h),\mathring{\varphi}_h\big).
\end{multline*}
Choosing $\mathring{\varphi}_h = y_h(w_h)-y_h(s_h) \in M_{h0}^{k}$ and 
using \eqref{eq:coercivity_adiff} and the convexity of $\Phi_+$, we have
\begin{align}\label{eq:CH:uniq_G_monotonicity}
\big(w_h-s_h,y_h(w_h)-y_h(s_h)\big) 
\geq K_\alpha\kappa \vert y_h(w_h)-y_h(s_h) \vert_{\mathrm{DG}}^2
\geq 0.
\end{align}
Substituting \eqref{eq:CH:uniq_G_monotonicity} into \eqref{eq:CH:monotonicity_1}, considering $\vert \cdot \vert_{\mathrm{DG}}$ is a norm in $M_{h0}^{k}$\,, the following inequality is strict whenever $w_h \neq s_h$,
\begin{equation*}
\langle\mathcal{F}(w_h)-\mathcal{F}(s_h), w_h-s_h\rangle \geq K_\alpha\tau\vert w_h-s_h\vert_{\mathrm{DG}}^2 \geq 0.
\end{equation*}
Thus we have established the strict monotonicity of $\mathcal{F}$.
Finally, let us check the continuity of $\mathcal{F}$. For any $\mathring{\chi}_h \in M_{h0}^{k}$ with $\vert{\mathring{\chi}_h}\vert_{\mathrm{DG}}=1$, by Cauchy--Schwarz's inequality, \eqref{eq:poincare_ineq}, \eqref{eq:adiffcont}, we have
\begin{equation}\label{eq:CH:uniq_G_continuity1}
\abs{\langle\mathcal{F}(w_h)-\mathcal{F}(s_h),\mathring{\chi}_h\rangle}
\leq C_\alpha\tau \vert w_h-s_h \vert_{\mathrm{DG}} + C_\mathrm{P}^2 \vert y_h(w_h)-y_h(s_h) \vert_{\mathrm{DG}}.
\end{equation}
To bound the second term in the right-hand side, we revert to \eqref{eq:CH:uniq_G_monotonicity} and use \eqref{eq:poincare_ineq} to obtain:
\begin{align}\label{eq:CH:uniq_G_continuity2}
\vert y_h(w_h)-y_h(s_h) \vert_{\mathrm{DG}} \leq \frac{C_\mathrm{P}^2}{K_\alpha\kappa} \vert w_h-s_h \vert_{\mathrm{DG}}.
\end{align}
Combining \eqref{eq:CH:uniq_G_continuity1} and \eqref{eq:CH:uniq_G_continuity2}, we have
\begin{align*}
\norm{\mathcal{F}(w_h)-\mathcal{F}(s_h)}{M_{h0}^{k\ast}} 
 =  \sup_{\substack{\mathring{\chi}_h \in M_{h0}^{k} \\ \vert \mathring{\chi}_h\vert_{\mathrm{DG}}=1}} \abs{\langle\mathcal{F}(w_h)-\mathcal{F}(s_h),\mathring{\chi}_h\rangle} 
 \leq  \Big(C_\alpha\tau + \frac{C_\mathrm{P}^4}{K_\alpha\kappa}\Big)\vert w_h-s_h\vert_{\mathrm{DG}},
\end{align*}
which means $\norm{\mathcal{F}(w_h)-\mathcal{F}(s_h)}{M_{h0}^{k\ast}}$ tends to zero whenever $\vert w_h-s_h\vert_{\mathrm{DG}}$ tends to zero, i.\,e., we proved the continuity of the operator $\mathcal{F}$.
We can then apply the Browder--Minty theorem to  conclude that there exists a unique solution $w_h^n$ such that $\langle\mathcal{F}(w_h^n),\mathring{\chi}_h\rangle = 0$ for all $\mathring{\chi}_h \in M_{h0}^{k}$. This implies that $(y_h(w_h^n),w_h^n)$ is the unique solution of system \eqref{eq:CHNS:ProblemQ:a}--\eqref{eq:CHNS:ProblemQ:b}.
\end{proof}
\begin{lemma}\label{lem:exivh}
Given $(c_h^{n-1},\mu_h^n,\vec{v}_h^{n-1},\vec{u}_h^{n-1}, p_h^{n-1}) \in M_{h0}^{k} \times M_{h0}^{k} \times \vec{X}_h^{k} \times \vec{X}_h^{k} \times M_{h0}^{k-1}$, there exists a unique solution $(\vec{v}_h^n, \vec{u}_h^{n}, p_h^n) \in \vec{X}_h^{k} \times \vec{X}_h^{k} \times M_{h0}^{k-1}$ to problem \eqref{eq:fully_dis3}-\eqref{eq:fully_dis6}. 
\end{lemma} 
\begin{proof}
Let us first show that $p_h^n$ belongs to $M_{h0}^{k-1}$ by induction on $n$. The statement trivially holds for $n=0$. Assume $p_h^{n-1} \in M_{h0}^{k-1}$. Then, take $\chi_h = 1$ in \eqref{eq:fully_dis5}. Since $\phi_h^n$ has zero average, we have
\begin{equation*}
\int_{\Omega} p_h^n = \int_\Omega p_h^{n-1} -\sigma_\chi \mu_s b_\mathcal{P}(\vec{v}_h^n,1) = 0.
\end{equation*}
The last equality is obtained by applying the induction hypothesis and the fact that $b_\mathcal{P}(\vec{v},1)=0$ for any $\vec{v}$.
Second, let us show the existence of the intermediate velocity $\vec{v}_h^n$. Since this is a linear problem in finite dimensions, it suffices to show uniqueness of the solution. Suppose there exist two solutions $\vec{v}_h^n$ and $\vec{\tilde{v}}_h^n$ 
and let $\vec{w}^n_h = \vec{v}_h^n - \vec{\tilde{v}}_h^n$ denote the difference. Choosing $\vec{\theta}_h = \vec{w}_h^n$ in \eqref{eq:fully_dis3}
and using \eqref{eq:aCformpos} and \eqref{eq:coercivity_astrain} yield
\begin{equation*}
\norm{\vec{w}_h^n}{}^2 + K_\strain\tau\mu_\mathrm{s}\norm{\vec{w}_h^n}{\mathrm{DG}}^2 \leq 0.
\end{equation*}
This implies that $\vec{w}_h^n = {\bf 0}$, which yields uniqueness and existence of $\vec{v}_h^n$.
The existence of $\phi_h^n \in M_{h0}^{k-1}$ follows by similar arguments, and the fact that $|\cdot|_\mathrm{DG}$ is a norm for the space $M_{h0}^{k-1}$.
Existence and uniqueness of $p_h^n$ and $\vec{u}_h^n$ is trivial.
\end{proof}
Lemma~\ref{lem:extint}, Lemma~\ref{lem:CH:solvability:B-M} and Lemma~\ref{lem:exivh} imply existence and uniqueness of a solution to problem~\eqref{eq:fully_dis1}-\eqref{eq:fully_dis6}.

%% file: Content/stability.tex
\section{Stability}\label{sec:stability}
We will prove that our scheme satisfies a discrete energy dissipation property in two steps. First we obtain the energy dissipation
property under an assumption on the boundedness of the order parameter. Second, we use an induction argument to
show that this assumption holds true.
Define the discrete total energy at time $t^n$ as follows.
\begin{equation}\label{eq:CHNS:discrete_energy}
F_h(c_h^n,\vec{u}_h^n) = \frac{1}{2}(\vec{u}_h^n,\vec{u}_h^n) + \big(\Phi(c_h^n),1\big) + \frac{\kappa}{2}a_{\mathrm{diff}}(c_h^n,c_h^n).
\end{equation}
We can bound the initial discrete energy by assuming enough regularity on the initial conditions and by using stability
of the $L^2$ projection and elliptic projection. 
\begin{equation*}
F_h(c_h^0,\vec{u}_h^0) \leq \frac12\norm{\vec{u}^0}{}^2 + C\norm{c^0}{H^1(\Omega)}^4 + C  \norm{c^0}{H^1(\Omega)}^2 + C \leq C.
\end{equation*}
%
%
We introduce auxiliary variables that play an important role in the analysis, similar variables were introduced in \cite{pyo2013error,inspaper1}.
Define $S_h^0 = 0$ and $\zeta_h^0 = p_h^0$. For any $1\leq n\leq\Nst$, we define $S_h^n \in M_{h0}^{k} $ and $\zeta_h^n \in M_{h0}^{k}$ as follows:
\begin{equation}
S_h^n = \sigma_\chi\mu_\mathrm{s} \sum_{i=1}^n\Big(\divh{\vec{v}_h^i} - R_h(\jump{\vec{v}_h^i})\Big), 
\quad
\zeta_h^n = p_h^n + S_h^n.\label{eq:defShzetah}
\end{equation}
We now show a discrete energy dissipation inequality. 
\begin{theorem}\label{thm:CHNS:dis_energy_dissipation}
Assume that $\tau$ and $h$ satisfy the following CFL condition. 
\begin{equation}\label{eq:CHNS:dis_energy_dissipation_CFL}
    \tau \leq \min \Big(  \frac{K_\alpha}{ 8 \tilde{C}_P^2 \max{(\frac{2}{K_\alpha\kappa},2)} F_h(c_h^0,\vec{u}_h^0)} \Theta_{h,d}^{-2},~ \frac{K_\alpha}{8 \overline{c_0}^2}\Big), 
    \end{equation}
where $\Theta_{h,d}$ is given \eqref{eq:def_Theta_hd}. Fix $n\geq 1$ and assume that $c_h^{n-1}$ satisfies the bound:
\begin{align}\label{eq:CHNS:dis_energy_dissipation_ass}
\vert c_h^{n-1}\vert_{\mathrm{DG}}^2 \leq \max{\left(\frac{2}{K_\alpha\kappa},2\right)} F_h(c_h^0,\vec{u}_h^0).
\end{align}
If $\sigma > M_{k-1}^2/d,\, \tilde{\sigma} >  \tilde{M}_k^2$, and $\sigma_\chi < K_\strain/(2d)$, 
the discrete dissipation inequality holds: 
\begin{align}
&\,F_h(c_h^n, \vec{u}_h^n) + \frac{\tau}{2\sigma_\chi\mu_\mathrm{s}}\norm{S_h^n}{}^2 + \frac{\tau^2}{2}a_{\mathrm{diff}}(\zeta_h^n,\zeta_h^n) \label{eq:CHNS:dis_energy_dissipation_ineq} \\
-&\, F_h(c_h^{n-1}, \vec{u}_h^{n-1}) - \frac{\tau}{2\sigma_\chi\mu_\mathrm{s}}\norm{S_h^{n-1}}{}^2 - \frac{\tau^2}{2}a_{\mathrm{diff}}(\zeta_h^{n-1},\zeta_h^{n-1})\nonumber \\
\leq&\, - \frac{K_\alpha\tau}{2}\vert \mu_h^n\vert_{\mathrm{DG}}^2 - \frac{K_\strain\mu_\mathrm{s}\tau}{2}\norm{\vec{v}_h^{n}}{\mathrm{DG}}^2 - \frac{1}{4}\norm{\vec{v}_h^n-\vec{u}_h^{n-1}}{}^2. \nonumber 
\end{align}
\end{theorem}
\begin{proof}
Fix $n\geq 1$ and take $\chi_h=\mu_h^n$ in \eqref{eq:fully_dis1}, $\varphi_h=\delta_\tau c_h^n$ in \eqref{eq:fully_dis2}, $\vec{\theta}_h=\vec{v}_h^n$ in \eqref{eq:fully_dis3}
\begin{eqnarray*}
(\delta_\tau c_h^n,\mu_h^n) + a_{\mathrm{diff}}(\mu_h^n,\mu_h^n) + a_{\mathrm{adv}}(c_h^{n-1},\vec{u}_h^{n-1},\mu_h^n) = 0, \\
(\Phi_{+}\,\!'(c_h^n)+\Phi_{-}\,\!'(c_h^{n-1}),\,\delta_\tau c_h^n) + \kappa a_{\mathrm{diff}}(c_h^n,\delta_\tau c_h^n) - (\mu_h^n,\delta_\tau c_h^n) = 0, \\
\frac{1}{\tau}(\vec{v}^n_h - \vec{u}^{n-1}_h,\vec{v}_h^n)
+ a_{\mathcal{C}}(\vec{u}_h^{n-1},\vec{u}_h^{n-1},\vec{v}_h^n,\vec{v}_h^n) + \mu_\mathrm{s} a_\strain(\vec{v}_h^n, \vec{v}_h^n) \hspace*{5.15em}\\ 
=  b_{\mathcal{P}}(\vec{v}_h^n, p_h^{n-1}) + b_\mathcal{I}(c_h^{n-1},\mu_h^n,\vec{v}_h^n).
\end{eqnarray*}
Adding the equations above, and using \eqref{eq:aCformpos}, \eqref{eq:coercivity_adiff} and \eqref{eq:coercivity_astrain},  we have
\begin{multline}\label{eq:CHNS:Dis_energy_1}
\frac{1}{\tau}(\vec{v}^n_h - \vec{u}^{n-1}_h,\vec{v}_h^n) + \big(\Phi_{+}\,\!'(c_h^n)+\Phi_{-}\,\!'(c_h^{n-1}),\,\delta_\tau c_h^n\big) + \kappa a_{\mathrm{diff}}(c_h^n,\delta_\tau c_h^n) \\
+ K_\alpha\vert \mu_h^n \vert_{\mathrm{DG}}^2 + K_\strain \mu_\mathrm{s} \norm{\vec{v}_h^n}{\mathrm{DG}}^2 
= b_{\mathcal{P}}(\vec{v}_h^n, p_h^{n-1}) - a_{\mathrm{adv}}(c_h^{n-1},\vec{u}_h^{n-1},\mu_h^n) + b_\mathcal{I}(c_h^{n-1},\mu_h^n,\vec{v}_h^n).
\end{multline}
With Taylor's expansions and the fact that $\Phi_+$ is convex and $\Phi_-$ is concave, we have for some
$\xi_h$ and $\eta_h$ between $c_h^{n-1}$ and $c_h^n$
\begin{align}
&\big(\Phi_{+}\,\!'(c_h^n)+\Phi_{-}\,\!'(c_h^{n-1}),\,\delta_\tau c_h^n\big) \label{eq:CHNS:Dis_energy_2} \big(\delta_\tau\Phi(c_h^n),\,1\big) + \frac{1}{2\tau}\big(\Phi_{+}\,\!''(\xi_h), (c_h^{n-1}-c_h^n)^2\big)
\\
&=\,- \frac{1}{2\tau}\big(\Phi_{-}\,\!''(\eta_h), (c_h^n-c_h^{n-1})^2\big) \geq  \big(\delta_\tau\Phi(c_h^n),\,1\big). \nonumber
\end{align}
With \eqref{eq:CHNS:Dis_energy_2} and the symmetry of $a_\mathrm{diff}$, we obtain
\begin{align}\label{eq:CHNS:Dis_energy_5}
&\frac{1}{2\tau}\norm{\vec{v}_h^n}{}^2 - \frac{1}{2\tau}\norm{\vec{u}_h^{n-1}}{}^2 + \frac{1}{2\tau}\norm{\vec{v}_h^n-\vec{u}_h^{n-1}}{}^2 + \big(\delta_\tau\Phi(c_h^n),\,1\big)\\
&+ \frac{\kappa}{2\tau}a_{\mathrm{diff}}(c_h^n,c_h^n) - \frac{\kappa}{2\tau}a_{\mathrm{diff}}(c_h^{n-1},c_h^{n-1}) + \frac{\kappa}{2\tau}a_{\mathrm{diff}}(c_h^n-c_h^{n-1},c_h^n-c_h^{n-1}) \nonumber\\
&+ K_\alpha\vert \mu_h^n\vert_{\mathrm{DG}}^2 + K_\strain \mu_\mathrm{s} \norm{\vec{v}_h^n}{\mathrm{DG}}^2 
\leq  b_{\mathcal{P}}(\vec{v}_h^n, p_h^{n-1}) \nonumber\\
&- a_{\mathrm{adv}}(c_h^{n-1},\vec{u}_h^{n-1},\mu_h^n) + b_\mathcal{I}(c_h^{n-1},\mu_h^n,\vec{v}_h^n). \nonumber 
\end{align}
Next, we choose $\vec{\theta}_h = \vec{u}_h^n$ in \eqref{eq:fully_dis6}
\begin{equation}\label{eq:stab_1}
\frac{1}{2}\norm{\vec{u}_h^n}{}^2 - \frac{1}{2}\norm{\vec{v}_h^{n}}{}^2 + \frac{1}{2}\norm{\vec{u}_h^n-\vec{v}_h^{n}}{}^2
= \tau b_{\mathcal{P}}(\vec{u}_h^n, \phi^n_h).
\end{equation}
Let $\vec{\theta}_h = \vec{u}_h^n - \vec{v}_h^n$ in \eqref{eq:fully_dis6}, then
\begin{equation}
\norm{\vec{u}_h^n-\vec{v}_h^{n}}{}^2
=  \tau b_{\mathcal{P}}(\vec{u}_h^n - \vec{v}_h^n, \phi^n_h).
\end{equation}
Therefore \eqref{eq:stab_1} becomes
\begin{equation}\label{eq:stab_1a}
\frac{1}{2}\norm{\vec{u}_h^n}{}^2 - \frac{1}{2}\norm{\vec{v}_h^{n}}{}^2 
- \frac{\tau}{2} b_\mathcal{P}(\vec{v}_h^n,\phi_h^n) 
= \frac{\tau}{2}  b_{\mathcal{P}}(\vec{u}_h^n, \phi^n_h).
\end{equation}
We can also show (see Lemma~5.1 in \cite{inspaper1})
\begin{align}    
b_{\mathcal{P}}(\bm{u}_h^n, q_h) = - \tau \sum_{e\in \Gamma_h} \frac{\tilde{\sigma}}{h} \int_{e} [\phi^n_h][q_h] 
+ \tau (\bm{G}_h([\phi^n_h]), \bm{G}_h([q_h])), \quad \forall q_h \in M_h^{k-1}. \label{eq:preposition_1}
\end{align}
Using \eqref{eq:fully_dis4} (with $\varphi_h = \phi_h^n$) and choosing $q_h = \phi_h^n$ in \eqref{eq:preposition_1}, we rewrite \eqref{eq:stab_1a} as
\begin{equation}\label{eq:stab_2}
\frac{1}{2}\norm{\vec{u}_h^n}{}^2 - \frac{1}{2}\norm{\vec{v}_h^{n}}{}^2 + \frac{\tau^2}{2} a_{\mathrm{diff}}(\phi_h^n,\phi_h^n) + \frac{\tau^2}{2}\sum_{e\in\Gammah}\frac{\tilde{\sigma}}{h}\norm{\jump{\phi_h^n}}{L^2(e)}^2 = \frac{\tau^2}{2}\norm{\vec{G}_h(\jump{\phi_h^n})}{}^2.
\end{equation}
With \eqref{eq:lift_prop_g}, we have
\begin{equation}\label{eq:stab_4}
\frac{1}{2\tau}\norm{\vec{u}_h^n}{}^2 - \frac{1}{2\tau}\norm{\vec{v}_h^{n}}{}^2 + \frac{\tau}{2} a_{\mathrm{diff}}(\phi_h^n,\phi_h^n) + \tau \frac{\tilde{\sigma}-\tilde{M}_{k}^2}{2h} \sum_{e\in\Gammah} \norm{\jump{\phi_h^n}}{L^2(e)}^2 \leq 0.
\end{equation}
Adding \eqref{eq:CHNS:Dis_energy_5} and \eqref{eq:stab_4} and choosing $\tilde{\sigma}>\tilde{M}_k^2$  yields
\begin{align}\label{eq:CHNS:Dis_energy_6}
&\,\frac{1}{2\tau}\norm{\vec{u}_h^n}{}^2 - \frac{1}{2\tau}\norm{\vec{u}_h^{n-1}}{}^2 + \frac{1}{2\tau}\norm{\vec{v}_h^n-\vec{u}_h^{n-1}}{}^2
+ \big(\delta_\tau\Phi(c_h^n),\,1\big)\\
+&\, \frac{\kappa}{2\tau}a_{\mathrm{diff}}(c_h^n,c_h^n) - \frac{\kappa}{2\tau}a_{\mathrm{diff}}(c_h^{n-1},c_h^{n-1}) + \frac{\kappa}{2\tau}a_{\mathrm{diff}}(c_h^n-c_h^{n-1},c_h^n-c_h^{n-1}) \nonumber\\
+&\, K_\alpha\vert \mu_h^n\vert_{\mathrm{DG}}^2 + K_\strain \mu_\mathrm{s} \norm{\vec{v}_h^n}{\mathrm{DG}}^2 
+ \frac{\tau}{2} a_{\mathrm{diff}}(\phi_h^n,\phi_h^n)  \nonumber \\
\leq&\, b_{\mathcal{P}}(\vec{v}_h^n, p_h^{n-1}) - a_{\mathrm{adv}}(c_h^{n-1},\vec{u}_h^{n-1},\mu_h^n) + b_\mathcal{I}(c_h^{n-1},\mu_h^n,\vec{v}_h^n).\nonumber 
\end{align}
To proceed with the term $b_{\mathcal{P}}(\vec{v}_h^n, p_h^{n-1})$, we rewrite it using the variables $S_h^n$ and $\zeta_h^n$.
\begin{align}\label{eq:stab:bpv}
b_{\mathcal{P}}(\vec{v}_h^n, p_h^{n-1}) = b_{\mathcal{P}}(\vec{v}_h^n, \zeta_h^{n-1}) - b_{\mathcal{P}}(\vec{v}_h^n,S_h^{n-1}).
\end{align}
Using \eqref{eq:defShzetah},  \eqref{eq:fully_dis5} and \eqref{eq:def_b_lift}, we note that
\begin{align}
\zeta_h^n - \zeta_h^{n-1} = p_h^n - p_h^{n-1} + \sigma_\chi \mu_\mathrm{s}(\divh{\vec{v}_h^n}-R_h\jump{\vec{v}_h^n}) = \phi_h^n.
\end{align}
Therefore, by \eqref{eq:fully_dis4}, we have
\begin{align}
b_{\mathcal{P}}(\vec{v}_h^n, \zeta_h^{n-1}) 
=& -\tau a_\mathrm{diff}(\phi_h^{n},\zeta_h^{n-1})
= -\tau a_\mathrm{diff}(\zeta_h^n - \zeta_h^{n-1},\zeta_h^{n-1}) \label{eq:stab:term1}\\
=& -\frac{\tau}{2} a_{\mathrm{diff}}(\zeta_h^n,\zeta_h^n) + \frac{\tau}{2} a_{\mathrm{diff}}(\zeta_h^{n-1},\zeta_h^{n-1}) + \frac{\tau}{2} a_{\mathrm{diff}}(\phi_h^n,\phi_h^n).\nonumber
\end{align}
For the second term of the right-hand side of \eqref{eq:stab:bpv}, by \eqref{eq:defShzetah} and \eqref{eq:def_b_lift}, we have
\begin{align}
b_{\mathcal{P}}(\vec{v}_h^n, S_h^{n-1})
&= (\div{\vec{v}_h^n}-R_h\jump{\vec{v}_h^n}, S_h^{n-1})
= \frac{1}{\sigma_\chi \mu_\mathrm{s}}(S_h^{n}-S_h^{n-1}, S_h^{n-1}) \label{eq:stab:term2}\\
&= \frac{1}{2\sigma_\chi \mu_\mathrm{s}}(\norm{S_h^n}{}^2 - \norm{S_h^{n-1}}{}^2 - \norm{S_h^n-S_h^{n-1}}{}^2). \nonumber
\end{align}
Substitute \eqref{eq:stab:term1} and \eqref{eq:stab:term2} into \eqref{eq:stab:bpv} to obtain
\begin{multline}
b_{\mathcal{P}}(\vec{v}_h^n, p_h^{n-1}) = -\frac{\tau}{2}a_{\mathrm{diff}}(\zeta_h^n,\zeta_h^n) + \frac{\tau}{2}a_{\mathrm{diff}}(\zeta_h^{n-1},\zeta_h^{n-1}) + \frac{\tau}{2}a_{\mathrm{diff}}(\phi_h^n,\phi_h^n) 
\\- \frac{1}{2\sigma_\chi\mu_\mathrm{s}}(\norm{S_h^n}{}^2 - \norm{S_h^{n-1}}{}^2 - \norm{S_h^n-S_h^{n-1}}{}^2). \label{eq:expression_b_pn}
\end{multline}
In addition, if we choose parameters $\sigma_\chi$ and $\sigma$ such that $\sigma_\chi < K_\strain/(2d)$ and $\sigma > M_{k-1}^2/d$, we have
\begin{equation}
\frac{1}{2\sigma_\chi\mu_\mathrm{s}}\norm{S_h^n-S_h^{n-1}}{}^2 \leq \frac{K_\strain\mu_\mathrm{s}}{2}\norm{\vec{v}_h^n}{\mathrm{DG}}^2.
\label{eq:bounddiffSh}
\end{equation}
Thus, with \eqref{eq:CHNS:DG_interface}, \eqref{eq:expression_b_pn}, and \eqref{eq:bounddiffSh}, the \eqref{eq:CHNS:Dis_energy_6} becomes 
\begin{align}
&\,\frac{1}{2\tau}\norm{\vec{u}_h^n}{}^2 - \frac{1}{2\tau}\norm{\vec{u}_h^{n-1}}{}^2 + \frac{1}{2\tau}\norm{\vec{v}_h^n-\vec{u}_h^{n-1}}{}^2
+ \big(\delta_\tau\Phi(c_h^n),\,1\big)\label{eq:CHNS:Dis_energy_7}\\
+&\, \frac{\kappa}{2\tau}a_{\mathrm{diff}}(c_h^n,c_h^n) - \frac{\kappa}{2\tau}a_{\mathrm{diff}}(c_h^{n-1},c_h^{n-1}) + \frac{\kappa}{2\tau}a_{\mathrm{diff}}(c_h^n-c_h^{n-1},c_h^n-c_h^{n-1}) \nonumber\\
+&\, K_\alpha\vert \mu_h^n\vert_{\mathrm{DG}}^2 + \frac{K_\strain \mu_\mathrm{s}}{2} \norm{\vec{v}_h^n}{\mathrm{DG}}^2
+ \frac{\tau}{2}a_{\mathrm{diff}}(\zeta_h^n,\zeta_h^n) - \frac{\tau}{2}a_{\mathrm{diff}}(\zeta_h^{n-1},\zeta_h^{n-1}) \nonumber\\
+&\, \frac{1}{2\sigma_\chi\mu_\mathrm{s}}\norm{S_h^n}{}^2 - \frac{1}{2\sigma_\chi\mu_\mathrm{s}}\norm{S_h^{n-1}}{}^2
\leq  a_{\mathrm{adv}}(c_h^{n-1},\vec{v}_h^n-\vec{u}_h^{n-1},\mu_h^n). \nonumber
\end{align}
Using the definition of $a_{\mathrm{adv}}$, Holder's inequality, \eqref{eq:trace_estimate_discrete}, and the
fact that $\tilde{\sigma}>\tilde{M}_k^2$ and $\tilde{M}_k = \sqrt{2} C_\mathrm{tr} N_\mathrm{face}^{1/2}$, we have
\begin{equation}
 a_{\mathrm{adv}}(c_h^{n-1},\vec{v}_h^n-\vec{u}_h^{n-1},\mu_h^n) 
\leq \norm{c_h^{n-1}}{L^\infty(\Omega)}\norm{\vec{v}_h^{n}-\vec{u}_h^{n-1}}{L^2{(\Omega)}}\vert \mu_h^n\vert_{\mathrm{DG}}. 
\end{equation}
Using \eqref{eq:sobolev_ineq}, Young's inequality, \eqref{thm:CHNS:discrete_mass_conservation}, and \eqref{eq:CHNS:dis_energy_dissipation_ass}, we obtain 
\begin{align}
& a_{\mathrm{adv}}(c_h^{n-1},\vec{v}_h^n-\vec{u}_h^{n-1},\mu_h^n)
\leq \frac{1}{4\tau}\norm{\vec{v}_h^{n}-\vec{u}_h^{n-1}}{}^2 
+  \tau \norm{c_h^{n-1}}{L^\infty(\Omega)}^2\vert \mu_h^n\vert_{\mathrm{DG}}^2 \label{eq:sec:stab:Dis_energy_adv}\\
\leq&\, \frac{1}{4\tau}\norm{\vec{v}_h^{n}-\vec{u}_h^{n-1}}{}^2 
+  \tau (2\tilde{C}_P^2 \Theta_{h,d}^2 \vert c_h^{n-1}\vert_{\mathrm{DG}}^2 + 2\overline{c_0}^2 ) \vert \mu_h^n\vert_{\mathrm{DG}}^2 \nonumber\\ 
\leq &  \frac{1}{4\tau}\norm{\vec{v}_h^{n}-\vec{u}_h^{n-1}}{}^2 
+  \tau \left(2\tilde{C}_P^2 \Theta_{h,d}^2\max \Big(\frac{2}{K_\alpha \kappa}, 2\Big)F_h(c_h^0,\vec{u}_h^0) + 2\overline{c_0}^2 \right) \vert \mu_h^n\vert_{\mathrm{DG}}^2. \nonumber
\end{align}
%
Substitute \eqref{eq:sec:stab:Dis_energy_adv} into the right-hand side of \eqref{eq:CHNS:Dis_energy_7}. The condition \eqref{eq:CHNS:dis_energy_dissipation_CFL} implies
\begin{align}
\tau \left(2\tilde{C}_P^2 \Theta_{h,d}^2\max \Big(\frac{2}{K_\alpha \kappa}, 2\Big)F_h(c_h^0,\vec{u}_h^0) + 2\overline{c_0}^2 \right) \leq \frac{K_\alpha}{2}.
\end{align} 
Then, multiply by $\tau$. The result follows.  
\end{proof}
\begin{remark}
In 3D, as $h$ tends to $0$, the CFL constraint \eqref{eq:CHNS:dis_energy_dissipation_CFL} simply reads
$\tau \leq C h$.   In 2D, the CFL constraint is milder:   $\tau \leq C(1+\vert \ln h \vert)^{-1}$. 
\end{remark}
\begin{theorem}\label{thm:CHNS:stability_bound1}
Assume that $\tau$ and $h$ satisfy the CFL condition \eqref{eq:CHNS:dis_energy_dissipation_CFL}.
If $\sigma > M_{k-1}^2/d,\, \tilde{\sigma} > \tilde{M}_k^2$ and $\sigma_\chi < K_\strain/(2d)$, 
for any $1 \leq \ell \leq \Nst$, we have
\begin{multline}\label{eq:CHNS:Stability}
\frac{1}{2}\norm{\vec{u}_h^\ell}{}^2 + \big(\Phi(c_h^\ell),1\big) + \frac{K_\alpha\kappa}{2}\vert c_h^\ell\vert_{\mathrm{DG}}^2  + \frac{\tau}{2\sigma_\chi\mu_\mathrm{s}}\norm{S_h^\ell}{}^2 + \frac{\tau^2}{2}a_{\mathrm{diff}}(\zeta_h^\ell,\zeta_h^\ell) \\
+ \frac{K_\alpha\tau}{2}\sum_{n=1}^\ell \vert \mu_h^n\vert_{\mathrm{DG}}^2 + \frac{K_\strain \mu_\mathrm{s}\tau}{2}\sum_{n=1}^\ell \norm{\vec{v}_h^n}{\mathrm{DG}}^2 + \frac{1}{4}\sum_{n=1}^\ell\norm{\vec{v}_h^n-\vec{u}_h^{n-1}}{}^2
\leq F_h(c_h^0, \vec{u}_h^0),
\end{multline}
\end{theorem}
\begin{proof}
To prove \eqref{eq:CHNS:Stability} we use an induction argument on $\ell$. The positivity of
the chemical energy density and \eqref{eq:coercivity_adiff} imply
\begin{equation}
F_h(c_h^0,\vec{u}_h^0) = \frac{1}{2}(\vec{u}_h^0,\vec{u}_h^0) + \big(\Phi(c_h^0),1\big) + \frac{\kappa}{2}a_{\mathrm{diff}}(c_h^0,c_h^0) \geq \frac{K_\alpha\kappa}{2}\vert c_h^0\vert_{\mathrm{DG}}^2.
\end{equation}
Therefore, assumption \eqref{eq:CHNS:dis_energy_dissipation_ass} holds for $n=1$. We apply \Cref{thm:CHNS:dis_energy_dissipation}
to obtain \eqref{eq:CHNS:dis_energy_dissipation_ineq} for $n=1$, which implies \eqref{eq:CHNS:Stability}  for $\ell = 1$ since $S_h^0=\zeta_h^0 = 0$.
\par
Fix $j \geq 1$ and assume that \eqref{eq:CHNS:Stability} holds for all $1\leq \ell\leq j$.  This means that
\eqref{eq:CHNS:dis_energy_dissipation_ass} is valid for all $1\leq n \leq j+1$. With \Cref{thm:CHNS:dis_energy_dissipation}, we have
that \eqref{eq:CHNS:dis_energy_dissipation_ineq} is valid for any $1\leq n \leq j+1$. Summing \eqref{eq:CHNS:dis_energy_dissipation_ineq} over $n$ yields
\eqref{eq:CHNS:Stability}  for $\ell =  j+1$. 
\end{proof}
%
\begin{remark}
Using a triangle inequality, Poincare's inequality \eqref{eq:PoincareMh}, and mass conservation \eqref{thm:CHNS:discrete_mass_conservation}, stability bound \eqref{eq:CHNS:Stability} implies that for any $p\leq 6$
\begin{equation}
\Vert c_h^\ell \Vert_{L^p(\Omega)} \leq C, \quad 1\leq \ell \leq N_T.
\label{eq:L2boundchLp}
\end{equation}
\end{remark}

%% file: Content/error_analysis_induction_const.tex
\section{Error analysis}\label{sec:error_analysis}
In the remainder of the paper, we assume that $\Omega$ is convex. The goal of this section is to show the following convergence result. 
\begin{theorem}\label{thm:conv_estimate_1}
Assume that $\sigma \geq \tilde{M}_{k-1}^2/d$, $\tilde{\sigma} \geq 4\tilde{M}_{k}^2$, and $\sigma_\chi \leq K_\strain/(2d)$.
Fix  $0<\delta< 1$.   There exist constants  $\gamma, C_\mathrm{err}, h_0, \tau_0$  independent of $h$ and $\tau$, 
such that if $h\leq h_0, \tau \leq \tau_0$ and    
\begin{equation} 
\tau \leq \gamma h^{ 1+ \delta }, \label{eq:cfl_cond}
\end{equation}
then the following error estimate holds. For $1\leq m \leq \Nst$,  
\begin{align}\label{eq:error_estimate_theorem}
K_\alpha \tau \sum_{n=1}^{m} \vert\mu_h^n - \mu^n \vert_{\mathrm{DG}}^2 
&+ K_\strain \mu_\mathrm{s}\tau \sum_{n=1}^m \|\bm{v}_h^n -\vec{u}^n \|_{\DG}^2 
+ \tau^2 \sum_{n=1}^m \vert\phi_h^n\vert_{\mathrm{DG}}^2 \\
&+ \kappa K_\alpha \vert c_h^m - c^m\vert_{\mathrm{DG}}^2 + \|\bm{u}_h^m -  \bm{u}^m \|^2 
\leq C_\mathrm{err}(\tau + h^{2k}). \nonumber
\end{align}
In addition, there exists a constant $\widetilde{C}_\mathrm{err}$ independent of $h$, $\tau$, such that the following improved estimate holds. For $1\leq m \leq \Nst$,
\begin{align}\label{eq:improved_error_estimate_theorem}
\|c_h^m - c^m\|^2 &+ \tau  \sum_{n=1}^m \|\mu_h^n - \mu^n\|^2 \\  
&+ \mu_\mathrm{s} \tau  \sum_{n=1}^m \left(\|\bm{v}_h^n - \bm{u}^n \|^2 +\|\bm{u}_h^n - \bm{u}^n \|^2 \right) \leq \widetilde{C}_\mathrm{err} (\tau^2 + \tau h^2 + h^{2k+2}).  \nonumber
\end{align}
The above estimates hold under the following regularity assumptions:
 $\grad{c^0} \cdot\bm{n} = 0$ on $\partial\Omega$ and 
\begin{align*}
    c, \mu &\in L^{\infty}(0,T;H^{k+1}(\Omega)), & \partial_t c &\in L^2(0,T;H^{k+1}(\Omega)), & \partial_{tt} c &\in L^2(0,T;L^2(\Omega)),  \\ 
    \bm{u} &\in L^\infty(0,T;H^{k+1}(\Omega)^d), & \partial_t \bm{u} &\in L^2(0,T;H^{k+1}(\Omega)^d), & p &\in L^{\infty}(0,T;H^{k}(\Omega)). 
\end{align*}
\end{theorem}
\begin{remark}
Hereinafter, we use an induction argument to prove \eqref{eq:error_estimate_theorem}. In each induction iteration,  the constants 
$\gamma, C_\mathrm{err}, h_0, \tau_0$  are unchanged. Therefore, the algorithm \eqref{eq:fully_dis1}-\eqref{eq:fully_dis6} is suited in simulations with any prescribed end time $T$.  
\end{remark} 
\begin{remark} The bound \eqref{eq:improved_error_estimate_theorem} is optimal for $k=1$ since $\tau h^2 \leq (\tau^2 + h^4)/2$. For $k \geq 2$, the reverse CFL condition ``$h^2 \leq \tau$"  is required for optimality. 
\end{remark}
\textbf{Proof Outline:} Since the proof of this result requires several intermediate Lemmas, we provide a brief outline here.
The proof of  \eqref{eq:error_estimate_theorem} is in Section \ref{sec:conv_estimate_1} and it is a consequence of the following
bound, valid for $1\leq m \leq \Nst$: 
\begin{equation}
     \tau^2 \sum_{n = 0}^{m-1} \vert \phi_h^n \vert_{\mathrm{DG}}^{2} + \tau \sum_{n = 0}^{m-1} \|\bm{v}_h^n - \Pi_h \bm{u}^n \|^2_{\DG} \leq \tau^{\frac{1}{1+\delta}}+ h^{\frac{2+\delta}{1+\delta}}. \label{eq:induction_hyp_linf}
\end{equation}
where $\Pi_h$ is a suitable interpolant, see \eqref{eq:def_elliptic_projection}.
We will show \eqref{eq:induction_hyp_linf} by induction on $m$.  For the starting value $m=1$, the bound \eqref{eq:induction_hyp_linf}
is easy to obtain since $\phi_h^0 = 0$ and $\bm{v}_h^0 = \bm{u}_h^0$ which is the $L^2$ projection of $\bm{u}^0$.
\par
Next, we fix $m\geq 1$ and assume that the induction hypothesis \eqref{eq:induction_hyp_linf} holds true. To show \eqref{eq:error_estimate_theorem}, our induction argument contains the following steps:
\begin{itemize}
\item[(i)] The induction hypothesis \eqref{eq:induction_hyp_linf} implies an $L^\infty(\Omega)$ bound for the discrete order parameter $c_h^m$ (see Lemma~\ref{lemma:L_infty_bound}).
\item[(ii)] We then obtain a bound on the dG norm of $c_h^m-\mathcal{P}_h c^m$, where $\mathcal{P}_h$ is a suitable projection (see Section~\ref{subsec:interpolation}). This bound is proved in Lemma~\ref{lemma:bound_xic} and uses Lemma~\ref{lemma:L_infty_bound}.
\item[(iii)] A  bound on the dG norm of $\mu_h^m -  \mathcal{P}_h \mu^m$ easily follows (see Lemma \ref{lemma:bound_ximu}).
\item[(iv)] We then show that \eqref{eq:error_estimate_theorem} holds true (see Lemma~\ref{lemma:err_eq_pressure_corretion} and
Section \ref{sec:conv_estimate_1} for more details).
\item[(v)] We complete the induction proof by then showing the induction hypothesis for $m+1$ (see Lemma~\ref{lem:prove_induction_hyp}).
\end{itemize}
\par
Finally, to show \eqref{eq:improved_error_estimate_theorem}, we use several duality arguments and \eqref{eq:error_estimate_theorem}. The main proof of \eqref{eq:improved_error_estimate_theorem} is provided in subsection \ref{sec:proof_improved_err}.
\subsection{Approximation operators}\label{subsec:interpolation}
As a prelude to showing the above steps, we introduce the several approximations of the exact solution that are  employed in the error analysis. Let $\mathcal{P}_h: H^2(\mathcal{T}_h) \rightarrow M_h^{k}$ be the elliptic projection operator. For $\phi \in H^2(\mathcal{T}_h)$, define $\mathcal{P}_h \phi$ as the solution of the following problem. 
\begin{align}\label{eq:def_elliptic_projection}
\adif(\mathcal{P}_h \phi - \phi, \chi_h )  = 0, \quad \forall \chi_h \in M_h^{k},  \quad 
\int_\Omega (\mathcal{P}_h \phi - \phi) = 0. 
\end{align} 
We have the following error bounds which can be derived from the dG error analysis for elliptic problems on convex domains \cite{riviere2008}. 
\begin{align}
\|\phi - \mathcal{P}_h \phi \| + h \vert \phi - \mathcal{P}_h \phi \vert_{\mathrm{DG}} \leq C h^{k+1} | \phi |_{H^{k+1}(\mathcal{T}_h)}, \quad \forall \phi \in H^{k+1}(\mathcal{T}_h). \label{eq:elliptic_projection_error} 
\end{align} 
For functions in $H^1_0(\mathcal{T}_h)^d$, we will make use of the operator $\Pi_h: H^1_0(\mathcal{T}_h)^d \rightarrow \mathbf{X}_h^{k}$. For $\bm{u}(t) \in  H^1_0(\mathcal{T}_h)^d$, this operator satisfies the following: 
\begin{equation}
b_{\mathcal{P}}(\Pi_h \bm{u}(t) - \bm{u}(t), q_h) = 0, \quad \forall q_h \in M_h^{k-1}.  \label{eq:def_Pih}
\end{equation}
The proof for the existence of this operator and the following approximation bounds can be found in \cite{chaabane2017convergence}. 
\begin{lemma}\label{lemma:Pi_projection_error}
    For $E \in \mathcal{T}_h$, $1\leq p \leq \infty$, $1 \leq s \leq k+1$, $0 \leq n \leq \Nst$, and $\bm{u}(t) \in (W^{s,p}(E) \cap H^1_0(\Omega))^d$,  
    \begin{align}
    \| \Pi_h \bm{u}(t) - \bm{u}(t) \|_{L^p(E)} +   h_E \| \grad (\Pi_h \bm{u}(t) - \bm{u}(t)) \|_{L^p(E) } &\leq Ch_E^s \vert \bm{u}(t) \vert_{W^{s,p}(\Delta_E)}.  \label{eq:approximation_prop_local} 
    \end{align}
    where $\Delta_E$ is a macro element that contains $E$. 
\end{lemma}
For $ 0\leq t\leq T$, if $\bm{u}(t) \in (W^{s,p}(\Omega) \cap H_0^1(\Omega))^d$ for $1\leq s \leq k+1$, then the above bound yields the global estimates: 
    \begin{align}
        \| \Pi_h \bm{u}(t) - \bm{u}(t) \|_{L^p(\Omega)} &\leq Ch^s \vert \bm{u}(t) \vert_{W^{s,p}(\Omega)}, \label{eq:approximation_prop_1} \\ 
         \| \Pi_h \bm{u}(t) - \bm{u}(t) \|_{\mathrm{DG}}& \leq C h^{s-1}  \vert \bm{u}(t) \vert_{H^{s}(\Omega)} \label{eq:approximation_prop_2}.
        \end{align}
        Define the  $L^2$ projection  $\pi_h: L^2(\Omega) \rightarrow M^{k-1}_h$ as follows: For $ 0\leq n \leq \Nst$,  a given function $p(t) \in L^2(\Omega)$,  and any $E \in \mathcal{T}_h$,  
        \begin{equation}
            \int_{E} (\pi_h p(t) - p(t))q_h = 0, \quad  \forall q_h \in \IP_{k-1}(E) .
        \end{equation} 
        The following error bound for the $L^2$ projection hold. For $t \in [0,T]$ and $p(t) \in H^s(\Omega)$,  
        \begin{align} \label{eq:l2_proj_approximation}
            \|\pi_h p(t) - p(t) \|+ h\|\gradh(\pi_h p(t) - p(t))\| \leq Ch^{\min(k,s)}\vert p(t) \vert_{H^s(\Omega)}.  
        \end{align}
    We will also make use of a linear operator $\mathcal{J}: M_{h0}^{k} \rightarrow M_{h0}^{k}$. Given $\chi_h \in M_{h0}^{k}$, define $\mathcal{J}(\chi_h) \in M_{h0}^{k}$ as the solution of the following elliptic problem:
\begin{equation}
    \adif(\mathcal{J}(\chi_h), \varphi_h ) = (\chi_h, \varphi_h), \quad \forall \varphi_h \in M_{h0}^{k}.  \label{eq:def_of_J}
\end{equation} 
The following property for this operator is shown in \cite{LiuRiviere2018numericalCHNS}.
\begin{lemma}\label{lemma:boundedness_J}
For a function $\lambda \in H^1(\mathcal{T}_h)$, there exists a constant $C_\mathrm{J}$ independent of $h$ such that 
\begin{equation}
| (\lambda, \chi_h) | \leq C_\mathrm{J} \vert\mathcal{J}(\chi_h)\vert_{\mathrm{DG}} \vert \lambda \vert_{\mathrm{DG}}, \quad \forall \chi_h \in M_{h0}^{k}. 
\end{equation}
\end{lemma}
We introduce a discrete Laplacian operator $\laplace_h: M_h^{k} \rightarrow M_{h0}^{k}$ via the following variational problem: Given $z_h \in M_h^{k}$, find $\laplace_h{z_h}\in M_{h0}^{k}$ such that
\begin{equation}
(\laplace_h{z_h},\chi_h) = -a_{\mathrm{diff}}(z_h,\chi_h), \quad \forall \chi_h \in M_h^{k}. \label{eq:def_discrete_laplace}
\end{equation}
We now show discrete broken versions to the following Agmon's and Gagliardo--Nirenberg inequalities in $d=2$ and $d = 3$. For $z \in H^2(\Omega)$, 
\begin{align}
\|z\|_{L^{\infty}(\Omega)} &\leq C \|z\|^{1-d/4}\|z\|_{H^2(\Omega)}^{d/4}, \label{eq:continuous_Agmon} \\
\|\grad z\|_{L^{3}(\Omega)} &\leq C \|z\|^{1/2-d/12} \|z\|^{1/2+d/12}_{H^2(\Omega)}. \label{eq:continuous_GN}
\end{align}
For a proof for the above inequalities, we refer to  Theorem 5.8 in \cite{adams1977cone} for \eqref{eq:continuous_Agmon} and to Lecture 2 in \cite{ASNSP_1959_3_13_2_115_0} for \eqref{eq:continuous_GN}.
The proof of the following Lemma follows the arguments in \cite{kay2009discontinuous} presented for $d=2$. For completeness, we provide a proof here for $d \in \{2,3\}.$
\begin{lemma}\label{lemma:broken_agmon_gagliardo} 
There exists a constant $C$ independent of $h$ such that 
\begin{align}
\norm{z_h - \frac{1}{|\Omega|}\int_{\Omega} z_h }{L^{\infty}(\Omega)} &\leq C \|z_h\|^{1-d/4}\|\laplace_h z_h \|^{d/4}, \quad \forall z_h \in M_h^k, \label{eq:discrete_Agmon}\\
\|\grad_h z_h\|_{L^{3}(\Omega)} &\leq C \|z_h\|^{1/2-d/12}\|\laplace_h z_h\|^{1/2+d/12}, \quad \forall z_h \in M_h^k. \label{eq:discrete_GN}
\end{align}
\end{lemma} 
\begin{proof}
We first define a Green's operator $\mathcal{G}: M_{h0}^{k} \rightarrow H^1(\Omega) \cap L^2_0(\Omega)$ as the solution to  $-\laplace (\mathcal{G} (\varphi_h)) = \varphi_h$ in $\Omega$ with homogeneous Neumann boundary conditions. By elliptic regularity, we have 
\begin{equation}
    \|\mathcal{G} (\varphi_h)\|_{H^2(\Omega)} \leq C \|\varphi_h\|, \quad \varphi_h \in M_{h0}^{k}.  \label{eq:elliptic_reg_G}
\end{equation} 
Since $\adif$ is symmetric and $\Omega$ is convex, we have \cite{riviere2008} 
\begin{align}
\|\mathcal{J} (\varphi_h) - \mathcal{G} (\varphi_h) \| \leq C h^2 \|\mathcal{G}(\varphi_h)\|_{H^2(\Omega)}, \quad \varphi_h \in M_{h0}^{k}. \label{eq:symm_ip_approximation}
 \end{align}
To simplify notation, for $z_h \in M_h^{k}$, let  $ \zeta_h = \laplace_h z_h $. From the definitions, it can be readily deduced that (see (2.18) in \cite{kay2009discontinuous})
\begin{equation}
  -  \mathcal{J} (\zeta_h) = z_h -\frac{1}{|\Omega|}\int_{\Omega} z_h.  \label{eq:identity_Gz}
\end{equation}
Let $\mathcal{I}_h$ denote the Scott--Zhang interpolation operator. 
Using the approximation properties of the Scott--Zhang operator, \eqref{eq:inverse_estmate_lp} (with $p=+\infty$ and $q = 2$), \eqref{eq:continuous_Agmon} and \eqref{eq:symm_ip_approximation}, 
\begin{align}
   \|\mathcal{J}(\zeta_h)\|_{L^{\infty}(\Omega)}  \leq &
\Vert \mathcal{J}(\zeta_h) - \mathcal{I}_h(\mathcal{G}(\zeta_h))\Vert_{L^{\infty}(\Omega)}
+ \Vert \mathcal{I}_h(\mathcal{G}(\zeta_h)) - \mathcal{G}(\zeta_h)\Vert_{L^{\infty}(\Omega)} \nonumber\\
& + \|\mathcal{G} (\zeta_h)\|_{L^{\infty}(\Omega)}\nonumber\\
 \leq & C h^{2-d/2} \|\mathcal{G}(\zeta_h)\|_{H^{2}(\Omega)} + C \|\mathcal{G}(\zeta_h)\|^{1-d/4}\|\mathcal{G} (\zeta_h)\|_{H^2(\Omega)}^{d/4}.  \label{eq:Agmon_1}
\end{align}
Note that with triangle inequality, \eqref{eq:symm_ip_approximation} and \eqref{eq:identity_Gz},
\begin{equation}
    \|\mathcal{G}(\zeta_h)\| \leq  2 \|z_h\| + Ch^{2} \|\mathcal{G}(\zeta_h)\|_{H^2(\Omega)}. \label{eq:bound_l2_G}
\end{equation}
Using the above bound, \eqref{eq:elliptic_reg_G}, and the definition of $\zeta_h$ in \eqref{eq:Agmon_1} yields 
\begin{align}
\|\mathcal{J}(\zeta_h)\|_{L^\infty(\Omega)} &\leq C \|z_h\|^{1-d/4} \|\laplace_h z_h\|^{d/4} + Ch^{2-d/2}\|\laplace_h z_h\|\label{eq:Agmon_2} \\ & =C \|z_h\|^{1-d/4} \|\laplace_h z_h\|^{d/4} + C(h^{2}\|\laplace_h z_h\|)^{1-d/4}\|\laplace_h z_h\|^{d/4}. \nonumber
\end{align}
Taking $\chi_h = \laplace_h z_h$ in  \eqref{eq:def_discrete_laplace} and  using \eqref{eq:adiffcont} and \eqref{eq:inverse_estimate_dg} yields
\begin{equation}
\|\laplace_h z_h\|^2 = \adif(z_h, \laplace_h z_h) \leq C \vert z_h\vert_{\mathrm{DG}} \vert \laplace_h z_h \vert_{\mathrm{DG}} \leq Ch^{-2} \|z_h\|\|\laplace_h z_h\|. \label{eq:bound_laplace_zh}
\end{equation}
Using  \eqref{eq:bound_laplace_zh} and \eqref{eq:identity_Gz} in \eqref{eq:Agmon_2} proves \eqref{eq:discrete_Agmon}. To show \eqref{eq:discrete_GN}, we proceed in a similar way. With Holder's inequality, \eqref{eq:inverse_estmate_lp}, \eqref{eq:inverse_estimate_dg}, approximation properties, and \eqref{eq:symm_ip_approximation}:
\begin{align*}
 \|\gradh \mathcal{J}(\zeta_h)\|_{L^3(\Omega)}  \leq & \|\gradh \mathcal{G}(\zeta_h)\|_{L^3(\Omega)} + \|\gradh (\mathcal{G}(\zeta_h) - \mathcal{I}_h(\mathcal{G}(\zeta_h))\|_{L^3(\Omega)} \nonumber\\
& + \| \grad_h (\mathcal{I}_h(\mathcal{G}(\zeta_h)) - \mathcal{J}(\zeta_h))\|_{L^3(\Omega)} \\ 
 \leq & C  \|\mathcal{G}(\zeta_h)\|^{1/2-d/12}\|\mathcal{G}(\zeta_h)\|^{1/2+d/12}_{H^2(\Omega)} + Ch^{1-d/6}\|\mathcal{G}(\zeta_h)\|_{H^2(\Omega)}  \nonumber\\
&+ Ch^{-d/6-1}\|\mathcal{I}_h(\mathcal{G}(\zeta_h)) - \mathcal{J}(\zeta_h)\|  \\ 
  \leq  &C  \|\mathcal{G}(\zeta_h)\|^{1/2-d/12}\|\mathcal{G}(\zeta_h)\|^{1/2+d/12}_{H^2(\Omega)} + Ch^{1-d/6}\|\mathcal{G}(\zeta_h)\|_{H^2(\Omega)} . 
\end{align*}
Using \eqref{eq:bound_l2_G}, \eqref{eq:elliptic_reg_G},  and \eqref{eq:bound_laplace_zh} yields 
\begin{multline}
    \|\gradh \mathcal{J}(\zeta_h)\|_{L^3(\Omega)} \leq C  \|z_h\|^{1/2-d/12}\|\laplace_h z_h\|^{1/2+d/12} \\ + C(h^2\|\laplace_h z_h \|)^{1/2-d/12}\|\laplace_h z_h\|^{1/2+d/12} \leq C  \|z_h\|^{1/2-d/12}\|\laplace_h z_h\|^{1/2+d/12}.
\end{multline}
The result is concluded by recalling \eqref{eq:identity_Gz}.
\end{proof}
We end this section by stating the consistency properties of our scheme. For readability, for any function $\bm{g} \in L^1(0,T;H^2(\mathcal{T}_h)^d)$ 
we denote $\bm{g}^n = \bm{g}(t^n)$ and use a similar notation for scalar functions. The weak solution of model problem \eqref{eq:CHNS:model} satisfies the following. For any $1 \leq n\leq \Nst$, for any $\chi_h \in M_h^k$, $\varphi_h \in M_h^{k}$, and $\bm{\theta}_h \in \mathbf{X}_h^{k}$,
\begin{align}
((\partial_t c)^n, \chi_h)  + \adif(\mu^n,\chi_h) + \aadv(c^n,\bm{u}^n,\chi_h)  = 0,  \label{eq:consistency_1} \\
(\Phi_+\,\!'(c^n) + \Phi_-\,\!'(c^n), \varphi_h)  + \kappa \adif(c^n,\varphi_h) -(\mu^n , \varphi_h)  = 0,  \label{eq:consistency_2}\\ 
((\partial_t \bm{u})^n ,\bm{\theta}_h ) + a_\mathcal{C}(\bm{u}^n,\bm{u}^n,\bm{u}^n, \bm{\theta}_h)  + \mu_\mathrm{s} a_\strain(\bm{u}^n, \bm{\theta}_h ) \label{eq:consistency_3} \\  = b_{\mathcal{P}}(\bm{\theta}_h, p^n) + b_\mathcal{I}(c^n,\mu^n,\bm{\theta}_h). \nonumber 
\end{align}
\subsection{The $L^\infty$ bound (Step (i))} 
\begin{lemma}\label{lemma:L_infty_bound}
Fix $m$, with $1 \leq m \leq \Nst$, and assume that \eqref{eq:induction_hyp_linf} holds.   
In addition, assume that $\tau$ satisfies \eqref{eq:cfl_cond}.
Then, there exists a constant $C$ independent of $h$, $\tau$, and $m$,  but depending linearly on $T$, such that
\begin{align}
\max_{1\leq n \leq m} \left( \norm{\mu_h^n}{}^2  + \norm{c_h^n}{L^\infty(\Omega)}^2 + \norm{\gradh{c_h^n }}{L^3(\Omega)}^2 \right) + \kappa\tau \sum_{n=1}^{m} \norm{\delta_\tau{c_h^n}}{}^2 \leq  C.  \label{eq:stability_infty1} 
\end{align}
\end{lemma}
\begin{proof}
Let $1\leq n \leq m$. From \eqref{thm:CHNS:discrete_mass_conservation}, \eqref{eq:discrete_Agmon}, \eqref{eq:discrete_GN}
and Young's inequality (of the form $ab\leq (1/p) a^p + (1/q) b^q$ for $p>1, q>1$ and $1/p+1/q=1$), we have 
\begin{align*}
\norm{c_h^n - \overline{c_0}}{L^\infty(\Omega)} 
\leq C \norm{c_h^n}{}^{1-d/4} \norm{\laplace_h{c_h^n}}{}^{d/4} \leq C (\Vert c_h^n\Vert + \Vert \laplace_h{c_h^n}\Vert), \\
 \norm{\nabla_h c_h^n}{L^3(\Omega)} 
\leq  C\norm{c_h^n}{}^{1/2- d/12} \norm{\laplace_h{c_h^n}}{}^{1/2+d/12} \leq C (\Vert c_h^n\Vert + \Vert \laplace_h{c_h^n}\Vert). 
\end{align*}
With \eqref{eq:L2boundchLp}, the above bounds yield
\begin{align}
\norm{c_h^n}{L^\infty(\Omega)}^2 + \norm{\grad{c_h^n}}{L^3(\Omega)}^2 &
\leq C ( 1+ \norm{\laplace_h{c_h^n}}{}^2 ).  \label{eq:CHNS:equation_1a}
\end{align}
Choosing $\varphi_h = \laplace_h{c_h^n}$ in \eqref{eq:fully_dis2}, and using the definition of the discrete Laplacian operator \eqref{eq:def_discrete_laplace}, Cauchy--Schwarz's and Young's inequalities yield  
\begin{multline}\label{eq:CHNS:equation_2}
\kappa\norm{\laplace_h{c_h^n}}{}^2 = -\kappa a_{\mathrm{diff}}(c_h^n, \laplace_h{c_h^n}) =  (\Phi_+\,\!'(c_h^n) + \Phi_-\,\!'(c_h^{n-1}), \laplace_h{c_h^n}) - (\mu_h^n, \laplace_h{c_h^n})\\
\leq  \frac{\kappa}{2} \norm{\laplace_h{c_h^n}}{}^2 + C \norm{\Phi_+\,\!'(c_h^n) + \Phi_-\,\!'(c_h^{n-1})}{}^2 + C\norm{\mu_h^n}{}^2.
\end{multline}
Since $\Phi_+\,\!'(c) = c^3$ and $\Phi_-\,\!'(c) = -c$, by \eqref{eq:L2boundchLp}, we obtain
\begin{align}\label{eq:CHNS:equation_3}
\norm{\Phi_+\,\!'(c_h^n) + \Phi_-\,\!'(c_h^{n-1})}{}^2 \leq 2(\norm{c_h^n}{L^6(\Omega)}^6 + \norm{c_h^{n-1}}{}^2) \leq C.
\end{align}
Thus, from \eqref{eq:CHNS:equation_1a}, \eqref{eq:CHNS:equation_2}, and \eqref{eq:CHNS:equation_3}, it follows that for all $1\leq n \leq m $,
\begin{align}
\norm{c_h^n}{L^\infty(\Omega)}^2 + \norm{\grad{c_h^n}}{L^3(\Omega)}^2  &\leq C ( 1+ \norm{\mu_h^n}{}^2). \label{eq:stab_mu_L2_ctrl}
\end{align}
Therefore to obtain \eqref{eq:stability_infty1}, it suffices to bound $\max_{1\leq n\leq m} \Vert \mu_h^n\Vert^2
+ \tau\sum_{n=1}^m \Vert \delta_\tau c_h^n\Vert^2$, which is done below via Gronwall's lemma.
We first set by convention $c_h^{-1} = 0$. 
Then, we introduce $\mu_h^0\in M_h^k$ via the following variational problem
\begin{align}
(\mu_h^0, \varphi_h) = (\Phi_+\,\!'(c_h^0) + \Phi_-\,\!'(c_h^{-1}), \varphi_h) + \kappa a_{\mathrm{diff}}(c_h^0,\varphi_h), && \forall \varphi_h \in M_h^{k}. \label{eq:defmuh0}
\end{align}
Note, $\mu_h^0$ is well-defined by Riesz representation theorem. 
Choose $\varphi_h = \mu_h^0$ above, use \eqref{eq:ch0ellip}  and a bound similar to \eqref{eq:CHNS:equation_3}, to obtain
\[
\Vert \mu_h^0\Vert^2 \leq C \Vert \mu_h^0 \Vert \, (\Vert c_h^0 \Vert_{L^6(\Omega)}^6 + \Vert c_h^0\Vert^2)^{1/2} 
+ \kappa \vert a_{\mathrm{diff}}(c^0,\mu_h^0) \vert. 
\]
Under the assumption $c^0\in H^2(\Omega)$ and $\grad{c^0}\cdot\bm{n} = 0$ on $\partial\Omega$, we have
\[
a_{\mathrm{diff}}(c^0,\mu_h^0) = -(\laplace{c^0}, \mu_h^0) 
\leq \Vert c^0 \Vert_{H^2(\Omega)} \Vert \mu_h^0\Vert.
\]
Then by \eqref{eq:L2boundchLp}, we obtain
\begin{equation}\label{eq:L2_bound_init_mu_h}
\Vert \mu_h^0 \Vert \leq C.
\end{equation}
Now, we subtract \eqref{eq:fully_dis2}  at time $t^{n-1}$ from itself at time $t^n$ (for $n=1$ we use \eqref{eq:defmuh0}) and choose $\varphi_h = \mu_h^n/\tau$.  We have
\begin{align*}
(\delta_\tau{\mu_h^n}, \mu_h^n) =& \frac{1}{\tau} (\Phi_+\,\!'(c_h^n) - \Phi_+\,\!'(c_h^{n-1}), \mu_h^n) \\
&+ \frac{1}{\tau} (\Phi_-\,\!'(c_h^{n-1}) - \Phi_-\,\!'(c_h^{n-2}), \mu_h^n)  + \kappa a_{\mathrm{diff}}(\delta_\tau{c_h^n},\mu_h^n). 
\end{align*}
Choosing $\chi_h =  \kappa  \delta_\tau{c_h^n}$ in \eqref{eq:fully_dis1}, we have
\begin{align*}
\kappa(\delta_\tau{c_h^n}, \delta_\tau{c_h^n}) = -\kappa a_{\mathrm{diff}}(\mu_h^n, \delta_\tau{c_h^n}) - \kappa a_{\mathrm{adv}}(c_h^{n-1}, \vec{u}_h^{n-1}, \delta_\tau{c_h^n}).
\end{align*}
Adding the two equations above yields 
\begin{multline}\label{eq:CHNS:equation_4}
\kappa\norm{\delta_\tau{c_h^{n}}}{}^2 + \frac{1}{2\tau}\norm{\mu_h^{n}}{}^2 - \frac{1}{2\tau}\norm{\mu_h^{n-1}}{}^2 
\leq  \frac{1}{\tau}(\Phi_+\,\!'(c_h^n) - \Phi_+\,\!'(c_h^{n-1}), \mu_h^n) \\ + \frac{1}{\tau}(\Phi_-\,\!'(c_h^{n-1}) - \Phi_-\,\!'(c_h^{n-2}), \mu_h^n) - \kappa a_{\mathrm{adv}}(c_h^{n-1}, \vec{u}_h^{n-1}, \delta_\tau{c_h^n}).
\end{multline}
We separately bound the terms on the right-hand side of \eqref{eq:CHNS:equation_4}. Using  
\eqref{eq:L2boundchLp}, Holder's and
Poincar\'e's inequalities  \eqref{eq:PoincareMh},   we obtain 
\begin{align}
\frac{1}{\tau}(\Phi_+\,\!'(c_h^n) - \Phi_+\,\!'(c_h^{n-1}), \mu_h^n) 
 &\leq \norm{\delta_\tau{c_h^n}}{} \norm{(c_h^{n})^2 + c_h^{n}c_h^{n-1} + (c_h^{n-1})^2}{L^3(\Omega)} \norm{\mu_h^n}{L^6(\Omega)}   \nonumber\\ 
&\leq \frac{\kappa}{6} \norm{\delta_\tau{c_h^n}}{}^2 + C(\norm{c_h^n}{L^6(\Omega)}^4 + \norm{c_h^{n-1}}{L^6(\Omega)}^4)\norm{\mu_h^n}{L^6(\Omega)}^2\nonumber\\
&\leq \frac{\kappa}{6} \norm{\delta_\tau{c_h^n}}{}^2 + C\vert \mu_h^n\vert_{\mathrm{DG}}^2 + C \Vert \mu_h^n \Vert^2.  \label{eq:phi_plus_bound}
\end{align}
Since $\Phi_-\,\!'(c) = -c$, with Cauchy--Schwarz's and Young's inequalities, we obtain
\begin{align}\label{eq:bound_phi-}
\frac{1}{\tau}(\Phi_-\,\!'(c_h^{n-1}) - \Phi_-\,\!'(c_h^{n-2}), \mu_h^n) \leq \frac{\kappa}{6} \norm{\delta_\tau{c_h^{n-1}}}{}^2 + C\norm{\mu_h^n}{}^2.
\end{align}
Handling the last term on the right-hand side of \eqref{eq:CHNS:equation_4} is technical. Here, we provide an outline of the proof and suppress details for brevity. Considering the definition of $a_{\mathrm{adv}}$, integrating by parts and  and rearranging terms,  yield
\begin{align*}
a_{\mathrm{adv}}(&c_h^{n-1}, \vec{u}_h^{n-1}, \delta_\tau{c_h^n})
=\, \sum_{E\in\setE_h} \int_E \grad{c_h^{n-1}}\cdot\vec{u}_h^{n-1}\, \delta_\tau{c_h^n}
+ \sum_{E\in\setE_h} \int_E c_h^{n-1}\div{\vec{u}_h^{n-1}}\, \delta_\tau{c_h^n}\\
- &\sum_{e\in\Gammah\cup\partial{\Omega}} \int_e \avg{c_h^{n-1}\delta_\tau{c_h^n}}\jump{\vec{u}_h^{n-1}\cdot\normal_e}
- \sum_{e\in\Gammah} \int_e \jump{c_h^{n-1}}\avg{\vec{u}_h^{n-1}\cdot\normal_e}\avg{\delta_\tau{c_h^n}}= \sum_{i=1}^4 T_i.
\end{align*}
For the terms $T_1$ to $T_3$, we apply Holder's, Poincar\'e's,  trace and triangle inequalities. We also use \eqref{eq:fully_dis6}, \eqref{eq:def_b_lift_2}, and inverse estimate \eqref{eq:inverse_estimate_dg} to deduce that 
\[\|\bm{u}_h^{n-1}\|_{\DG} \leq C_\mathrm{inv}h^{-1}\|\bm{u}_h^{n-1} - \bm{v}_h^{n-1}\| + \|\bm{v}_h^{n-1}\|_{\DG} \leq C \tau h^{-1} |\phi_h^{n-1}|_{\DG} +\|\bm{v}_h^{n-1}\|_{\DG}. \] 
Thus, with the above bound, we obtain  
\begin{align*}
\sum_{i=1}^3 \abs{T_i} \leq&\, C(\norm{\gradh{c_h^{n-1}}}{L^3(\Omega)} + \norm{c_h^{n-1}}{L^\infty(\Omega)})(\tau h^{-1}  \vert \phi_h^{n-1}\vert_{\mathrm{DG}} + \norm{\vec{v}_h^{n-1}}{\mathrm{DG}})  \norm{\delta_\tau{c_h^n}}{}\\
\leq&\, \frac{1}{12}\norm{\delta_\tau{c_h^n}}{}^2 + C(\norm{\gradh{c_h^{n-1}}}{L^3(\Omega)} + \norm{c_h^{n-1}}{L^\infty(\Omega)})^2(\tau^2h^{-2}\vert\phi_h^{n-1}\vert_{\mathrm{DG}}^2 + \norm{\vec{v}_h^{n-1}}{\mathrm{DG}}^2).
\end{align*}
To handle $T_4$, we add and subtract the approximation $\Pi_h \bm{u}^{n-1}$
\begin{align*}
T_4 &= 
- \sum_{e\in\Gammah} \int_e \jump{c_h^{n-1}}\avg{(\vec{u}_h^{n-1} - \Pi_h \bm{u}^{n-1}) \cdot\normal_e}\avg{\delta_\tau{c_h^n}}  
- \sum_{e\in\Gammah} \int_e \jump{c_h^{n-1}}\avg{\Pi_h \bm{u}^{n-1}\cdot\normal_e}\avg{\delta_\tau{c_h^n}}  \\
&= T_4^1+T_4^2. 
\end{align*}
For $T_4^1$, we split it by inserting $\bm{v}_h^{n-1}$. We use trace inequality \eqref{eq:trace_estimate_discrete}, inverse estimate \eqref{eq:inverse_estmate_lp},  Poincar\'e's and Young's inequalities. With \eqref{eq:fully_dis6}, \eqref{eq:CHNS:Stability}, and \eqref{eq:def_b_lift_2}, we obtain
\begin{multline}
|T_4^1| \leq C h^{-1}\|c_h^{n-1}\|_{L^\infty(\Omega)} \|\bm{u}_h^{n-1} - \bm{v}_h^{n-1}\| \|\delta_\tau c_h^{n}\| \\ + C |c_h^{n-1}|_{\DG} \|\bm{v}_h^{n-1} - \Pi_h \bm{u}^{n-1}\|_{L^6(\Omega)}\|\delta_\tau c_h^n\|_{L^3(\Omega)} 
 \\ \leq \frac{1}{24}\|\delta_\tau c_h^n\|^2 + C \tau^2 h^{-2} \|c_h^{n-1}\|_{L^{\infty}(\Omega)} ^2 |\phi_h^{n-1}|^2_{\DG}  + C h^{-d/3}\|\bm{v}_h^{n-1} - \Pi_h \bm{u}^{n-1}\|^2_{\DG}. \nonumber 
\end{multline}
For $T_4^2$, we simply have  
\begin{equation} |T_4^2| \leq C \|\Pi_h \bm{u}^{n-1}\|_{L^{\infty}(\Omega)} \vert c_h^{n-1}\vert_{\mathrm{DG}} \|\delta_\tau c_h^n\| \leq \frac{1}{24} \|\delta_\tau c_h^n\|^2 + C.  \nonumber 
\end{equation} 
We combine the bounds on $T_1$ to $T_4$ with \eqref{eq:stab_mu_L2_ctrl}. 
We obtain 
\begin{multline}
\abs{a_{\mathrm{adv}}(c_h^{n-1},  \vec{u}_h^{n-1}, \delta_\tau{c_h^n})}
\leq \frac{1}{6}\norm{\delta_\tau{c_h^n}}{}^2 + Ch^{-d/3}\|\bm{v}_h^{n-1} - \Pi_h \bm{u}^{n-1}\|^2_{\DG}
\\ + C(1+ \|\mu_h^{n-1}\|^2 )(\norm{\vec{v}_h^{n-1}}{\mathrm{DG}}^2 + \tau^2 h^{-2} \vert \phi_h^{n-1}\vert_{\mathrm{DG}}^2) + C. \label{eq:bound_adv}
\end{multline}
We substitute  bounds \eqref{eq:phi_plus_bound}, \eqref{eq:bound_phi-}, and \eqref{eq:bound_adv} into \eqref{eq:CHNS:equation_4}, multiply by $2 \tau$,
sum from $n = 1$ to $\ell$, with $\ell \leq m$ (recall that $\delta_\tau c_h^0 = 0$ by convention):
\begin{align*}
&\norm{\mu_h^{\ell}}{}^2 + \kappa\tau\sum_{n=1}^{\ell}\norm{\delta_\tau{c_h^{n}}}{}^2 + \frac{\kappa\tau}{3}\norm{\delta_\tau{c_h^{\ell}}}{}^2 \leq  C T + C\kappa\tau h^{-d/3} \sum_{n=0}^{\ell-1} \|\bm{v}_h^{n} - \Pi_h \bm{u}^{n}\|^2_{\DG} \\ 
&  +  C\kappa \tau\sum_{n=0}^{\ell-1}(1 + \|\mu_h^{n}\|^2)(\norm{\vec{v}_h^{n}}{\mathrm{DG}}^2 + \tau^2 h^{-2}\vert \phi_h^{n}\vert_{\mathrm{DG}}^2)+ C\tau\sum_{n=1}^\ell(\Vert \mu_h^n\Vert^2 + \vert \mu_h^{n}\vert_{\mathrm{DG}}^2) + \norm{\mu_h^{0}}{}^2. 
\end{align*}
%
Take $\varphi_h = \mu_h^n$ in \eqref{eq:fully_dis2}, use \eqref{eq:adiffcont}, Cauchy--Schwarz's and Young's inequalities, we have
\begin{align}
\norm{\mu_h^n}{}^2 
&\leq \norm{(c_h^n)^3 - c_h^{n-1}}{}\norm{\mu_h^n}{} + C_\alpha \kappa \vert c_h^n \vert_{\mathrm{DG}} \vert \mu_h^n \vert_{\mathrm{DG}} \\
&\leq \frac{1}{2}\norm{\mu_h^n}{}^2 + \norm{c_h^n}{L^6(\Omega)}^6 + \norm{c_h^{n-1}}{}^2 + \frac{C_\alpha \kappa^2}{2}\vert c_h^n \vert_{\mathrm{DG}}^2 + \frac{C_\alpha}{2}\vert \mu_h^n \vert_{\mathrm{DG}}^2. \nonumber
\end{align}
Multiply by $\tau$ the above inequality and sum from $n=1$ to $\ell$. By stability bounds \eqref{eq:CHNS:Stability}, \eqref{eq:L2boundchLp}, and \eqref{eq:L2_bound_init_mu_h}, we have
\begin{align}\label{eq:stab_mu_L2_intermediate}
\tau  \sum_{n=1}^{\ell} \|\mu_h^{n}\|^2 \leq C + CT.
\end{align}
With the induction hypothesis \eqref{eq:induction_hyp_linf} and \eqref{eq:cfl_cond}, we have 
\begin{multline}
 \tau\sum_{n=0}^{\ell-1} ( \tau^2 h^{-2} \vert \phi_h^{n}\vert_{\mathrm{DG}}^2 + h^{-d/3}\|\bm{v}_h^{n} - \Pi_h \bm{u}^{n}\|_{\DG}^2)  
\leq (\tau h^{-2} + h^{-d/3}) (\tau^{\frac{1}{1+\delta}} +h^{\frac{2+\delta}{1+\delta}})\\
\leq \gamma^{\frac{2+\delta}{1+\delta}} + \gamma^{\frac{1}{1+\delta}} + \gamma h + h^{1/2} \leq C. 
\end{multline}
With the above bounds and  \eqref{eq:CHNS:Stability}, we obtain
\begin{multline}\label{eq:bound_mu_delta_tau_c_inter}
\norm{\mu_h^{\ell}}{}^2 + \kappa\tau\sum_{n=1}^{\ell}\norm{\delta_\tau{c_h^{n}}}{}^2 
\leq C T + C 
+ C\kappa\tau\sum_{n=0}^{\ell-1} (\norm{\vec{v}_h^{n}}{\mathrm{DG}}^2 + \tau^2 h^{-2} \vert \phi_h^{n}\vert^2_{\mathrm{DG}})\|\mu_h^{n}\|^2. 
\end{multline}
%
We now apply the Gronwall's inequality \cite{HeywoodRannacher1990} to obtain 
\begin{align*}
\norm{\mu_h^{\ell}}{}^2 + \kappa\tau\sum_{n=1}^{\ell}\norm{\delta_\tau{c_h^{n}}}{}^2
\leq C(1+T), 
\end{align*}
which concludes the proof.
Finally, it is important to point out that, in the above proof, every generic constant $C$ is independent of the induction iteration index $m$. In other words, at  each induction iteration the constant $C$ in \eqref{eq:stability_infty1} is unchanged.
\end{proof}
\subsection{Intermediate results (Steps (ii) and (iii))}
To simplify the writeup, we define the following projection and discretization error functions. 
\begin{alignat*}{3}
    \eta_c^n  &= \mathcal{P}_h c^n - c^n , && \quad  \eta_\mu^n = \mathcal{P}_h \mu^n - \mu^n ,&& \quad \eta_p^n  = \pi_h p^n - p^n ,  \\ \xi_c^n &  = c_h^n - \mathcal{P}_h c^n, && \quad \xi_\mu^n = \mu_h^n - \mathcal{P}_h \mu^n, && \quad \xi_p^n  = p_h^n - \pi_h p^n, \\ 
\bm{\eta}_{\bm{u}}^n &= \Pi_h \bm{u}^n - \bm{u}^n, && \quad  \bm{\xi}_{\bm{v}}^n = \bm{v}_h^n - \Pi_h \bm{u}^n , && \quad \bm{\xi}_{\bm{u}}^n = \bm{u}_h^n - \Pi_h \bm{u}^n. 
\end{alignat*}
From the consistency equations \eqref{eq:consistency_1}-\eqref{eq:consistency_3}, the fully discrete scheme \eqref{eq:fully_dis1}-\eqref{eq:fully_dis6}, and the definition of $\mathcal{P}_h$ \eqref{eq:def_elliptic_projection}, we obtain the following error equations for all $1\leq n\leq \Nst$. For all $\chi_h \in M_h^{k}$,
\begin{multline} 
    (\delta_\tau \xi_c^n, \chi_h)  + \adif(\xi_\mu^n, \chi_h) \\  = (-\delta_\tau c^n + (\partial_t c)^n - \delta_\tau \eta_c^n, \chi_h) - \aadv(c_h^{n-1}, \bm{u}_h^{n-1}, \chi_h) + \aadv(c^{n}, \bm{u}^n, \chi_h). \label{eq:first_err_eq}
\end{multline}
For all $\varphi_h \in M_h^{k}$,
\begin{multline}
    \kappa \adif(\xi_c^n, \varphi_h) -(\xi_\mu^n, \varphi_h) \\  = (\eta_\mu^n, \varphi_h) + (\Phi_+\,\!'(c^n) - \Phi_+\,\!'(c_h^n),\varphi_h) + (\Phi_-\,\!'(c^{n}) - \Phi_-\,\!'(c_h^{n-1}), \varphi_h).\label{eq:second_err_eq}
\end{multline}
For all $\bm{\theta}_h \in \mathbf{X}_h^{k}$, 
\begin{align} 
& \frac{1}{\tau}( \errv^n- \erru^{n-1}, \bm{\theta}_h) +  a_\mathcal{C}(\bm{u}_h^{n-1},\bm{u}_h^{n-1}, \errv^n, \bm{\theta}_h) +  \mu_\mathrm{s} a_{\strain}(\errv^n, \bm{\theta}_h)  = - \mu_\mathrm{s} a_{\strain}(\bm{\eta}_{\bm{u}}^n, \bm{\theta}_h )
\label{eq:third_err_eq} \\ \quad & + b_{\mathcal{P}}(\bm{\theta}_h, p_h^{n-1} - p^n) + a_\mathcal{C}(\bm{u}^n,\bm{u}^n, \bm{u}^n,\bm{\theta}_h) - a_\mathcal{C}(\bm{u}_h^{n-1}, \bm{u}_h^{n-1}, \Pi_h \bm{u}^n, \bm{\theta}_h) \nonumber \\ &  \quad+ ((\partial_t \bm{u})^n - \delta_\tau \bm{u}^n - \delta_\tau \bm{\eta}_{\bm{u}}^n , \bm{\theta}_h) + b_{\mathcal{I}}(c_h^{n-1},\mu_h^n, \bm{\theta}_h) - b_{\mathcal{I}}(c^n, \mu^n, \bm{\theta}_h) \nonumber . 
\end{align}
The next lemma gives a bound on the last two terms in \eqref{eq:first_err_eq}. 
\begin{lemma}\label{lemma:bound_aadv_terms}
There exists a constant $C$ independent of $h$, $\tau$ such that for any $\epsilon>0$ and $\chi_h \in M_h^{k}$, the bound holds
for all $n\geq 1$:
\begin{multline}
|\aadv(c^{n}, \bm{u}^n, \chi_h) - \aadv(c_h^{n-1}, \bm{u}_h^{n-1}, \chi_h ) |
\leq  5\epsilon \vert\chi_h \vert_{\mathrm{DG}}^2 + \frac{C}{\epsilon} h^{2k+2} \\
+ \frac{C}{\epsilon} \tau \int_{t^{n-1}}^{t^n} \left(\|\partial_t c\|^2_{H^1(\Omega)} + \|\partial_t \bm{u}\|^2_{H^1(\Omega)}\right)  + \frac{C}{\epsilon} \left( \|\xi_c^{n-1}\|^2  + \|\erru^{n-1}\|^2 \right). 
\end{multline}
\end{lemma}
The proof is in Appendix~\ref{sec:app1}.
We are now ready to complete Step (ii). 
\begin{lemma} \label{lemma:bound_xic}
Fix $m$, with $1\leq m \leq \Nst$. Assume that \eqref{eq:induction_hyp_linf} and \eqref{eq:cfl_cond} hold. Then, for any $n$ with $1\leq n \leq m$, we have   for a constant $C$ independent of $h$ and $\tau$ 
\begin{multline}
 \frac{K_\alpha}{2} \tau \vert\mathcal{J}(\delta_\tau \xi_c^n)\vert_{\mathrm{DG}}^2 +  \kappa \left(\adif(\xi_c^{n}, \xi_c^n) - \adif(\xi_c^{n-1},\xi_c^{n-1})\right)   +  \kappa K_\alpha \vert\xi_c^n - \xi_c^{n-1}\vert_{\mathrm{DG}}^2  \label{eq:error_eq_xi_c} \\
\leq C \tau \int_{t^{n-1}}^{t^n} \left( \tau \|\partial_{tt}c \|^2 +\tau^{-1} h^{2k}|\partial_t c|^2_{H^{k+1}(\Omega)} + \tau \|\partial_t c\|^2_{H^1(\Omega)} + \tau \|\partial_t \bm{u}\|_{H^1(\Omega)}^2 \right)  \\ + C \tau h^{2k}  
+ C \tau \left(\vert\xi_c^n\vert_{\mathrm{DG}}^2 + \vert\xi_c^{n-1}\vert_{\mathrm{DG}}^2 +  \|\erru^{n-1} \|^2\right). 
     \end{multline}
\end{lemma}
\begin{proof}
Note that with the definition of the operator $\mathcal{J}$ in  \eqref{eq:def_of_J}, for any $\chi_h \in M_{h0}^k$ and $\phi_h \in M_h^k$, we have: 
\begin{align}
\adif(\mathcal{J}(\chi_h), \phi_h) = \adif(\mathcal{J}(\chi_h), \phi_h - \overline{\phi_h}) = (\chi_h, \phi_h -\overline{\phi_h} ) = (\chi_h, \phi_h). 
\end{align}Thus, we obtain 
\begin{align}
(\delta_\tau \xi_c^n, \mathcal{J}(\delta_\tau \xi_c^n)) &= \adif(\mathcal{J}(\delta_\tau \xi_c^n),\mathcal{J}(\delta_\tau \xi_c^n)), \quad 
\adif(\mathcal{J}(\delta_\tau \xi_c^n), \xi_\mu^n) = (\xi_\mu^n , \delta_\tau \xi_c^n ). \nonumber
\end{align}
Let $\chi_h = \mathcal{J}(\delta_\tau \xi_c^n)$ in \eqref{eq:first_err_eq}, and use the coercivity properties of $\adif$ \eqref{eq:coercivity_adiff}. With the above equalities, we obtain  
\begin{multline}
K_\alpha \vert\mathcal{J}(\delta_\tau \xi_c^n)\vert_{\mathrm{DG}}^2 + (\xi_\mu^n , \delta_\tau  \xi_c^n )  \leq  (-\delta_\tau c^n + (\partial_t c)^n - \delta_\tau \eta_c^n, \mathcal{J}(\delta_\tau \xi_c^n)) \nonumber\\  \quad - \aadv(c_h^{n-1}, \bm{u}_h^{n-1}, \mathcal{J}(\delta_\tau \xi_c^n)) + \aadv(c^{n}, \bm{u}^n, \mathcal{J}(\delta_\tau \xi_c^n) )  =T_1+T_2+T_3.
\end{multline}  
Let $\varphi_h = \delta_\tau \xi_c^n$ in \eqref{eq:second_err_eq}. With the symmetry and coercivity of $\adif$, we have: 
\begin{multline}
\frac{\kappa}{2\tau}  \left(\adif(\xi_c^{n}, \xi_c^n) - \adif(\xi_c^{n-1},\xi_c^{n-1}) + K_\alpha \vert\xi_c^n - \xi_c^{n-1}\vert_{\mathrm{DG}}^2\right) - (\xi_\mu, \delta_\tau \xi_c^n) \nonumber \\  \leq (\eta_\mu^n,\delta_\tau \xi_c^n) + (\Phi_+\,\!'(c^n) - \Phi_+\,\!'(c_h^n),\delta_\tau \xi_c^n) + (\Phi_-\,\!'(c^{n}) - \Phi_-\,\!'(c_h^{n-1}), \delta_\tau \xi_c^n) = T_4+T_5+T_6. 
\end{multline}
Adding the above two inequalities yields: 
\begin{multline}
K_\alpha  \vert\mathcal{J}(\delta_\tau \xi_c^n)\vert_{\mathrm{DG}}^2 +  \frac{\kappa}{2\tau}  \left(\adif(\xi_c^{n}, \xi_c^n) - \adif(\xi_c^{n-1},\xi_c^{n-1}) \right) \\ + \frac{\kappa K_\alpha}{2\tau}  \vert\xi_c^n - \xi_c^{n-1}\vert_{\mathrm{DG}}^2 \leq \sum_{i=1}^{6} T_i. \label{eq:step2_error_eq0}
\end{multline}
With Cauchy--Schwarz's, Poincar\'e's and Young's inequalities, Taylor expansions, and \eqref{eq:elliptic_projection_error}, we obtain 
\begin{align}
  |T_1| & \leq \frac{C}{K_\alpha} (\|\delta_\tau c^n - (\partial_t c)^n\|^2 + \|\delta_\tau \eta_c^n\|^2 )  + \frac{K_\alpha}{18} \vert\mathcal{J}(\delta_\tau \xi_c^n)\vert_{\mathrm{DG}}^2 \label{eq:bound_T1} \\ 
  & \leq \frac{C}{K_\alpha}\int_{t^{n-1}}^{t^n} \big(\tau \| \partial_{tt} c\|^2 + \tau^{-1} h^{2k}|\partial_t c|^2_{H^{k+1}(\Omega)}\big) + \frac{K_\alpha}{18} \vert\mathcal{J}(\delta_\tau \xi_c^n)\vert_{\mathrm{DG}}^2. \nonumber
\end{align}
The terms $T_2 + T_3$ are bounded via Lemma~\ref{lemma:bound_aadv_terms} with $\chi_h = \mathcal{J}(\delta_\tau \xi_c^n)$ and $\epsilon = K_\alpha/18$. We also use Lemma~\ref{lemma:L_infty_bound} and obtain 
\begin{align}
|T_2  + T_3| \leq & \frac{5K_\alpha}{18} \vert\mathcal{J}(\delta_\tau \xi_c^n)\vert_{\mathrm{DG}}^2 + \frac{C}{K_\alpha} \tau \int_{t^{n-1}}^{t^n} \Big(\|\partial_t c\|^2_{H^1(\Omega)} + \|\partial_t \bm{u}\|^2_{H^1(\Omega)}\Big) \\ &  + \frac{C}{K_\alpha}h^{2k} + \frac{C}{K_\alpha} ( \vert\xi_c^{n-1}\vert_{\mathrm{DG}}^2 + \|\bm{\xi}_{\bm{u}}^{n-1}\|^2  ).
\end{align} 
The term $T_4$ is bounded with Lemma~\ref{lemma:boundedness_J}, Young's inequality, and \eqref{eq:elliptic_projection_error}. 
\begin{equation}
|T_4| \leq \frac{C}{K_\alpha} h^{2k} \|\mu\|^2_{L^{\infty}(0,T; H^{k+1}(\Omega))} + \frac{K_\alpha}{18} \vert\mathcal{J}(\delta_\tau \xi_c^n)\vert_{\mathrm{DG}}^2.
\end{equation}
For $T_5$ and $T_6$, we insert $\Phi_-\,\!'(c^{n-1})$ and 
apply Lemma \ref{lemma:boundedness_J} to obtain: 
\begin{multline} 
|T_5|+|T_6|\leq C \vert\Phi_+\,\!'(c^n) - \Phi_+\,\!'(c_h^n)\vert_{\mathrm{DG}}  \vert\mathcal{J}(\delta_\tau \xi_c^n)\vert_{\mathrm{DG}} \\ + C \left(\vert\Phi_-\,\!'(c^{n}) - \Phi_-\,\!'(c^{n-1})\vert_{\mathrm{DG}} + \vert\Phi_-\,\!'(c^{n-1}) - \Phi_-\,\!'(c_h^{n-1})\vert_{\mathrm{DG}} \right)\vert\mathcal{J}(\delta_\tau \xi_c^n)\vert_{\mathrm{DG}}.\label{eq:bound_T3T6_0} \end{multline}
For the first term and for the Ginzburg--Landau potential,  we have 
Using the $L^\infty$ bounds on $c^n$ and $c_h^n$ and H\"older's inequality, we can write
\[
\vert (c^n)^3 - (c_h^n)^3\vert_{\mathrm{DG}}
\leq C \Vert \nabla_h (c^n-c_h^n)\Vert +
C  \Vert c^n - c_h^n \Vert_{L^6(\Omega)} \left( \Vert \nabla c^n \Vert_{L^3(\Omega)}
+ \Vert \nabla_h c_h^n\Vert_{L^3(\Omega)}\right).
\]
 Lemma~\ref{lemma:L_infty_bound} and Poincar\'e's inequality applied to $c^n-c_h^n$ with \eqref{thm:CHNS:discrete_mass_conservation}
yield
\[
 \vert (c^n)^3 - (c_h^n)^3\vert_{\mathrm{DG}} 
\leq C \vert c_h^n -c^n\vert_{\mathrm{DG}}.
\]
We substitute the above bound in \eqref{eq:bound_T3T6_0}, and we apply a Taylor's expansion, \eqref{eq:elliptic_projection_error} and Young's inequality. We obtain 
\begin{align*}
|T_5| + |T_6| \!&\leq\! C(|c_h^n - c^n|_{\DG} 
+  |c^n - c^{n-1}|_{\DG} + |c^{n-1} - c_h^{n-1}|_{\DG})|\mathcal{J}(\delta_\tau \xi_c^n)|_{\DG}  \\ 
&\leq \frac{C}{K_\alpha} \Big(\vert \xi_c^n\vert_{\mathrm{DG}}^2 + \vert\xi_c^{n-1}\vert_{\mathrm{DG}}^2 + h^{2k} \|c\|^2_{L^{\infty}(0,T; H^{k+1}(\Omega))} \Big) \\  
& \qquad+ \frac{C}{K_\alpha}  \tau\int_{t^{n-1}}^{t^n}\|\partial_t c\|_{H^{1}(\Omega)}^2 +  \frac{K_\alpha}{18} \vert \mathcal{J}(\delta_\tau \xi_c^n)\vert_{\mathrm{DG}}^2. \nonumber
\end{align*}
Using the bounds on the terms $T_i$'s in \eqref{eq:step2_error_eq0} and multiplying by $2\tau$ yield the result. 
\end{proof}
Next, we derive a bound for $\xi_\mu^n$ in the energy norm.
\begin{lemma}\label{lemma:bound_ximu}
    Fix $m$, with $1\leq m \leq \Nst$. Assume that \eqref{eq:induction_hyp_linf} and \eqref{eq:cfl_cond} hold. Then, for any $n$ with $1\leq n \leq m$, we have   for a constant $C$ independent of $h$ and $\tau$ 
\begin{align}
\frac{K_\alpha}{2}  \vert\xi_\mu^n\vert_{\mathrm{DG}}^2  \!\leq\! &\frac{4C_\mathrm{J}^2 }{K_\alpha} \vert\mathcal{J}(\delta_\tau \xi_c^n)\vert_{\mathrm{DG}}^2 + C \tau\! \int_{t^{n-1}}^{t^n}\! \big(\|\partial_t c\|^2_{H^1(\Omega)} \!+ \|\partial_t \bm{u}\|^2_{H^1(\Omega)} \!+ \|\partial_{tt} c\|^2\big) \label{eq:error_xi_mu} \\ &  + C h^{2k} \Big( 1 + \tau^{-1} \int_{t^{n-1}}^{t^n} |\partial_t c|_{H^{k+1}(\Omega)}^2 \Big) + C\big( \vert \xi_c^{n-1}\vert_{\mathrm{DG}}^2  + \|\erru^{n-1} \|^2 \big). \nonumber
\end{align}
\end{lemma}
\begin{proof}
Let $\chi_h = \xi_\mu^n$ in \eqref{eq:first_err_eq}. With the coercivity property of $\adif$, we obtain: 
\begin{align}
K_{\alpha} \vert\xi_\mu^n\vert_{\mathrm{DG}}^2 \leq -(\delta_\tau \xi_c^n, \xi_\mu^n) +(-\delta_\tau c^n+ (\partial_t c)^n - \delta_\tau \eta_c^n, \xi_\mu^n)  \label{eq:intres1}\\
 - \aadv(c_h^{n-1}, \bm{u}_h^{n-1}, \xi_\mu^n) + \aadv(c^{n}, \bm{u}^n, \xi_\mu^n). \nonumber
\end{align}
The first term is bounded by Lemma \ref{lemma:boundedness_J} and Young's inequality. We have: 
\begin{align}
|(\delta_\tau \xi_c^n, \xi_\mu^n)| \leq \frac{K_\alpha}{16} \vert\xi_\mu^n\vert_{\mathrm{DG}}^2 + \frac{4C_\mathrm{J}^2 }{K_\alpha} \vert\mathcal{J}(\delta_\tau \xi_c^n)\vert_{\mathrm{DG}}^2. \nonumber 
 \end{align}
By taking $\chi_h = 1$ in \eqref{eq:first_err_eq}, we observe that the average of $\delta_\tau c^n -(\partial_t c)^n + \delta_\tau \eta_c^n$ is zero. Hence,  we bound the second term with Cauchy--Schwarz's and Poincar\'e's inequalities.  
\begin{align}
& |(\delta_\tau c^n -(\partial_t c)^n + \delta_\tau \eta_c^n, \xi_\mu^n)| = |(\delta_\tau c^n -(\partial_t c)^n  + \delta_\tau \eta_c^n, \xi_\mu^n - \overline{\xi_\mu^n})| \\ 
& \leq \frac{K_\alpha}{16} \vert\xi_\mu^n\vert_{\mathrm{DG}}^2 + \frac{C}{K_\alpha}\int_{t^{n-1}}^{t^n} \big(\tau \| \partial_{tt} c\|^2 + \tau^{-1} h^{2k}|\partial_t{c}|^2_{H^{k+1}(\Omega)}\big). \nonumber
\end{align}
We use the above two bounds in \eqref{eq:intres1}  and Lemma~\ref{lemma:bound_aadv_terms} with $\chi_h = \xi_\mu^n$ and $\epsilon = K_\alpha/16$ to bound the last two terms. We then conclude by  using Lemma~\ref{lemma:L_infty_bound}.
\end{proof}
We now show an estimate involving the errors $\erru$ and $\errv$.  To this end, we denote 
\begin{equation*}
    A^n_1  = \sum_{e\in\Gamma_h} \frac{\tilde{\sigma}}{h}\|[\phi_h^n]\|^2_{L^2(e)} - \sum_{e\in\Gamma_h} \frac{\tilde{\sigma}}{h}\|[\phi_h^{n-1}]\|_{L^2(e)}^2, \,\,\,  
    A^n_2  = \|\bm{G}_h([\phi_h^n]) \|^2 - \| \bm{G}_h([\phi_h^{n-1}])\|^2 .
\end{equation*}
 In the following lemma, the notation $\delta_{n,1}$ is the Kronecker symbol: $\delta_{1,1}= 1$ and $\delta_{n,1} = 0$
for $n > 1$. 
\begin{lemma}\label{lemma:err_eq_pressure_corretion}
Assume that $\sigma \geq \tilde{M}_{k-1}^2/d$, $\tilde{\sigma} \geq 4\tilde{M}_{k}^2$, and $\sigma_\chi \leq K_\strain/(2d)$. 
There exists a constant $C$  independent of $h$ and $\tau$  such that for all $n \geq 1$
and any $\epsilon >0$, we have
\begin{multline}
\frac{1}{2\tau} \left(\|\erru^n\|^2 - \|\erru^{n-1}\|^2 + \|\erru^n - \erru^{n-1}\|^2\right) + \frac{K_\strain\mu_\mathrm{s}}{4} \|\errv^n\|_{\DG}^2 + \frac{\tau}{16}\vert \phi_h^n \vert_{\mathrm{DG}}^2 \label{eq:err_eq_pressure_corretion}\\ 
+ \frac{1}{2\sigma_\chi \mu_\mathrm{s}} \left(\|S_h^n \|^2 - \|S_h^{n-1}\|^2\right)  + \frac{\tau}{2} \big(\adif(\zeta_h^n, \zeta_h^n) -  \adif(\zeta_h^{n-1},\zeta_h^{n-1}) + A^n_1 - A^n_2\big) \leq \epsilon  \vert\xi_\mu^n\vert_{\mathrm{DG}}^2 \\ 
+ C (h^{2k} +  \tau) 
 + C\Big(\|\erru^{n-1}\|^2 + \vert \xi_c^{n-1}\vert_{\mathrm{DG}}^2 + \frac{1}{\epsilon^2} \big(\|\erru^n\|^2 + \tau^2  \vert\phi_h^n\vert_{\mathrm{DG}}^2\big) \Big) \\ 
+ C\tau\int_{t^{n-1}}^{t^n} \big( \|\partial_t \bm{u}\|^2 + \|\partial_{tt} \bm{u}\|^2 +\|\partial_t c\|^2 + \tau^{-2} h^{2k} |\partial_t \bm{u} |^2_{H^{k}(\Omega)} \big)  + \delta_{n,1}\abs{b_{\mathcal{P}}(\erru^0, \phi_h^1)}.  
\end{multline} 
\end{lemma}
\begin{proof}
Taking $\bm{\theta}_h = \errv^n$ in \eqref{eq:third_err_eq} and using \eqref{eq:aCformpos} and \eqref{eq:coercivity_astrain} yield: 
\begin{multline}
 \frac{1}{2\tau} \left(\|\errv^n\|^2 - \|\erru^{n-1}\|^2 + \|\errv^n - \erru^{n-1}\|^2\right) + \mu_\mathrm{s} K_\strain \|\errv^n\|_{\DG}^2  \label{eq:conv_0} \\    
\leq  - \mu_\mathrm{s} a_{\strain}(\bm{\eta}_{\bm{u}}^n, \errv^n )
+ b_{\mathcal{P}}(\errv^n, p_h^{n-1} - p^n) + a_\mathcal{C}(\bm{u}^n,\bm{u}^n, \bm{u}^n,\errv^n) - a_\mathcal{C}(\bm{u}_h^{n-1}, \bm{u}_h^{n-1}, \Pi_h \bm{u}^n, \errv^n)  \\  +((\partial_t \bm{u})^n - \delta_\tau \bm{u}^n - \delta_\tau \bm{\eta}_{\bm{u}}^n , \errv^n) + b_{\mathcal{I}}(c_h^{n-1},\mu_h^n, \errv^n)  - b_{\mathcal{I}}(c^n, \mu^n, \errv^n )= \sum_{i=1}^{7}W_i.
\end{multline}   
Following similar arguments to the proof of Theorem~1 in \cite{inspaper1}, we can prove  (for completeness, we provide some details in the Appendix~\ref{sec:app:conv_1add2})
\begin{multline}
    \frac{1}{2\tau} \left( \|\erru^n \|^2 - \|\errv^n\|^2 +\|\erru^n - \erru^{n-1}\|^2 \right) +\frac{\tau}{4} | \phi_h^n|_{\DG}^2   + \frac{\tau}{2} \adif(\phi_h^n,\phi_h^n) \\+\frac{\tau}{2}(A_1^n - A_2^n)    +  \frac{\tau}{4} \sum_{e\in\Gamma_h} \frac{\tilde{\sigma}}{h}\|[\phi_h^n - \phi_h^{n-1}]\|^2_{L^2(e)}   \leq \frac{1}{2\tau}\|\errv^n  -\erru^{n-1}\|^2  + \delta_{n,1}  \abs{b_{\mathcal{P}}(\bm{\xi}_{\bm{u}}^0, \phi_h^1)}.\label{eq:conv_1add2}
\end{multline}
We add \eqref{eq:conv_1add2} with \eqref{eq:conv_0} to obtain  
\begin{multline}
\frac{1}{2\tau} \left(\|\erru^n\|^2 - \|\erru^{n-1}\|^2 + \|\erru^n - \erru^{n-1}\|^2\right) + \mu_\mathrm{s} K_\strain \|\errv^n\|_{\DG}^2  
+ \frac{\tau}{2}(A_1^n - A_2^n) \\ + \frac{\tau}{4}\vert\phi_h^n\vert_{\mathrm{DG}}^2 +  \frac{\tau}{4} \sum_{e\in\Gamma_h} \frac{\tilde{\sigma}}{h}\|[\phi_h^n - \phi_h^{n-1}]\|^2_{L^2(e)}  + \frac{\tau}{2} \adif(\phi_h^n, \phi_h^n) 
\leq   \delta_{n,1}  \abs{b_{\mathcal{P}}(\erru^0,\phi_h^1)} + \sum_{i=1}^7 W_i. \label{eq:conv_3}
\end{multline}  
The bounds for $W_i$ for $i = 1,\ldots, 5$ are handled in a similar way as in \cite{inspaper1}. We recall the main ideas for completeness. 
Using the continuity of $a_\mathcal{D}$, we have 
\begin{align}
|W_1| 
\leq C\mu_\mathrm{s} h^{2k} |\bm{u}^n|_{H^{k+1}(\Omega)}^2 + \frac{K_\mathcal{D} \mu_\mathrm{s}}{32} \|\errv^n\|_{\DG}^2. 
\end{align}
To handle $W_2$, we split it as follows. 
\begin{align}
W_2 =  b_{\mathcal{P}}(\errv^n, p_h^{n-1} - p^n) = b_{\mathcal{P}}(\errv^n, p_h^{n-1}) -  b_{\mathcal{P}}(\errv^n, \pi_h p^n) +  b_{\mathcal{P}}(\errv^n, \pi_h p^n - p^n).
\end{align} 
For the first term, recall that $b_{\mathcal{P}}(\Pi_h \bm{u}^n, p_h^{n-1}) = b_{\mathcal{P}}(\bm{u}^n, p_h^{n-1}) = 0$. From  \eqref{eq:expression_b_pn},
\begin{align}
b_{\mathcal{P}}(\errv^n, p_h^{n-1}) = b_{\mathcal{P}}(\vec{v}_h^n, p_h^{n-1})
= -\frac{\tau}{2}a_{\mathrm{diff}}(\zeta_h^n,\zeta_h^n) + \frac{\tau}{2}a_{\mathrm{diff}}(\zeta_h^{n-1},\zeta_h^{n-1}) + \frac{\tau}{2}a_{\mathrm{diff}}(\phi_h^n,\phi_h^n) \nonumber\\
- \frac{1}{2\sigma_\chi\mu_\mathrm{s}}\Big(\norm{S_h^n}{}^2 - \norm{S_h^{n-1}}{}^2 - \norm{S_h^n-S_h^{n-1}}{}^2\Big). \nonumber
\end{align}
Using the fact $\divh{\Pi_h{\vec{u}^n}} - R_h(\jump{\Pi_h{\vec{u}^n}})=0$, taking $\sigma_\chi \leq K_\strain/(2d)$ and $\sigma \geq \tilde{M}_{k-1}^2/d$, by \eqref{eq:lift_prop_r}, we have
\begin{equation}
\frac{1}{2\sigma_\chi\mu_\mathrm{s}} \|S_h^n - S_h^{n-1}\|^2 = \frac{\sigma_\chi\mu_\mathrm{s}}{2} \|\divh{\errv^n} - R_h([\errv^n])\|^2 \leq \frac{K_\strain \mu_\mathrm{s}}{2}\|\errv^n\|_{\DG}^2. \nonumber
\end{equation}
Since $\pi_h p^n \in M^{k-1}_{h0}$,  we use  \eqref{eq:fully_dis4}, \eqref{eq:adiffcont}, and stability of the $L^2$ projection. 
\begin{equation}
| - b_{\mathcal{P}}(\errv^n, \pi_h p^n)| = \tau| \adif(\phi_h^n, \pi_h p^n)| \!\leq\! C_\alpha\tau \vert \phi^n_h\vert_{\DG} \vert \pi_h p^n\vert_{\DG} \leq \frac{\tau}{8} \vert\phi_h^n\vert_{\mathrm{DG}}^2 + C\tau |p^n |_{H^1(\Omega)}^2.  \nonumber 
\end{equation}
Since $\div{\errv^n} \in M^{k-1}_h$, by the definition of $\pi_h p^n$ the first term in $b_{\mathcal{P}}(\errv^n, \pi_h p^n - p^n)$ is zero.  
Hence, by using trace estimate \eqref{eq:trace_estimate_continuous} and \eqref{eq:l2_proj_approximation}, we obtain 
\begin{align}
|b(\pi_h p^n - p^n, \errv^n) | &\leq  C (\| \pi_h p^n - p^n \| + h\|\gradh( \pi_h p^n - p^n) \|)\|\errv^n\|_{\DG}  \nonumber \\  
& \leq  \frac{C}{\mu_\mathrm{s}}  h^{2k} \vert p^n\vert_{H^{k}(\Omega)}^2 + \frac{K_\mathcal{D} \mu_\mathrm{s}}{32}  \|\errv^n \|_{\DG}^2.\nonumber
\end{align}
With the above bounds and expressions, \eqref{eq:conv_3} becomes: 
\begin{align}
&\frac{1}{2\tau} \left(\|\erru^n\|^2 - \|\erru^{n-1}\|^2 + \|\erru^n - \erru^{n-1}\|^2\right) + \frac{K_\strain\mu_\mathrm{s}}{2} \|\errv^n\|_{\DG}^2 + \frac{\tau}{8}\vert \phi_h^n \vert_{\mathrm{DG}}^2 \label{eq:conv_4}  \\ \quad &  + \frac{1}{2\sigma_\chi \mu_\mathrm{s}} \left(\|S_h^n \|^2 - \|S_h^{n-1}\|^2\right)+ \frac{\tau}{2}(A_1^n - A_2^n) +\frac{\tau}{4} \sum_{e\in\Gamma_h} \frac{\tilde{\sigma}}{h}\|[\phi_h^n - \phi_h^{n-1}]\|^2_{L^2(e)}  \nonumber \\ &+ \frac{\tau}{2} (\adif(\zeta_h^n, \zeta_h^n) - \adif(\zeta_h^{n-1},\zeta_h^{n-1}) )  \leq C h^{2k} |\bm{u}^n|_{H^{k+1}(\Omega)}^2  + \frac{K_\mathcal{D} \mu_\mathrm{s}}{16}  \|\errv^n\|_{\DG}^2   \nonumber \\ & + C \tau |p^n|^2_{H^1(\Omega)} + C h^{2k} \vert p^n\vert_{H^{k}(\Omega)}^2   +  \delta_{n,1}  \abs{b_{\mathcal{P}}(\erru^0, \phi_h^1)} + \sum_{i=3}^7 W_i. \nonumber 
\end{align} 
The terms $W_3 + W_4$ are bounded by Lemma~6.6 in the paper \cite{inspaper1}. We have:
\begin{align}
|W_3 + W_4|  \leq C \tau \int_{t^{n-1}}^{t^n} \|\partial_t \bm{u} \|^2 + C  h^{2k} + C \|\erru^{n-1}\|^2 +  \frac{K_\strain \mu_\mathrm{s}}{16}  \|\errv^n \|^2_{\DG}\label{eq:bound_W3W4}.  
\end{align}
The term $W_5$ is bounded with Cauchy--Schwarz's, Young's and Poincar\'e's inequalities, a Taylor expansion, and Lemma \ref{lemma:Pi_projection_error}. We obtain: 
\begin{align}
|W_5| &\leq C (\|(\partial_t \bm{u})^n - \delta_\tau \bm{u}^n\| + \|\delta_\tau \bm{\eta}_{\bm{u}}^n\| )\|\errv^n\|_{\DG} \label{eq:bound_W5}  \\ 
& \leq \frac{C}{\mu_\mathrm{s}}  \int_{t^{n-1}}^{t^n}\big( \tau \|\partial_{tt} \bm{u}\|^2 + \tau^{-1} h^{2k} |\partial_t \bm{u} |^2_{H^{k}(\Omega)}\big) + \frac{K_\mathcal{D} \mu_\mathrm{s}}{32} \|\errv^n\|_{\DG}^2. \nonumber 
\end{align}  
It remains to handle $W_6 + W_7$. Using \eqref{eq:CHNS:DG_interface} and following a similar approach to \cite{LiuRiviere2018numericalCHNS}, we  write: 
\begin{multline}
W_6 + W_7  = \aadv(c_h^{n-1}, \errv^n , \xi_\mu^n) + \aadv(c_h^{n-1}, \errv^n, \eta_\mu^n)  \\ 
  - \aadv(c^n - c^{n-1},\errv^n, \mu^n) + \aadv(\xi_c^{n-1}, \errv^n, \mu^n) + \aadv(\eta_c^{n-1}, \errv^n, \mu^n)  = \sum_{i=1}^5 B_i.\label{eq:bound_W67}
\end{multline}
For the term $B_1$, use \eqref{eq:bound_aadv_2}, Theorem~\ref{thm:CHNS:stability_bound1}, triangle inequality, and Young’s inequality, for any $\epsilon > 0$, we have 
\begin{align}
|B_1| &\leq C \|\errv^n\|^{1/2}\|\errv^n\|_{\DG}^{1/2}\vert\xi_\mu^n\vert_{\mathrm{DG}} 
\leq \frac{C}{\epsilon} (\|\errv^n - \erru^n\| + \|\erru^n\|)\|\errv^n\|_{\DG} + \epsilon  \vert\xi_\mu^n\vert_{\mathrm{DG}}^2.
\end{align}
 With \eqref{eq:fully_dis6}, \eqref{eq:def_b_lift_2} and \eqref{eq:lift_prop_g}, we have 
\begin{align}\label{eq:erruvint4}
\|\erru^n - \errv^n \| \leq  \|\tau \gradh{\phi_h^n} - \tau \bm{G}_h(\jump{\phi_h^n})\| \leq (1+\tilde{M}_{k})\tau \vert\phi_h^n\vert_{\mathrm{DG}}.
\end{align}
Then, by Young’s inequality, we obtain
\begin{align}
|B_1| \leq \frac{C}{\epsilon^2 \mu_\mathrm{s}} \Big(\|\erru^n\|^2 + \tau^2 (1+\tilde{M}_{k})^2 \vert\phi_h^n\vert_{\mathrm{DG}}^2\Big) 
+ \frac{K_\mathcal{D} \mu_\mathrm{s}}{32}  \|\errv^n\|^2_{\DG} + \epsilon \vert\xi_\mu^n\vert_{\mathrm{DG}}^2. \label{eq:bound_S1}
\end{align}
For the term $B_2$, apply Holder's, trace and  Poincar\'e's inequalities, \eqref{eq:L2boundchLp}, \eqref{eq:elliptic_projection_error}, and Young's inequality. 
\begin{align}
|B_2| \leq C \|c_h^{n-1}\|_{L^6(\Omega)} \|\errv^n\|_{L^3(\Omega)} \vert\eta_\mu^n\vert_{\mathrm{DG}} 
\leq \frac{C}{\mu_\mathrm{s}} h^{2k}|\mu|^2_{L^{\infty}(0,T;H^{k+1}(\Omega))} + \frac{K_\mathcal{D} \mu_\mathrm{s}}{32} \|\errv^n\|^2_{\DG}. \label{eq:bound_S2}  
\end{align}
The terms $B_3$, $B_4$, and $B_5$ simplify since the jumps of $\mu^n$ are zero. With Holder's and Poincar\'e's inequalities, we have 
\begin{multline}
|B_3| + |B_4| +|B_5| \leq (\|c^n  - c^{n-1}\| + \|\xi_c^{n-1}\| + \|\eta_c^{n-1}\|)  \|\errv^n\|_{L^6(\Omega)}|\mu^n|_{W^{1,3}(\Omega)}  \\ 
 \leq \frac{K_\mathcal{D} \mu_\mathrm{s}}{32}  \|\errv^n\|^2_{\DG} + \frac{C}{ \mu_\mathrm{s}} \Big(\vert\xi_c^{n-1}\vert_{\mathrm{DG}}^2 + h^{2k} + \tau \int_{t^{n-1}}^{t^n} \|\partial_t c\|^2\Big). \label{eq:bound_S345}
\end{multline}
Substituting bounds \eqref{eq:bound_W3W4}-\eqref{eq:bound_S1} into \eqref{eq:conv_4} yield the result.
\end{proof} 
With the above intermediate results, we are now ready to provide the proof of \eqref{eq:error_estimate_theorem} in 
the following section.
\subsection{Proof of \eqref{eq:error_estimate_theorem} in Theorem \ref{thm:conv_estimate_1}} \label{sec:conv_estimate_1}
\begin{proof}
As mentioned in the outline, the proof is based on an induction argument. We suppose that \eqref{eq:induction_hyp_linf} holds and we show that \eqref{eq:error_estimate_theorem} holds at time step $m$. 
We multiply \eqref{eq:err_eq_pressure_corretion} by $\tau$, choose $\epsilon =  K_\alpha^3/(32C_\mathrm{J}^2)$, 
denote $\alpha = 1/(K_\strain \epsilon^2 \mu_\mathrm{s})$, and substitute bound \eqref{eq:error_xi_mu} into \eqref{eq:err_eq_pressure_corretion}.  After adding the resulting bound to \eqref{eq:error_eq_xi_c}, we obtain
\begin{align}
&\frac{K_\alpha}{4}\tau \vert\mathcal{J}(\delta_\tau \xi_c^n)\vert_{\mathrm{DG}}^2 
+ \kappa\left( \adif(\xi_c^{n}, \xi_c^n) - \adif(\xi_c^{n-1},\xi_c^{n-1}) + K_\alpha \vert \xi_c^n - \xi_c^{n-1}\vert_{\mathrm{DG}}^2 \right) \nonumber \\
&+ \frac{1}{2}\left(\|\erru^n\|^2 - \|\erru^{n-1}\|^2\right) 
+ \frac{\tau}{2\sigma_\chi \mu_\mathrm{s}} \left(\|S_h^n \|^2 - \|S_h^{n-1}\|^2\right)  + \frac{K_\strain\mu_\mathrm{s}}{4} \tau\|\errv^n\|_{\DG}^2\nonumber \\ 
&+ \frac{\tau^2}{2} \left(\adif(\zeta_h^n, \zeta_h^n) - \adif(\zeta_h^{n-1},\zeta_h^{n-1}) + A^n_1 - A^n_2\right) 
+ \frac{1}{2}\|\erru^n - \erru^{n-1}\|^2 \nonumber \\
&+ \frac{1}{16}\tau^2\vert \phi_h^n\vert_{\mathrm{DG}}^2 \nonumber 
\leq C (\tau^2 + \tau h^{2k})
+ C \alpha\tau \big(\|\erru^n\|^2 + \tau^2  \vert\phi_h^n\vert_{\mathrm{DG}}^2\big) \\ 
&+ C  \tau ( |\xi_c^{n-1}|_{\DG}^2 + |\xi_c^{n}|_{\DG}^2 + \|\erru^{n-1}\|^2 ) + \delta_{n,1}\tau\abs{b_{\mathcal{P}}(\erru^0, \phi_h^1)} + \Lambda^n,  
\nonumber 
\end{align}
where 
\begin{align*}
\Lambda^n &= C \tau^2 \int_{t^{n-1}}^{t^n} \Big( \|\partial_t c\|^2_{H^1(\Omega)} + \|\partial_{tt}c \|^2 + \|\partial_t \bm{u}\|_{H^1(\Omega)}^2 + \|\partial_{tt} \bm{u}\|^2 \Big) \\
&+ C h^{2k} \int_{t^{n-1}}^{t^n} \Big(|\partial_t c|^2_{H^{k+1}(\Omega)} + |\partial_t \bm{u} |^2_{H^{k}(\Omega)}\Big). 
\end{align*}
We sum the above inequality from $n = 1$ to $n=m$. Recalling that $S_h^0 = 0$ and $\zeta_h^0 = 0$, we obtain
\begin{multline}\label{eq:induction_error_after_sum}
\frac{K_\alpha}{4} \tau \sum_{n=1}^{m} \vert\mathcal{J}(\delta_\tau \xi_c^n)\vert_{\mathrm{DG}}^2 
+ \kappa K_\alpha \sum_{n=1}^{m} \vert\xi_c^n - \xi_c^{n-1}\vert_{\mathrm{DG}}^2 
+ \frac{1}{2} \sum_{n=1}^{m} \|\erru^n - \erru^{n-1}\|^2 \\
+ (\kappa K_\alpha - C\tau)\vert \xi_c^m\vert_{\mathrm{DG}}^2
+ \Big(\frac{1}{2} - C \alpha \tau\Big) \|\erru^m\|^2
+ \Big(\frac{1}{16} - C\alpha\tau\Big)\tau^2 \sum_{n=1}^{m} \vert \phi_h^n\vert_{\mathrm{DG}}^2 \\
+ \frac{K_\strain\mu_\mathrm{s}}{4}\tau\sum_{n=1}^{m} \|\errv^n\|_{\DG}^2
+\frac{\tau^2}{2} \sum_{n=1}^{m} (A_1^n - A_2^n) 
\leq CT (\tau + h^{2k}) \\
+ C  \tau \sum_{n=0}^{m-1} \vert \xi_c^{n}\vert_{\mathrm{DG}}^2
+ C  \tau \sum_{n=0}^{m-1} \|\bm{\xi}_{\bm{u}}^{n}\|^2 + \tau\abs{b_{\mathcal{P}}(\erru^0, \phi_h^1)}
+ \kappa \adif(\xi_c^{0},\xi_c^{0})
+ \frac{1}{2}\|\erru^{0}\|^2.
\end{multline}
Noting that $\phi_h^0 = 0$ and using \eqref{eq:lift_prop_g} and the assumption that $\tilde{\sigma} \geq \tilde{M}_{k}^2$,  we have 
\begin{equation*}
\sum_{n=1}^m (A_1^n - A_2^n) =\sum_{e\in\Gamma_h} \frac{\tilde{\sigma}}{h}\|[\phi_h^m]\|^2_{L^2(e)} - \|\bm{G}_h[\phi_h^m]\|^2 \geq 0. 
\end{equation*}
Since $\vec{u}_h^0$ is the $L^2$ projection of the initial condition, 
 with \eqref{eq:CHNS:DG_pressure2} and the approximation properties,  we obtain
\begin{align*}
   |b_{\mathcal{P}}(\erru^0, \phi_h^1)| & \leq Ch^{k+1}\vert\phi_h^1\vert_{\mathrm{DG}}|\bm{u}^0|_{H^{k+1}(\Omega)} \leq \frac{\tau}{128} \vert\phi_h^1\vert_{\mathrm{DG}}^2 + C \tau^{-1}  h^{2k+2}|\bm{u}^0|_{H^{k+1}(\Omega)}^2. 
\end{align*}
We substitute the above bound into \eqref{eq:induction_error_after_sum} and recall that $\xi_c^0 = 0$ by definition. 
Assuming that the time step size $\tau$ is small enough (meaning that $\tau_0$ is small enough), 
we have
\begin{align}
&\frac{K_\alpha}{4} \tau \sum_{n=1}^{m} \vert\mathcal{J}(\delta_\tau \xi_c^n)\vert_{\mathrm{DG}}^2 
+ \kappa K_\alpha \sum_{n=1}^{m} \vert\xi_c^n - \xi_c^{n-1}\vert_{\mathrm{DG}}^2 
+ \frac{1}{2} \sum_{n=1}^{m} \|\erru^n - \erru^{n-1}\|^2 \\
& +\frac{\tau^2}{32} \sum_{n=1}^{m} \vert \phi_h^n\vert_{\mathrm{DG}}^2 
+ \frac{\kappa K_\alpha}{2} \vert \xi_c^m\vert_{\mathrm{DG}}^2
+ \frac{1}{4} \|\erru^m\|^2
+ \frac{K_\strain\mu_\mathrm{s}}{4}\tau\sum_{n=1}^{m} \|\errv^n\|_{\DG}^2 \nonumber \\ 
&\leq C T (\tau + h^{2k}) + 
C \tau \sum_{n=0}^{m-1} \vert \xi_c^{n}\vert_{\mathrm{DG}}^2
+  C  \tau \sum_{n=0}^{m-1} \|\bm{\xi}_{\bm{u}}^{n}\|^2. \nonumber
\end{align}
Use Gronwall's inequality, we obtain
\begin{align}
K_\alpha \tau \sum_{n=1}^{m} \vert\mathcal{J}(\delta_\tau \xi_c^n)\vert_{\mathrm{DG}}^2 
+ \kappa K_\alpha \sum_{n=1}^{m} \vert\xi_c^n - \xi_c^{n-1}\vert_{\mathrm{DG}}^2 
+ \sum_{n=1}^{m} \|\erru^n - \erru^{n-1}\|^2
+ \tau^2 \sum_{n=1}^{m} \vert \phi_h^n\vert_{\mathrm{DG}}^2 \label{eq:inter22} \\
+ \kappa K_\alpha \vert \xi_c^m\vert_{\mathrm{DG}}^2
+ \|\erru^m\|^2
+ K_\strain\mu_\mathrm{s}\tau\sum_{n=1}^{m} \|\errv^n\|_{\DG}^2 
\leq CT e^{C T}\, (\tau + h^{2k}). \nonumber
\end{align}
 Using triangle inequalities and approximation results, the bound above will yield the desired error estimate \eqref{eq:error_estimate_theorem} except for the bound on the first term $\sum_{n=1}^m \vert \mu_h^n -\mu^n\vert_{\mathrm{DG}}^2$.   
To obtain a bound on the error of the chemical potential, we multiply \eqref{eq:error_xi_mu} by $\tau$, sum the inequality from $n = 1$ to $n=m$, and use \eqref{eq:inter22}. We note that every generic constant $C$ in the proof is independent of the induction iteration index $m$.
\end{proof}
Finally, we complete the induction proof by verifying the induction hypothesis for $m+1$. 
\begin{lemma}\label{lem:prove_induction_hyp}
There exists a constant $C_\mathrm{err}$ independent of $h$ and $\tau$, such that  under the conditions $\tau \leq \gamma \leq C_\mathrm{err}^{-(1+\delta)/\delta}$ and $h \leq C_\mathrm{err}^{-(1+\delta)/\delta}$, the bound \eqref{eq:induction_hyp_linf} holds.
\end{lemma}
\begin{proof}
From the previous proof, it is straightforward to see that,
\begin{equation*}
\tau^2 \sum_{n = 0}^{(m+1)-1} \vert \phi_h^n \vert_{\mathrm{DG}}^{2} + \tau \sum_{n = 0}^{(m+1)-1} \|\errv^n\|_{\DG}^2 
\leq C_\mathrm{err} (\tau + h^{2k})
\leq (C_\mathrm{err}\tau^{\frac{\delta}{1+\delta}}  )\tau^{\frac{1}{1+\delta}} 
+ (C_\mathrm{err} h^{ \frac{\delta}{1+\delta} }) h^{\frac{2+\delta}{1+\delta} }.
\end{equation*}
Therefore, the induction hypothesis holds for sufficiently small time step length and mesh size, namely 
$C_\mathrm{err}\tau^{\frac{\delta}{1+\delta}} \leq 1$ and $C_\mathrm{err} h^{\frac{\delta}{1+\delta}} \leq 1$.
\end{proof}

%% file: Content/error_analysis_improved_const.tex
\subsection{Improved Estimate} 
In this section, we use duality arguments to obtain estimate \eqref{eq:improved_error_estimate_theorem} in Theorem~\ref{thm:conv_estimate_1}. 
We start by deriving bounds for $\|\xi_c^n\|$ and $\|\xi_\mu^n - \overline{\xi_\mu^n}\|$. 
Note that for any $\chi_h \in M_h^k$ and $\phi_h \in M_{h0}^k$, we have 
\begin{equation}\label{eq:relation_J_nonzero}
a_\mathrm{diff}(\mathcal{J}(\chi_h - \overline{\chi_h}), \phi_h) = (\chi_h - \overline{\chi_h}, \phi_h) = (\chi_h, \phi_h). 
\end{equation}
\begin{lemma} \label{lemma:bound_l2norm_c}
There exists a constant $C$  independent of $h$ and $\tau$ such that for all $m \geq 1$, we have  
\begin{align}
\kappa\|\xi_c^m\|^2 
+ \tau \sum_{n=1}^m \|\delta_\tau J(\xi_c^n) \|^2
\leq C (\tau^2 + h^{2k+2}) + C\tau \sum_{n=0}^m \|\xi_c^n\|^2 + C\tau \sum_{n=0}^{m-1} \|\erru^{n}\|^2.
\end{align}
\end{lemma}
\begin{proof}
For readibility, let us denote $\hat{c}_h^n = \mathcal{J}(\xi_c^n)$. 
We note  that the linear operator $\mathcal{J}$ is commutable with operator $\delta_\tau$. Recalling that  $\delta_\tau \xi_c^n \in M_{h0}^k$ and $\delta_\tau \hat{c}_h^n \in M_{h0}^k$, we have 
\begin{align} 
(\delta_\tau \xi_c^n, \mathcal{J}(\delta_\tau \hat{c}_h^n)) &= \adif(\delta_\tau \hat{c}_h^n, \mathcal{J}(\delta_\tau \hat{c}_h^n)) = \|\delta_\tau \hat{c}_h^n\|^2, \\ 
\adif(\xi_\mu^n, \mathcal{J}(\delta_\tau \hat{c}_h^n)) 
&= \adif(\xi_\mu^n -\overline{\xi_\mu^n}, \mathcal{J}(\delta_\tau \hat{c}_h^n)) 
= (\delta_\tau \hat{c}_h^n, \xi_\mu^n -\overline{\xi_\mu^n})
= (\delta_\tau \hat{c}_h^n, \xi_\mu^n),\\
\adif(\xi_c^n, \delta_\tau \hat{c}_h^n) &= (\xi_c^n, \delta_\tau \xi_c^n). \label{eq:hatcp3}
\end{align}
Choose $\chi_h = \mathcal{J}(\delta_\tau \hat{c}_h^n)$ in \eqref{eq:first_err_eq} and 
 $\varphi_h = \delta_\tau \hat{c}_h^n$ in \eqref{eq:second_err_eq}, and add the resulting equations. 
With the identities above, we obtain
 \begin{multline}
  \|\delta_\tau \hat{c}_h^n\|^2 +  \frac{\kappa}{2\tau}\left(\|\xi_c^n\|^2 - \|\xi_c^{n-1}\|^2 + \|\xi_c^n - \xi_c^{n-1}\|^2\right) = \\(-\delta_\tau c^n + (\partial_t c)^n - \delta_\tau \eta_c^n, \mathcal{J}(\delta_\tau \hat{c}_h^n)) - \aadv(c_h^{n-1}, \bm{u}_h^{n-1}, \mathcal{J}(\delta_\tau \hat{c}_h^n)) + \aadv(c^{n}, \bm{u}^n, \mathcal{J}(\delta_\tau \hat{c}_h^n)) \\ + (\eta_\mu^n, \delta_\tau \hat{c}_h^n)  + (\Phi_{+}\,\!'(c^n) - \Phi_{+}\,\!'(c_h^n),\delta_\tau \hat{c}_h^n) + (\Phi_{-}\,\!'(c^{n}) - \Phi_{-}\,\!'(c_h^{n-1}), \delta_\tau \hat{c}_h^n) = \sum_{i=1}^{6} W_i.   \label{eq:l2bound_xic_semi}
 \end{multline} 
With \eqref{eq:poincare_ineq},  \eqref{eq:coercivity_adiff} and \eqref{eq:def_of_J}, it follows that 
\begin{equation}
    \vert\mathcal{J}(\delta_\tau \hat{c}_h^n)\vert_{\mathrm{DG}} \leq \frac{C_\mathrm{P}}{K_\alpha} \|\delta_\tau \hat{c}_h^n\|.  \label{eq:dg_bd_J}
\end{equation}
Hence, the terms $W_1$ and $W_4$ are bounded with Cauchy--Schwarz's inequality, Taylor's expansions, Poincar\'e's inequality, and the optimal bounds on $\|\eta_\mu^n\| $ and on $\|\partial_t \eta_c^n\|$. We have
\begin{multline}
|W_1| + |W_4| \leq C \|-\delta_\tau c^n + (\partial_t c)^n - \delta_\tau \eta_c^n\| |\mathcal{J}(\delta_\tau \hat{c}_h^n)|_{\DG} 
+ C \|\eta_\mu^n\| \|\delta_\tau \hat{c}_h^n\|  \\ \leq  \frac{1}{8} \|\delta_\tau \hat{c}_h^n\|^2 + C\int_{t^{n-1}}^{t^n} \Big(\tau \| \partial_{tt} c\|^2 + \tau^{-1} h^{2k+2}|\partial_t c|^2_{H^{k+1}(\Omega)}\Big)   + Ch^{2k + 2 } \|\mu\|^2_{L^\infty(0,T;H^{k+1}(\Omega))}. \nonumber 
\end{multline}
To handle $W_2+W_3$, we use Lemma \ref{lemma:bound_aadv_terms}, with  $\epsilon = K_\alpha^2/(40 C_\mathrm{P}^2)$, and \eqref{eq:dg_bd_J}. 
\begin{align}
|W_2 + W_3| \leq \frac{1}{8} \|\delta_\tau \hat{c}_h^n\|^2 + C \tau \int_{t^{n-1}}^{t^n} \Big(\|\partial_t c\|^2_{H^1(\Omega)} + \|\partial_t \bm{u}\|^2_{H^1(\Omega)}\Big) \nonumber \\ 
+\, Ch^{2k+2} + C \left( \|\xi_c^{n-1}\|^2  + \|\erru^{n-1}\|^2 \right).\nonumber
\end{align}
We bound $W_5$ and $W_6$ by applying Cauchy--Schwarz's inequality and \Cref{lemma:L_infty_bound}. We have
\begin{multline}
|W_5| + |W_6| \leq C\left(\|c_h^n - c^n\|\|(c^n)^2 + c^nc_h^n + (c_h^n)^2\|_{L^\infty(\Omega)}  + \|c^n - c_h^{n-1}\| \right) \|\delta_\tau \hat{c}_h^n\|  \\ 
 \leq \frac{1}{8}\|\delta_\tau \hat{c}_h^n\|^2 + C\Big(\|\xi_c^n\|^2 + \|\xi_c^{n-1}\|^2 + h^{2k+2}\|c\|_{L^{\infty}(0,T;H^{k+1}(\Omega))}^2 + \tau \int_{t^{n-1}}^{t^n} \|\partial_t c\|^2 \Big). \nonumber
\end{multline}
We substitute the above bounds into \eqref{eq:l2bound_xic_semi}, multiply by $2\tau$, sum from $n=1$ to $n=m$, and conclude the proof by noticing $\xi_c^{0}=0$. 
\end{proof}
We now show an $L^2$ bound on $\xi_\mu^n - \overline{\xi_\mu^n}$. 
\begin{lemma}\label{lemma:bound_l2norm_mu} 
There exists a constant $C$ independent of $h$ and $\tau$ such that for all $m \geq 1$, we have 
\begin{equation}
\tau\sum_{n=1}^{m}\|\xi_\mu^n - \overline{\xi_\mu^n}\|^2 
\leq C(\tau^2 + h^{2k+2}) 
+ C\tau\sum_{n=0}^{m} \|\xi_c^n\|^2 + C\tau\sum_{n=0}^{m-1} \|\bm{\xi}_{\bm{u}}^{n}\|^2.
\end{equation}
\end{lemma}
\begin{proof}
Observe that  (with \eqref{eq:relation_J_nonzero} for the second equation):
\begin{align*}
\adif(\xi_\mu^n, \hat{\mu}_h^n) = \adif(\xi_\mu^n - \overline{\xi_\mu^n}, \hat{\mu}_h^n) = \|\xi_\mu^n - \overline{\xi_\mu^n}\|^2, \\
(\xi_\mu^n,  \mathcal{J}(\delta_\tau \xi_c^n)) = \adif(\mathcal{J}(\xi_\mu^n - \overline{\xi_\mu^n}), \mathcal{J}(\delta_\tau \xi_c^n)) = (\delta_\tau \xi_c^n, \mathcal{J}(\xi_\mu^n - \overline{\xi_\mu^n})).
\end{align*}
Choose $\chi_h =  \mathcal{J}(\xi_\mu^n - \overline{\xi_\mu^n})$ in \eqref{eq:first_err_eq} 
and $\varphi_h = \mathcal{J}(\delta_\tau \xi_c^n) = \delta_\tau \hat{c}_h^n$ in \eqref{eq:second_err_eq}. Adding the resulting
equations, using the identities above and \eqref{eq:hatcp3}, we obtain
\begin{multline}
\|\xi_\mu^n - \overline{\xi_\mu^n}\|^2 +  \frac{\kappa}{2\tau}\left(\|\xi_c^n\|^2 - \|\xi_c^{n-1}\|^2 + \|\xi_c^n - \xi_c^{n-1}\|^2\right) =  \\ (-\delta_\tau c^n + (\partial_t c)^n - \delta_\tau \eta_c^n,\mathcal{J} (\xi_\mu^n-\overline{\xi_\mu^n})) - \aadv(c_h^{n-1}\!\!, \bm{u}_h^{n-1}\!\!, \mathcal{J}(\xi_\mu^n-\overline{\xi_\mu^n})) + \aadv(c^{n}\!, \bm{u}^n\!, \mathcal{J}(\xi_\mu^n-\overline{\xi_\mu^n})) \\ + (\eta_\mu^n, \delta_\tau \hat{c}_h^n)  + (\Phi_{+}\,\!'(c^n) - \Phi_{+}\,\!'(c_h^n),\delta_\tau \hat{c}_h^n) + (\Phi_{-}\,\!'(c^{n}) - \Phi_{-}\,\!'(c_h^{n-1}), \delta_\tau \hat{c}_h^n) = \sum_{i=1}^{6} Z_i.  
\end{multline}
Comparing with \eqref{eq:l2bound_xic_semi}, we remark that 
$Z_i = W_i$ for $i\in\{4,5,6\}$ and we will use the same bounds in the proof of \Cref{lemma:bound_l2norm_c}.
For the other terms,  $Z_1, Z_2, Z_3$, we use similar bounds as for the terms $W_1, W_2, W_3$ where $\delta_\tau \hat{c}_h^n$ is replaced by $\xi_\mu^n - \overline{\xi_\mu^n}$. Thus, we obtain
\begin{multline}
\frac{1}{4}\|\xi_\mu^n - \overline{\xi_\mu^n}\|^2 +  \frac{\kappa}{2\tau}\left(\|\xi_c^n\|^2 - \|\xi_c^{n-1}\|^2 + \|\xi_c^n - \xi_c^{n-1}\|^2\right) \\ 
\leq Ch^{2k + 2} + C(\|\delta_\tau \hat{c}_h^n\|^2 +\|\xi_c^n\|^2 + \|\xi_c^{n-1}\|^2 + \|\bm{\xi}_{\bm{u}}^{n-1}\|^2) \\ 
+ C\tau \int_{t^{n-1}}^{t^n}\Big( \|\partial_t c\|_{H^{1}(\Omega)}^2 + \| \partial_{tt} c\|^2 +\|\partial_t \bm{u}\|_{H^{1}(\Omega)}^2 \Big) 
+ C \int_{t^{n-1}}^{t^n}\tau^{-1} h^{2k+2}|\partial_t c|^2_{H^{k+1}(\Omega)}.
\end{multline} 
Multiplying the above bound by $4\tau$, summing from $n = 1$ to $n = m$, and applying the results of Lemma~\ref{lemma:bound_l2norm_c}, we conclude the proof. 
\end{proof}
\subsubsection{Proof of the improved estimate \eqref{eq:improved_error_estimate_theorem}} \label{sec:proof_improved_err}
We now proceed to obtain bounds on $\|\erru\|$ and $\|\errv\|$
 via constructing a  dual Stokes problem and its dG discretization and we follow the argument
in \cite{masri2021improved}. 
To this end, define the error function 
\begin{alignat*}{2}
\bm{\chi}_{\bm{u}}(t) & = \bm{u}_h^n - \bm{u}(t), \quad &&\forall t^{n-1} < t\leq t^n, \quad  \forall n\geq 1, \quad
\bm{\chi}_{\bm{u}} (0) = \bm{u}_h^0 - \bm{u}^0.  
\end{alignat*}
Further, for $t\geq 0$, define $(\bm{U}(t),P(t)) \in  H^1_0(\Omega)^d \times L^2_0(\Omega)$ 
the weak solution  of the dual Stokes problem:  
 \begin{alignat}{2}
 -\laplace{\bm{U}(t)}  + \grad{P(t)} &= \bm{\chi}_{\bm{u}}(t) && \quad \mathrm{in}\,\, \Omega,  \label{eq:aux_pb_1}\\ 
 \div{\bm{U}(t)} & = 0 && \quad \mathrm{in } \,\, \Omega, \label{eq:aux_pb_2} \\
 \bm{U}(t) & = 0 && \quad \mathrm{on }\,\, \partial \Omega, \label{eq:aux_pb_3}
 \end{alignat} 
and $(\bm{U}_h(t), P_h(t)) \in \mathbf{X}^{k}_h \times M^{k-1}_{h0}$ its dG solution:
\begin{alignat}{2}
a_{\mathcal{D}}(\vec{U}_h(t), \vec{\theta}_h) - b_{\mathcal{P}}(\bm{\theta}_h, P_h(t)) &=(\vec{\chi}_{\vec{u}}(t), \vec{\theta}_h) && \quad \forall \vec{\theta}_h \in \mathbf{X}^{k}_h, \label{eq:aux_1_dg} \\ 
b_{\mathcal{P}}(\bm{U}_h(t), q_h) &= 0 && \quad \forall q_h \in M^{k-1}_{h0}. \label{eq:aux_2_dg}
\end{alignat}
Existence and uniqueness of $(\bm{U}_h(t), P_h(t))$ for all $t>0$ is a consequence of the coercivity of $a_{\mathcal{D}}$ and the inf-sup condition for the pair of spaces $(\mathbf{X}^{k}_h,M^{k-1}_{h0})$ \cite{riviere2008}. 
We take $\bm{\theta}_h = \bm{U}_h^n = \bm{U}_h(t^n)$ in \eqref{eq:third_err_eq}. We begin with handling the last two terms. Namely, we write 
\begin{align*}
    & b_{\mathcal{I}}(c_h^{n-1},\mu_h^n, \bm{U}_h^n) - b_{\mathcal{I}}(c^n, \mu^n, \bm{U}_h^n) = \aadv(c_h^{n-1}, \bm{U}_h^n, \xi_\mu^n) + \aadv(c_h^{n-1}, \bm{U}_h^n, \eta_\mu^n) \nonumber \\ 
    & - \aadv(c^n - c^{n-1},\bm{U}_h^n, \mu^n) + \aadv(\xi_c^{n-1}, \bm{U}_h^n, \mu^n) + \aadv(\eta_c^{n-1}, \bm{U}_h^n, \mu^n)  = \sum_{i=1}^5 K_i.
\end{align*}
We rewrite the term $K_1$ by applying 
Green's theorem. 
\begin{align*}
K_1 = &\sum_{E \in \mathcal{T}_h} \int_{E} \left(\grad c_h^{n-1} \cdot \bm{U}_h^n + c_h^{n-1} \div \bm{U}_h^n\right) (\xi_\mu^n - \overline{\xi_\mu^n}) \\&  - \sum_{e \in \Gamma_h \cup \partial \Omega} \int_{e} \{c_h^{n-1} (\xi_\mu^n - \overline{\xi_\mu^n})\}[\bm{U}_h^n \cdot \bm{n}_e]  - \sum_{e\in\Gamma_h} \int_{e} [c_h^{n-1}]\{\bm{U}_h^n \cdot \bm{n}_e\} \{\xi_\mu^n - \overline{\xi_\mu^n}\}. 
\end{align*}
With Holder's and Poincar\'e's inequalities and trace estimates, we have 
\begin{align*}
|K_1| \leq\, &C \left(\|\grad_h c_h^{n-1}\|_{L^3(\Omega)} + \|c_h^{n-1}\|_{L^{\infty}(\Omega)} \right)\|\bm{U}_h^n\|_{\DG}\|\xi_\mu^n - \overline{\xi_\mu^n}\| \\
+\, &C|c_h^{n-1} - c^{n-1}|_{\DG} \|\bm{U}_h^n\|_{L^{\infty}(\Omega)} \|\xi_\mu^n - \overline{\xi_\mu^n}\|.
\end{align*}
Note that with \eqref{eq:error_estimate_theorem} in Theorem~\ref{thm:conv_estimate_1}, inverse estimate \eqref{eq:inverse_estmate_lp}, Poincar\'e's inequality \eqref{eq:poincare_ineq}, and \eqref{eq:cfl_cond}, we obtain  
\begin{multline}
\label{eq:boundonchcn}
|c_h^{n-1} - c^{n-1}|_{\DG} \|\bm{U}_h^n\|_{L^{\infty}(\Omega)} \\ 
\leq \Big(\frac{C_\mathrm{err}}{\kappa K_\alpha}\Big)^{1/2} (\tau + h^{2k})^{1/2} h^{-d/6} \|\bm{U}_h^n\|_{L^6(\Omega)} 
\leq C \|\bm{U}_h^n\|_{\DG}.
\end{multline}
We remark that for $h \leq h_0 \leq  \min(1, C_{\mathrm{err}}^{-1}, C_\mathrm{err}^{-1/\delta})$, the constant $C$ in \eqref{eq:boundonchcn} is independent of $C_{\mathrm{err}}$.  

The term $K_2$ is handled similarly. That is, we use the same integration by parts formula. Here, we use the trace estimates for functions in $H^2(\mathcal{T}_h)$ and approximation properties.  We have 
\begin{equation*}
|K_2| \leq Ch^{k+1}|\mu^n|_{H^{k+1}(\Omega)}\|\bm{U}_h^n\|_{\DG}. 
\end{equation*}
From Lemma \ref{lemma:L_infty_bound} and Young's inequality, we have for $\epsilon  > 0$, 
\begin{equation*}
|K_1| + |K_2| \leq \epsilon \mu_{\mathrm{s}}\|\xi_\mu^n - \overline{\xi_\mu^n}\|^2 + C \Big(\frac{T+1}{\epsilon \mu_\mathrm{s}} + 1\Big) \|\bm{U}_h^n\|_{\DG}^2 + C h^{2k+2}|\mu |^2_{L^{\infty}(0,T;H^{k+1}(\Omega))}.  
\end{equation*}
For $K_3$, $K_4$, and $K_5$, the jumps of $\mu^n$ evaluate to zero. 
Thus, with Holder's inequality, we have 
\begin{align} 
& |K_3| + |K_4| + |K_5|  \leq C (\|c^{n} - c^{n-1}\| + \|\xi_c^{n-1}\| + \| \eta_c^{n-1}\| )\|\bm{U}_h^n\|_{\DG} |\mu^n|_{W^{1,3}(\Omega)} \nonumber \\ 
& \leq \epsilon \mu_\mathrm{s} \|\xi_c^{n-1}\|^2 + C\Big(\frac{1}{\epsilon \mu_\mathrm{s}} + 1\Big) \|\bm{U}_h^n\|_{\DG}^2
+ Ch^{2k+2} |c|^2_{L^{\infty}(0,T;H^{k+1}(\Omega))} + C\tau \int_{t^{n-1} }^{t^n} \|\partial_t c\|^2. \nonumber  
\end{align}
We refer to the proof of Theorem~1 in \cite{masri2021improved} to handle all the remaining terms in \eqref{eq:third_err_eq} with $\bm{\theta}_h = \bm{U}_h^n$ since the same arguments can be  used here.  For completeness, most details are given in Appendix~\ref{sec:app3}.
We have
\begin{align}
\label{eq:boundUhpaper}
& K_\mathcal{D} \|\vec{U}_h^m\|^2_{\DG}
+  \mu_\mathrm{s} \tau \sum_{n=1}^m \| \vec{\xi}_{\vec{u}}^n\|^2  \leq 
C(\tau^2+h^{2k+2} + \tau h^2)\\
&+\, C \tau \Big( \frac{1}{\mu_\mathrm{s}} + 1 \Big) \sum_{n=1}^m \|\vec{U}_h^n\|_{\DG}^2
+ \tau \sum_{n=1}^m \big(b_{\mathcal{I}}(c_h^{n-1},\mu_h^n, \vec{U}^n_h) - b_{\mathcal{I}}(c^n, \mu^n, \vec{U}^n_h)\big). \nonumber
\end{align}
With the bounds on $K_i$'s, we obtain
\begin{align*}
 & K_\mathcal{D} \|\vec{U}_h^m\|^2_{\DG}
+  \mu_\mathrm{s} \tau \sum_{n=1}^m \| \vec{\xi}_{\vec{u}}^n\|^2  \leq 
C (\tau^2+h^{2k+2} + \tau h^2)\\
&+\, C \Big( \frac{T+1}{\epsilon\mu_\mathrm{s}} + 1 \Big) \tau \sum_{n=1}^m \|\vec{U}_h^n\|_{\DG}^2
+\epsilon \mu_s  \tau \sum_{n=1}^m \|\xi_c^{n-1}\|^2
+\epsilon \mu_s  \tau \sum_{n=1}^m \|\xi_\mu^n - \overline{\xi_\mu^n}\|^2.
\end{align*}
We next use Lemma~\ref{lemma:bound_l2norm_mu}  to bound the last bound above.  
We then multiply the bound in Lemma~\ref{lemma:bound_l2norm_c} by $\epsilon \mu_\mathrm{s}$ and add to the resulting
inequality. 
\begin{align*}
 & K_\mathcal{D} \|\vec{U}_h^m\|^2_{\DG} + \epsilon \mu_s \kappa \Vert \xi_c^m\Vert^2 
+  \mu_\mathrm{s} (1-\hat{C} \epsilon) \tau \sum_{n=1}^m \| \vec{\xi}_{\vec{u}}^n\|^2  \leq 
C ((1+\epsilon)(\tau^2+h^{2k+2}) + \tau h^2)\\
&+\, C \Big( \frac{T+1}{\epsilon\mu_\mathrm{s}} + 1 \Big) \tau \sum_{n=1}^m \|\vec{U}_h^n\|_{\DG}^2
+ C \epsilon \mu_s  \tau \sum_{n=0}^m \|\xi_c^{n}\|^2
+ C \tau \epsilon\mu_s \Vert \xi_{\vec{u}}^0\Vert^2.
\end{align*}
We choose $\epsilon = 1/(2\hat{C})$ and note that $\|\bm{\xi}_{\bm{u}}^{0}\|^2 \leq Ch^{2k+2}$.
\begin{align*}
 & K_\mathcal{D} \|\vec{U}_h^m\|^2_{\DG} + \frac{\mu_s\kappa }{2\hat{C}}   \Vert \xi_c^m\Vert^2 
+  \frac{\mu_\mathrm{s}}{2}   \tau \sum_{n=1}^m \| \vec{\xi}_{\vec{u}}^n\|^2  \leq 
C (\tau^2+h^{2k+2} + \tau h^2)\\
&+\, C \Big( \frac{T+1}{\mu_\mathrm{s}} + 1 \Big) \tau \sum_{n=1}^m \|\vec{U}_h^n\|_{\DG}^2
+ C  \mu_s  \tau \sum_{n=0}^m \|\xi_c^{n}\|^2.
\end{align*}
Therefore, choosing $\tau$ small enough, 
and applying Gronwall's inequality, we obtain
\begin{equation}
\mu_\mathrm{s} \tau \sum_{n=1}^m \| \erru^n \|^2 
\leq C (\tau^2 + h^{2k+2} + \tau h^2). \label{eq:bounding_little_l2_erru} 
\end{equation}
To obtain a bound on $\|\errv^n\|$, we use \eqref{eq:erruvint4}, \eqref{eq:bounding_little_l2_erru} and \eqref{eq:error_estimate_theorem}: 
\begin{align}
\mu_\mathrm{s} \tau \sum_{n=1}^m \| \errv^n \|^2 
\leq \mu_\mathrm{s} \tau \sum_{n=1}^m \| \erru^n \|^2 + \mu_\mathrm{s} \tau^3 \sum_{n=1}^{m} |\phi_h^n|_{\DG}^2 
\leq C (\tau^2 + h^{2k+2} + \tau h^2).
\end{align}
Use \eqref{eq:bounding_little_l2_erru} in Lemma~\ref{lemma:bound_l2norm_c} and apply Gronwall's inequality to have
\begin{equation}
\kappa \|\xi_c^m\|^2 
\leq C (\tau^2 + h^{2k+2} + \tau h^2). 
\end{equation}
This bound with Lemma~\ref{lemma:bound_l2norm_mu} and \eqref{eq:bounding_little_l2_erru} yields
\begin{equation}
\mu_\mathrm{s} \tau \sum_{n=1}^m \|\xi_\mu^n - \overline{\xi_\mu^n}\|^2 
\leq C (\tau^2 + h^{2k+2} + \tau h^2). 
\end{equation}
To obtain a bound on $\|\xi_\mu^n\|$, it suffices to derive a bound on $\overline{\xi_\mu^n}$. Let $\varphi_h = 1$ in \eqref{eq:second_err_eq}.  Since the averages of $\eta_\mu^n$ and of $(c^n - c_h^{n-1})$ are zero,  we obtain: 
\begin{align*} 
| \int_\Omega \xi_\mu^n | 
\leq \|(c^n)^2 + c^n c_h^n + (c_h^n)^2\| \|c^n - c_h^n \| 
\leq C (\tau^2 + h^{2k+2} + \tau h^2)^{\frac{1}{2}}. 
\end{align*} 
Finally, we conclude the proof of \eqref{eq:improved_error_estimate_theorem} by using the bounds above, triangle inequality and approximation error bounds.

%% file: Content/numerical_experiments.tex
\section{Numerical Experiments}\label{sec:numerical_experimants}
In this section, our numerical method is first verified via convergence rate tests. Next, the spinodal decomposition simulation shows the proposed algorithm enjoys mass conservation and energy dissipation properties.
For all the numerial results, we choose $\sigma_\chi = 1/12$.

\subsection{Convergence study}\label{sec:num_exp:convergence}
We utilize the manufactured solution method for convergence study on the unit cube $\Omega=(0,1)^3$. The simulation end time is $T=1$. 
For convenience, we select parameters $\mu_\mathrm{s} = 1$ and $\kappa = 1$. The prescribed solution is defined as follows \cite{liu2022pressure}.
\begin{align*}
c(t,x,y,z)   &= \exp{(-t)}\sin{(2\pi x)}\sin{(2\pi y)}\sin{(2\pi z)},\\
v_1(t,x,y,z) &= -\exp{(-t+x)}\sin{(y+z)}-\exp{(-t+z)}\cos{(x+y)},\\
v_2(t,x,y,z) &= -\exp{(-t+y)}\sin{(x+z)}-\exp{(-t+x)}\cos{(y+z)},\\
v_3(t,x,y,z) &= -\exp{(-t+z)}\sin{(x+y)}-\exp{(-t+y)}\cos{(x+z)},\\
p(t,x,y,z)   &= -\exp{(-2t)}\Big(\exp(x+z)\sin{(y+z)}\cos{(x+y)}\\
&\hspace{1.95cm}+\exp(x+y)\sin{(x+z)}\cos{(y+z)}\\
&\hspace{1.95cm}+\exp(y+z)\sin{(x+y)}\cos{(x+z)}\\
&\hspace{1.95cm}+\frac{1}{2}\exp{(2x)} + \frac{1}{2}\exp{(2y)} + \frac{1}{2}\exp{(2z)}
-\bar{p}^0\Big),
\end{align*}
where, $\bar{p}^0=7.63958172715414$, which guarantees zero average pressure over $\Omega$ (up to machine precision). 
Here in above, the order parameter field is taken from \cite{frank2018finite}. The chemical potential is an intermediate variable, which value is derived by the order parameter. The velocity and pressure fields are borrowed from the Beltrami flow \cite{liu2019interior}, which enjoys the property that the nonlinear convection is balanced by the pressure gradient and the velocity is parallel to vorticity.
In addition, the initial conditions and Dirichlet boundary condition for velocity are imposed by the manufactured solutions. For order parameter and chemical potential, we apply Neumann boundary condition.
\par
We obtain spatial rates of convergence by computing the solutions on a sequence of uniformly refined meshes (see the second column of \Cref{tab:space_rate} for $h$). 
We fix $\tau = 1/2^{10}$ for $k=1$ ($\mathbbm{P}1$--$\mathbbm{P}1$--$\mathbbm{P}1$--$\mathbbm{P}0$ scheme); 
we fix $\tau = 1/2^{13}$ for $k=2$ ($\mathbbm{P}2$--$\mathbbm{P}2$--$\mathbbm{P}2$--$\mathbbm{P}1$ scheme); 
and we fix $\tau = 1/2^{15}$ for $k=3$ ($\mathbbm{P}3$--$\mathbbm{P}3$--$\mathbbm{P}3$--$\mathbbm{P}2$ scheme) to guarantee the spatial error dominates. 
We use SIPG and add subscript to distinguish the penalty parameter $\tilde{\sigma}$ in form $a_\mathrm{diff}$, namely in \eqref{eq:fully_dis1}-\eqref{eq:fully_dis2} by $\tilde{\sigma}_\mathrm{CH}$ and in \eqref{eq:fully_dis4} by $\tilde{\sigma}_\mathrm{ellip}$, respectively. Recall the penalty parameter in form $a_\mathcal{D}$ is $\sigma$.
For $\mathbbm{P}1$--$\mathbbm{P}1$--$\mathbbm{P}1$--$\mathbbm{P}0$ scheme, we set $\tilde{\sigma}_\mathrm{CH} = 2$, $\tilde{\sigma}_\mathrm{ellip} = 1$, $\sigma = 8$ on $\Gamma_h$ and $\sigma = 16$ on $\partial\Omega$. 
For $\mathbbm{P}2$--$\mathbbm{P}2$--$\mathbbm{P}2$--$\mathbbm{P}1$ scheme, we set $\tilde{\sigma}_\mathrm{CH} = 4$, $\tilde{\sigma}_\mathrm{ellip} = 2$, $\sigma = 64$ on $\Gamma_h$ and $\sigma = 128$ on $\partial\Omega$. 
For $\mathbbm{P}3$--$\mathbbm{P}3$--$\mathbbm{P}3$--$\mathbbm{P}2$ scheme, we set $\tilde{\sigma}_\mathrm{CH} = 8$, $\tilde{\sigma}_\mathrm{ellip} = 8$, $\sigma = 128$ on $\Gamma_h$ and $\sigma = 256$ on $\partial\Omega$. 
If $\mathtt{err}_{h}$ is the error on a mesh with resolution $h$, then the rate is defined by $\ln(\mathtt{err}_{h}/\mathtt{err}_{h/2})/\ln(2)$. We show the errors and rates in \Cref{tab:space_rate}.  The convergence rates are optimal.
\begin{table}[ht!]
\centering
\begin{tabularx}{\linewidth}{@{~}c@{~}|@{~}c@{~}|C@{\hspace{-0.6666em}}C@{\hspace{-0.6666em}}|C@{\hspace{-0.6666em}}C@{\hspace{-0.6666em}}|C@{\hspace{-0.6666em}}C@{\hspace{-0.6666em}}}
\toprule
$k$ & $h$ & $\|c_h^{\Nst}\!-c(T)\|$ & rate & $\|\bm{u}_h^{\Nst}\!-\bm{u}(T)\|$ & rate & $\|p_h^{\Nst}\!-p(T)\|$ & rate\\
\midrule
$1$ & $1/2^2$  & $6.363$\,E$-2$ & ---     & $2.306$\,E$-2$ & ---     & $1.449$\,E$-1$ & ---      \\
$~$ & $1/2^3$  & $2.599$\,E$-2$ & $1.292$ & $3.342$\,E$-3$ & $2.787$ & $3.417$\,E$-1$ & $-1.238$ \\
$~$ & $1/2^4$  & $8.242$\,E$-3$ & $1.657$ & $7.918$\,E$-4$ & $2.078$ & $1.381$\,E$-1$ & $1.307$  \\
$~$ & $1/2^5$  & $2.241$\,E$-3$ & $1.879$ & $1.988$\,E$-4$ & $1.994$ & $4.082$\,E$-2$ & $1.758$  \\
$~$ & $1/2^6$  & $5.767$\,E$-4$ & $1.958$ & $5.071$\,E$-5$ & $1.971$ & $1.136$\,E$-2$ & $1.845$  \\
\midrule
$2$ & $1/2^1$  & $4.939$\,E$-2$ & ---     & $6.827$\,E$-3$ & ---     & $1.810$\,E$-1$ & ---    \\
$~$ & $1/2^2$  & $2.413$\,E$-2$ & $1.033$ & $1.014$\,E$-3$ & $2.751$ & $3.061$\,E$-1$ & $-0.758$\\
$~$ & $1/2^3$  & $3.112$\,E$-3$ & $2.955$ & $1.327$\,E$-4$ & $2.934$ & $4.210$\,E$-2$ & $2.862$\\
$~$ & $1/2^4$  & $3.285$\,E$-4$ & $3.244$ & $1.512$\,E$-5$ & $3.134$ & $3.377$\,E$-3$ & $3.640$\\
$~$ & $1/2^5$  & $3.741$\,E$-5$ & $3.135$ & $1.638$\,E$-6$ & $3.207$ & $2.972$\,E$-4$ & $3.506$\\
\midrule
$3$ & $1/2^0$  & $1.553$\,E$-1$ & ---     & $5.130$\,E$-3$ & ---     & $9.792$\,E$-1$ & ---    \\
$~$ & $1/2^1$  & $6.856$\,E$-2$ & $1.180$ & $1.035$\,E$-3$ & $2.309$ & $9.538$\,E$-1$ & $0.038$\\
$~$ & $1/2^2$  & $5.764$\,E$-3$ & $3.572$ & $1.761$\,E$-4$ & $2.555$ & $6.414$\,E$-2$ & $3.895$\\
$~$ & $1/2^3$  & $3.279$\,E$-4$ & $4.136$ & $6.186$\,E$-6$ & $4.832$ & $2.365$\,E$-3$ & $4.762$\\
$~$ & $1/2^4$  & $1.989$\,E$-5$ & $4.043$ & $1.829$\,E$-7$ & $5.080$ & $7.923$\,E$-5$ & $4.900$\\
\bottomrule
\end{tabularx}
\caption{Errors and spatial convergence rates of order parameter, velocity, and pressure.}
\label{tab:space_rate}
\end{table}

\subsection{Spinodal decomposition}
Spinodal decomposition serves as a widely used benchmark test for the validation of numerical algorithms on solving the CHNS equations. As a mechanism of phase separation, an initially thermodynamical unstable homogeneous mixture decomposes into two separated phases, which are more thermodynamically favorable. Throughout this process, the mass of the system is conserved and the total energy is dissipated. 
\par
The computational domain, $\Omega = (0,1)^3$, is partitioned uniformly into cubic elements of edge length equal to $10^{-2}$. We select the time step size $\tau=10^{-3}$. The initial order parameter field is generated by sampling numbers from a discrete uniform distribution, namely, $\on{c^0}{E_k} \sim \mathrm{U}\{-1,+1\}$. The initial velocity field is taken to be zero. The polynomial degree is $k=1$. We choose the
viscosity $\mu_\mathrm{s} = 1$ and parameter $\kappa = 10^{-4}$. Following the same notation in Section~\ref{sec:num_exp:convergence}, the penalty parameters are  $\tilde{\sigma}_\mathrm{CH} = 2$, $\tilde{\sigma}_\mathrm{ellip} = 1$, and $\sigma = 8$. We compute $2^{15}$ time steps in total, which is equivalent to the end time $T = 32.768$. \Cref{Fig:spinodal_c_filed} displays 3D snapshots of the order parameter field.  \Cref{Fig:mass_and_energy} shows that 
the mass conservation and energy dissipation properties are satisfied.
\begin{figure}
\begin{tabularx}{\linewidth}{@{}c@{~}c@{~}c@{~}c@{}}
\includegraphics[width=0.25\textwidth]{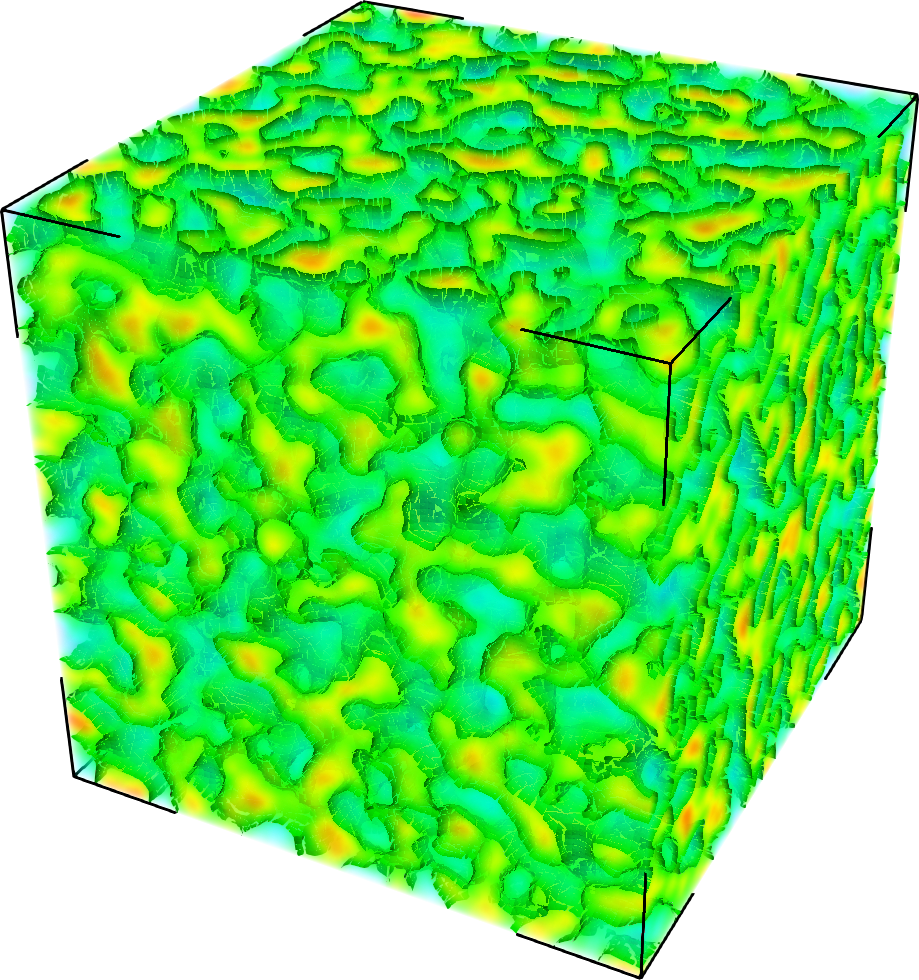} &
\includegraphics[width=0.25\textwidth]{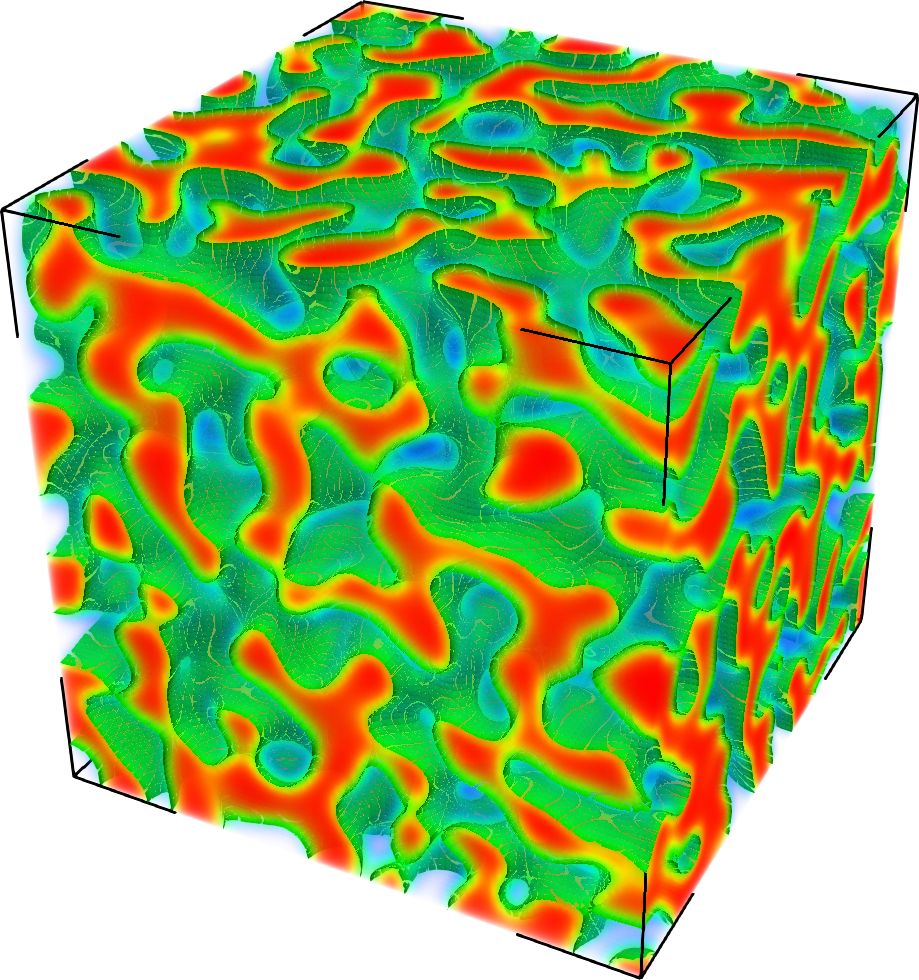} &
\includegraphics[width=0.25\textwidth]{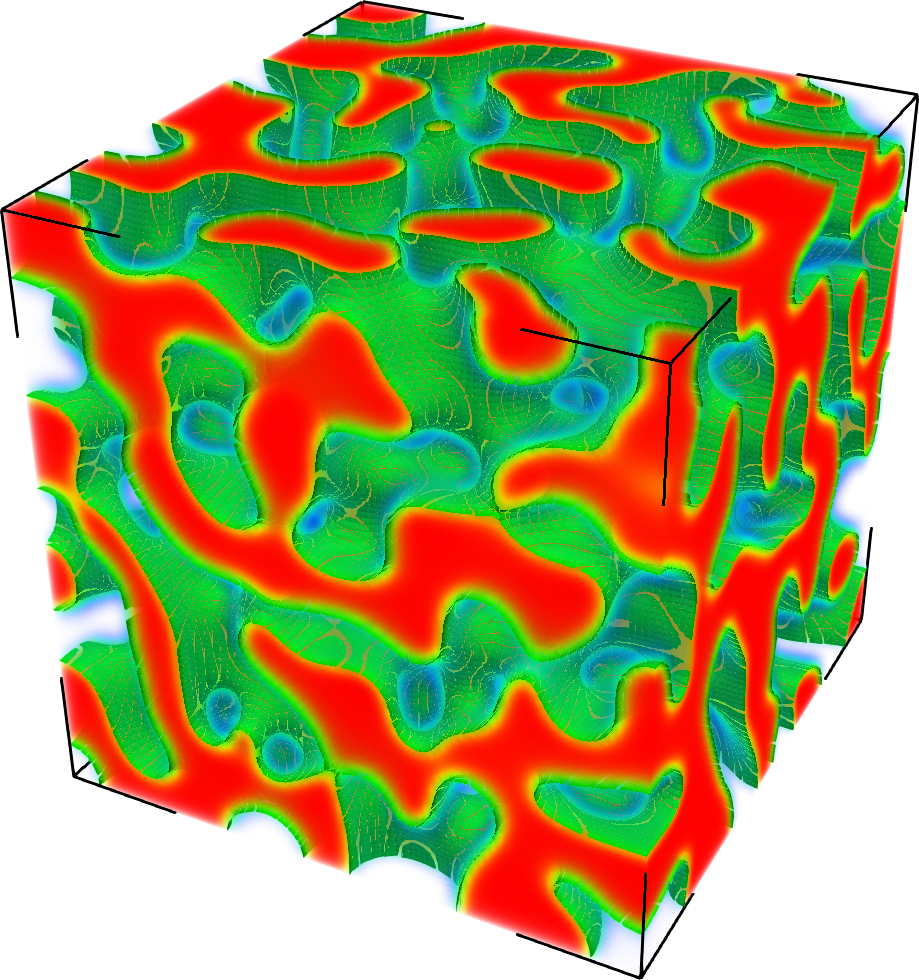} &
\includegraphics[width=0.25\textwidth]{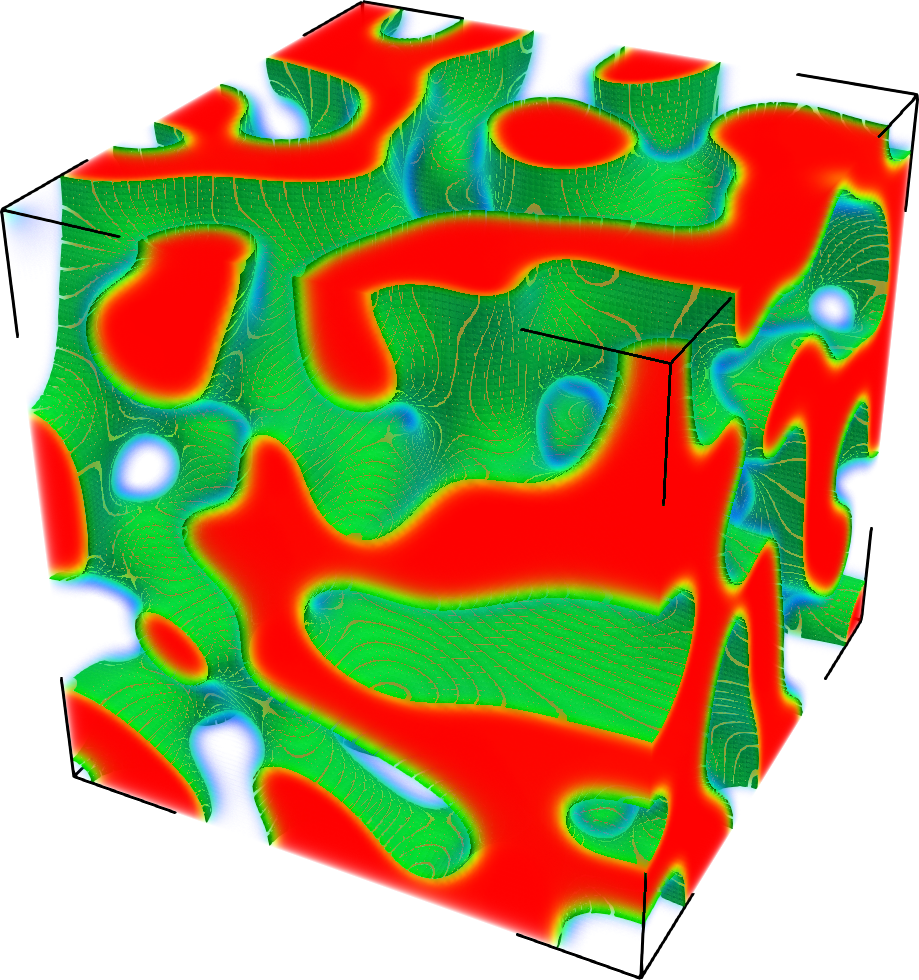} \\
\includegraphics[width=0.25\textwidth]{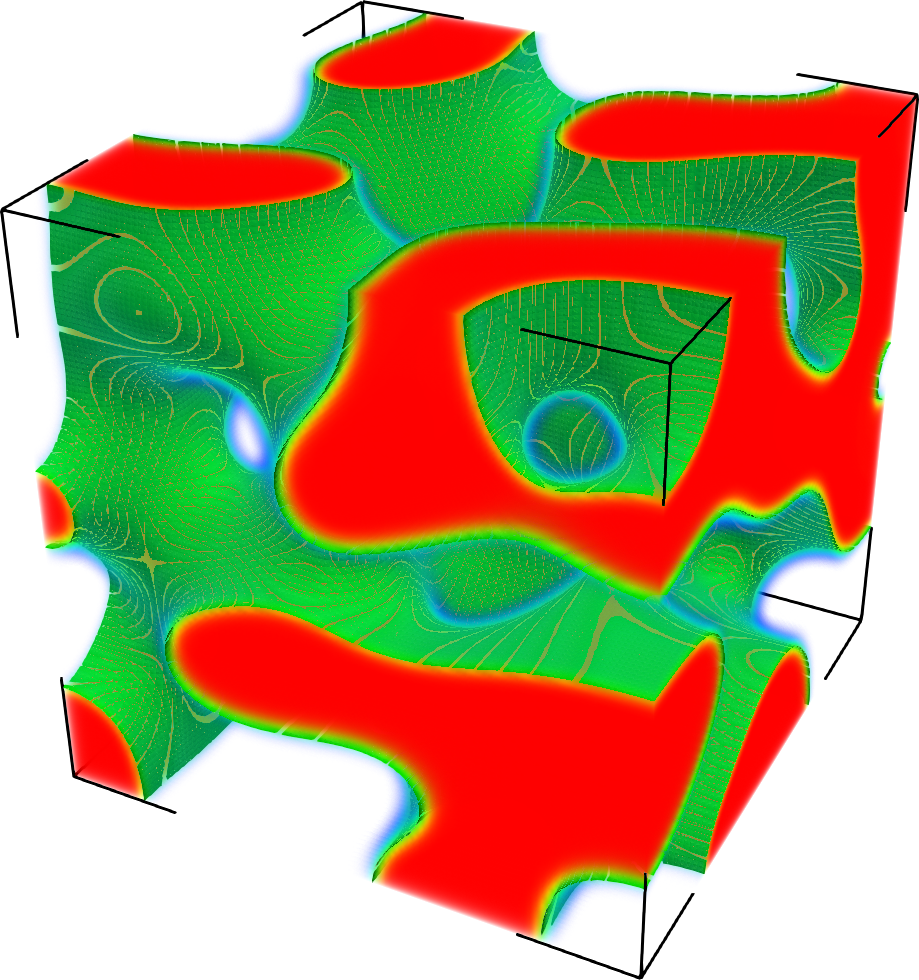} &
\includegraphics[width=0.25\textwidth]{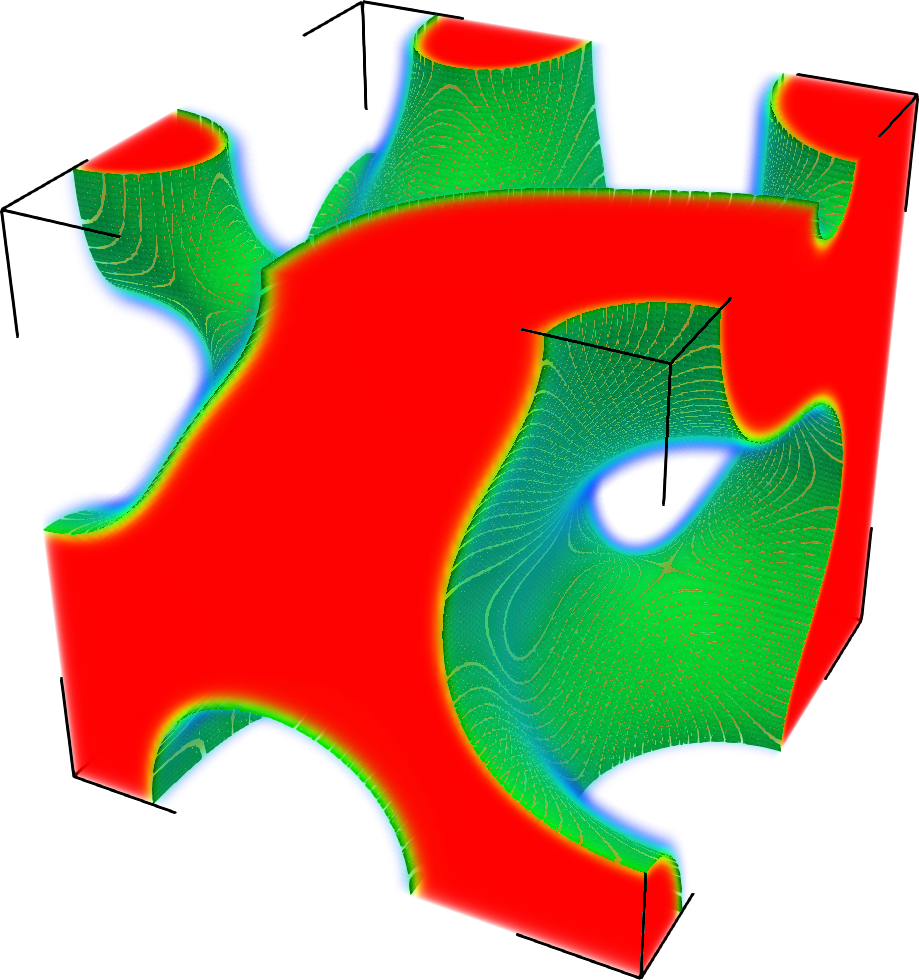} &
\includegraphics[width=0.25\textwidth]{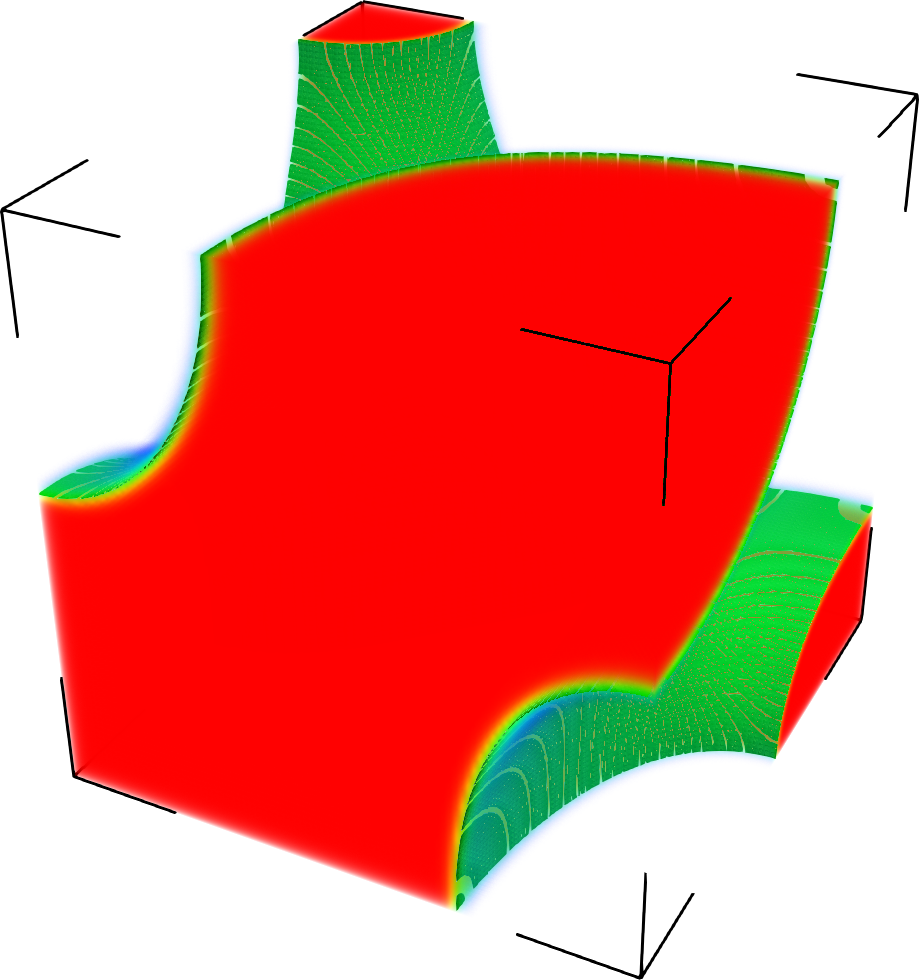} &
\includegraphics[width=0.25\textwidth]{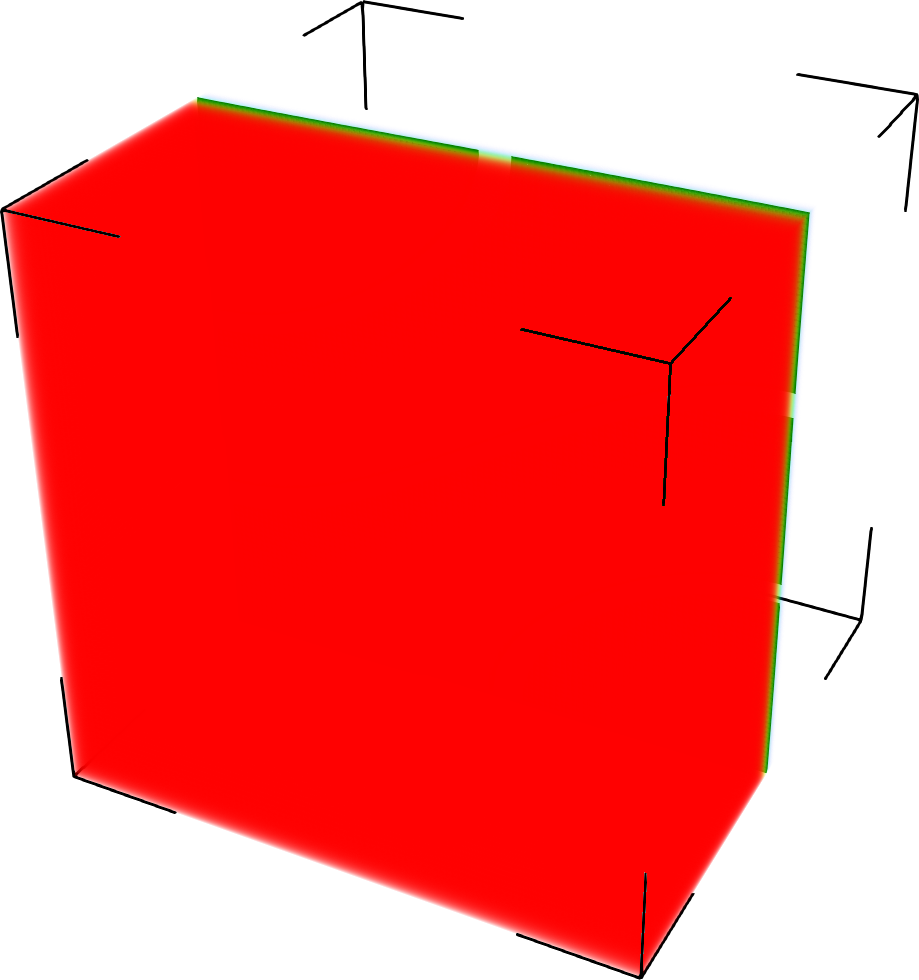} \\
\end{tabularx}
\caption{Selected snapshots for order parameter in spinodal decomposition at time step $2^n$, where $n = 1,3,\cdots,15$. The phase $\mathrm{A}$ (order parameter $c_h^n = +1$) is displayed in red and the phase $\mathrm{B}$ (order parameter $c_h^n = -1$) is in transparent. The green surface corresponds to the center of diffusive interface (order parameter $c_h^n = 0$).} 
\label{Fig:spinodal_c_filed}
\end{figure}
\begin{figure}
\begin{tabularx}{\linewidth}{@{}c@{~}c@{}}
\includegraphics[width=0.5115\textwidth]{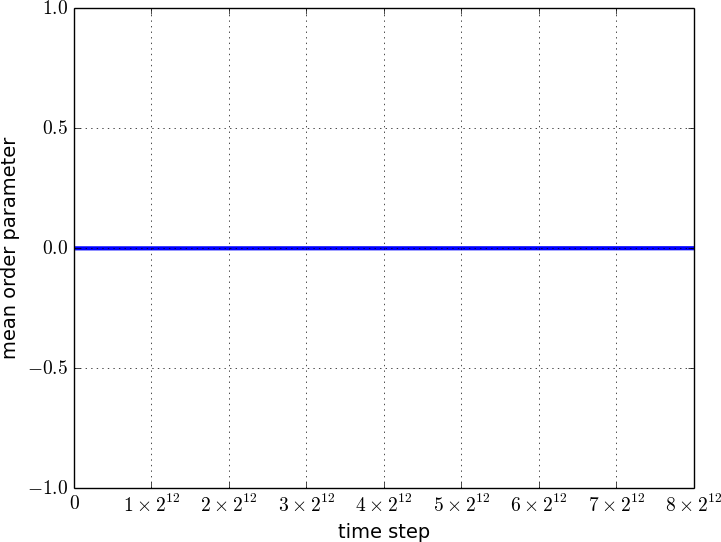} &
\includegraphics[width=0.4885\textwidth]{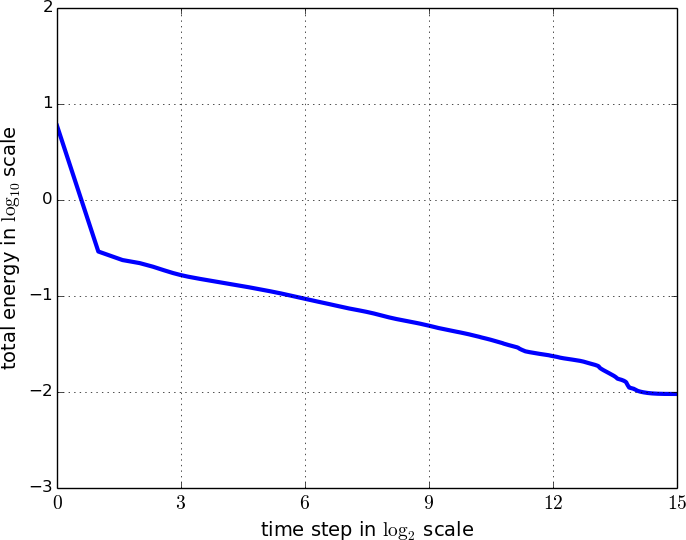} \\
\end{tabularx}
\caption{The total mass of the mixture system (left) and $\log$--$\log$ plot of the total energy (right).}
\label{Fig:mass_and_energy}
\end{figure}

%% file: Content/appendix.tex
\section{Proof of Lemma~\ref{lemma:bound_aadv_terms}}
\label{sec:app1}
\begin{proof}
Let $\mathcal{A} = \aadv(c^{n}, \vec{u}^n,\chi_h ) - \aadv(c_h^{n-1}, \vec{u}_h^{n-1}, \chi_h )$ and write  
\begin{multline}
\mathcal{A} = \aadv(c^n - c^{n-1}, \vec{u}^n,\chi_h ) + \aadv(c^{n-1}, \vec{u}^{n} - \vec{u}^{n-1},\chi_h) - \aadv(\eta_c^{n-1}, \vec{u}^{n-1},\chi_h ) \\    - \aadv(\xi_c^{n-1}, \vec{u}^{n-1},\chi_h)- \aadv(c_h^{n-1}, \vec{\eta}_{\vec{u}}^{n-1},\chi_h) - \aadv(c_h^{n-1},  \vec{\xi}_{\vec{u}}^{n-1},\chi_h) = \sum_{i=1}^6 \mathcal{A}_i. \nonumber
\end{multline}
We begin with bounding $\mathcal{A}_1$. 
With Holder's inequality, trace estimate \eqref{eq:trace_estimate_continuous}, Young's inequality and a Taylor expansion at $t^{n-1}$,  we obtain
\begin{align}
|\mathcal{A}_1| &\leq C \|\vec{u}^n\|_{L^{\infty}(\Omega)}\vert\chi_h \vert_{\mathrm{DG}} (\|c^n - c^{n-1}\|+ h\| \grad{(c^n - c^{n-1})}\|)  \\
& \leq \frac{C}{\epsilon} \tau \|\vec{u}\|_{L^{\infty}(0,T;H^{2}(\Omega))}^2 \int_{t^{n-1}}^{t^{n}} \|\partial_t c\|_{H^1(\Omega)}^2 + \epsilon \vert\chi_h\vert_{\mathrm{DG}}^2. \nonumber
\end{align}
For $\mathcal{A}_2$, similar arguments yield
\begin{align}
|\mathcal{A}_2| &\leq C\|c^{n-1}\|_{L^{\infty}(\Omega)}\vert\chi_h\vert_{\mathrm{DG}}(\|\vec{u}^n - \vec{u}^{n-1}\| + h \| \grad{(\vec{u}^n - \vec{u}^{n-1})}\|)  \\
& \leq  \frac{C}{\epsilon} \tau \|c\|^2_{L^{\infty}(0,T;H^2(\Omega))} \int_{t^{n-1}}^{t^n} \|\partial_t \vec{u}\|^2_{H^1(\Omega)} +  \epsilon  \vert\chi_h \vert_{\mathrm{DG}}^2.\nonumber
\end{align}
For $\mathcal{A}_3$, we utilize the convexity of the computational domain. Namely, using \eqref{eq:elliptic_projection_error}, 
we have
\begin{align}
|\mathcal{A}_3| 
&\leq C\|\vec{u}^{n-1}\|_{L^{\infty}(\Omega)} \vert\chi_h\vert_{\mathrm{DG}}(\|\eta_c^{n-1}\| + h \| \grad_h{\eta_c^{n-1}}\|)\\
&\leq \frac{C}{\epsilon} h^{2k+2} \|\vec{u}\|^2_{L^{\infty}(0,T;H^2(\Omega))}\|c\|^2_{L^{\infty}(0,T;H^2(\Omega))} + \epsilon \vert \chi_h \vert_{\mathrm{DG}}^2. \nonumber
\end{align}
For $\mathcal{A}_4$, with Holder's inequality, trace estimate \eqref{eq:trace_estimate_discrete}, and Young's inequality, we obtain the following bound.
\begin{align}
|\mathcal{A}_4| 
&\leq C\|\xi_c^{n-1}\| \|\vec{u}^{n-1}\|_{L^{\infty}(\Omega)} \vert \chi_h \vert_{\mathrm{DG}}  
\leq \frac{C}{\epsilon} \Vert \xi_c^{n-1}\Vert^2 \|\vec{u}\|_{L^{\infty}(0,T;H^{2}(\Omega))}^2 + \epsilon  \vert\chi_h \vert_{\mathrm{DG}}^2. 
\end{align}
For $\mathcal{A}_5$ and $\mathcal{A}_6$, we use trace inequalities \eqref{eq:trace_estimate_continuous} and \eqref{eq:trace_estimate_discrete} respectively. We obtain
\begin{align*}
|\mathcal{A}_5| + |\mathcal{A}_6| \leq C \|c_h^{n-1}\|_{L^{\infty}(\Omega)} (\|\vec{\eta}_{\vec{u}}^{n-1}\| + h \|\grad_h \vec{\eta}_{\vec{u}}^{n-1}\|+ \|\erru^{n-1}\|)\vert\chi_h\vert_{\mathrm{DG}}\\  \leq \frac{C}{\epsilon} \|c_h^{n-1}\|_{L^\infty(\Omega)}^2( h^{2k+2} \|\vec{u}\|_{L^{\infty}(0,T;H^{k+1}(\Omega))}^2 +
\|\erru^{n-1}\|^2) + \epsilon  \vert \chi_h \vert_{\mathrm{DG}}^2.
\end{align*}
Combining the bounds on $\mathcal{A}_i$ for $i = 1,\ldots, 6$ yields the result.
\end{proof}

\section{Proof of bound \eqref{eq:conv_1add2}}\label{sec:app:conv_1add2}
From Lemma~6.3 in \cite{inspaper1}, 
for all $q_h \in M_h^{k-1}$ and $n \geq 1$, the following equalities holds: 
\begin{align}
    b_{\mathcal{P}}(\erru^n, q_h) =\, &b_{\mathcal{P}}(\errv^n, q_h) + \tau\adif(\phi_h^n,q_h) \label{eq:div_error}\\ &-\tau \sum_{e \in \Gamma_h} \frac{\tilde{\sigma}}{h} \int_e [\phi_h^n][q_h] + \tau(\vec{G}_h([\phi_h^n]), \vec{G}_h([q_h])),\nonumber \\
    b_{\mathcal{P}}(\erru^n, q_h) =\, &-\tau \sum_{e \in \Gamma_h} \frac{\tilde{\sigma}}{h} \int_e [\phi_h^n][q_h] + \tau(\vec{G}_h([\phi_h^n]), \vec{G}_h([q_h])). \label{eq:div_error_2}
\end{align}
In addition, for $n \geq 1$ we have \cite{inspaper1}
\begin{align}
 \|\errv^n -\erru^{n-1}\|^2 =  & \|\erru^{n} - \erru^{n-1} \|^2 + \tau^2(\| \gradh \phi_h^n \|^2 + \|\vec{G}_h([\phi_h^n])\|^2) + \tau^2(A^n_1 -A^n_2) \label{eq:error_diff} \\   
&+  \tau^2\Big( \sum_{e\in\Gamma_h} \frac{\tilde{\sigma}}{h}\|[\phi_h^n - \phi_h^{n-1}]\|^2_{L^2(e)} - \|\vec{G}_h([\phi_h^n - \phi_h^{n-1}])\|^2\Big)  \nonumber\\
&-2\tau^2(\gradh \phi_h^n, \vec{G}_h([\phi_h^n]))  + 2\delta_{n,1} \tau b_{\mathcal{P}}(\vec{\xi}_{\vec{u}}^0, \phi_h^1).  \nonumber
\end{align}
Using \eqref{eq:lift_prop_g} and assuming $\tilde{\sigma} > 2\tilde{M}_{k}^2$, we have 
\begin{align*}
 \|\vec{G}_h([\phi_h^n - \phi_h^{n-1}])\|^2 & \leq  \sum_{e\in \Gamma_h} \frac{\tilde{M}_{k}^2}{h} \| [\phi_h^{n} - \phi_h^{n-1}]\|^2_{L^2(e)} \leq  \frac{1}{2}\sum_{e\in \Gamma_h} \frac{\tilde{\sigma}}{h} \|[\phi_h^{n} - \phi_h^{n-1}]\|^2_{L^2(e)}. 
\end{align*}
Similarly, with Cauchy--Schwarz's inequality, Young's inequality, \eqref{eq:lift_prop_g}, and the assumption that $\tilde{\sigma} \geq 4\tilde{M}_{k}^2$, we obtain 
\begin{align*}
|(\grad \phi_h^n , \vec{G}_h([\phi_h^n ]))|  
& \leq \frac{1}{4}\|\grad \phi_h^n\|^2 + \| \vec{G}_h([\phi_h^n])\|^2 
\leq \frac{1}{4}\|\grad \phi_h^n\|^2 + \sum_{e\in\Gamma_h} \frac{\tilde{M}_{k}^2}{h} \|[\phi_h^n] \|^2_{L^2(e)}\nonumber \\ 
& \leq  \frac{1}{4}\|\grad \phi_h^n\|^2 + \frac14 \sum_{e\in\Gamma_h} \frac{\tilde{\sigma}}{h} \|[\phi_h^n]\|^2_{L^2(e)}. 
\end{align*}
Using the above bounds in \eqref{eq:error_diff}, we obtain: 
\begin{multline}
\frac{1}{2} \|\erru^{n} - \erru^{n-1} \|^2 + \frac{\tau^2}{4} \| \grad \phi_h^n \|^2 + \frac{\tau^2}{2}\|\vec{G}_h([\phi_h^n])\|^2 \\
+ \frac{ \tau^2}{2} (A^n_1 - A^n_2) + \frac{\tau^2}{4} \sum_{e\in\Gamma_h} \frac{\tilde{\sigma}}{h}\|[\phi_h^n - \phi_h^{n-1}]\|^2_{L^2(e)} \leq \frac{1}{2}\|\errv^n  -\erru^{n-1}\|^2 \\
+ \frac{\tau^2}{4} \sum_{e\in\Gamma_h} \frac{\tilde{\sigma}}{h} \|[\phi_h^n]\|^2_{L^2(e)} + \delta_{n,1} \tau \abs{b_{\mathcal{P}}(\vec{\xi}_{\vec{u}}^0, \phi_h^1)}. \label{eq:conv_1} 
\end{multline}  
Inserting $\Pi_h \vec{u}^n$ in \eqref{eq:fully_dis6} yields: 
\begin{equation}
(\erru^n, \vec{\theta}_h) = (\errv^n, \vec{\theta}_h) + \tau b_{\mathcal{P}}(\vec{\theta}_h, \phi_h^n), \quad \forall \vec{\theta}_h \in \mathbf{X}_h^{k}. \label{eq:fourth_err_eq}
\end{equation}
Let $\vec{\theta}_h = \erru^n$ in \eqref{eq:fourth_err_eq} and use \eqref{eq:div_error_2}, we obtain 
\begin{equation}
\frac12 \left( \|\erru^n\|^2 - \|\errv^n\|^2 \right) + \frac12 \|\erru^n - \errv^n \|^2 + \tau^2 \sum_{e\in\Gamma_h} \frac{\tilde{\sigma}}{h} \|[\phi_h^n] \|_{L^2(e)}^2 = \tau^2 \|\vec{G}_h([\phi_h^n]) \|^2. \label{eq:conv_2_step1}
\end{equation}
We let $\vec{\theta}_h = \erru^n - \errv^n$ in \eqref{eq:fourth_err_eq} and use \eqref{eq:div_error}. We have
\begin{align*}
&\|\erru^n - \errv^n\|^2 
= \tau b(\erru^n - \errv^n, \phi_h^n) \\
=&\, \tau^2 \adif(\phi_h^n,\phi_h^n) - \tau^2 \sum_{e\in\Gamma_h} \frac{\tilde{\sigma}}{h} \|[\phi_h^n] \|_{L^2(e)}^2 + \tau^2 \|\vec{G}_h([\phi_h^n]) \|^2.
\end{align*}
Hence, \eqref{eq:conv_2_step1} reads 
\begin{equation}
    \frac12 \left( \|\erru^n \|^2 - \|\errv^n\|^2 \right) + \frac{\tau^2}{2} \adif(\phi_h^n,\phi_h^n) + \frac{\tau^2}{2} \sum_{e\in \Gamma_h} \frac{\tilde{\sigma}}{h} \|[\phi_h^n] \|^2_{L^2(e)} = \frac{\tau^2}{2} \|\vec{G}_h([\phi_h^n])\|^2.  \label{eq:conv_2}
\end{equation}
Taking the sum of \eqref{eq:conv_1} with \eqref{eq:conv_2} and multiplying by $1/\tau$ yield the result. 

\section{Proof of bound \eqref{eq:boundUhpaper}}\label{sec:app3}
We define the norm $\vertiii{\cdot} = \norm{\cdot}{L^\infty(\Omega)} + \norm{\cdot}{W^{1,3}(\Omega)}$.
To start the proof, we first note that: 
\begin{align}
\|\vec{U}(t)-\vec{U}_h(t)\| + h\|\vec{U}(t) -\vec{U}_h(t) \|_{\DG} + h\|P(t) - P_h(t)\| &\leq Ch^2\|\vec{\chi}_{\vec{u}}(t)\|, \label{eq:error_dg_aux} \\  
\vertiii{\vec{U}(t)} + \|\vec{U}_h(t)\|_{L^{\infty}(\Omega)} &\leq C \| \vec{\chi}_{\vec{u}}(t) \|. \label{eq:Linf_bd_Uh} 
\end{align}
The above estimates result from the error analysis of the dG formulation for the dual problem \eqref{eq:aux_pb_1}-\eqref{eq:aux_pb_3}. For more details, we refer to Lemma~5 in \cite{masri2021improved}. 
Choosing $\vec{\theta}_h = \vec{U}_h^n$ in \eqref{eq:third_err_eq} and multiplying by $\tau$ yields
\begin{align}\label{eq:first_error_eq_dual} 
(\vec{v}_h^n - \vec{u}^n - \vec{\chi}_{\vec{u}}^{n-1}\!,\vec{U}_h^n) + \tau \tilde{R}_{\mathcal{C}} (\vec{U}_h^n) + \tau \mu_\mathrm{s} a_\mathcal{D} (\vec{v}_h^n - \vec{u}^n, \vec{U}_h^n) = \tau b_{\mathcal{P}}(\vec{U}_h^n, p_h^{n-1}- p^n) \\
+ (\tau (\partial_t \vec{u})^n - (\vec{u}^n - \vec{u}^{n-1}), \vec{U}_h^n) + \tau b_{\mathcal{I}}(c_h^{n-1},\mu_h^n, \vec{U}^n_h) - \tau b_{\mathcal{I}}(c^n, \mu^n, \vec{U}^n_h),\nonumber
\end{align}
where 
\begin{align*}
\tilde{R}_{\mathcal{C}}(\vec{U}_h^n) = a_\mathcal{C}(\vec{u}_h^{n-1}; \vec{u}_h^{n-1}, \vec{v}_h^n, \vec{U}_h^{n}) -  a_\mathcal{C}(\vec{u}^{n}; \vec{u}^{n}, \vec{u}^n, \vec{U}_h^{n}).
\end{align*}
Let us begin with the first term on the right-hand side of \eqref{eq:first_error_eq_dual}.
We use the fact that $\vec{U}_h^n$ satisfies \eqref{eq:aux_2_dg} and the definition of $\pi_h p^n$ to obtain the following.  
\begin{align*}
b_{\mathcal{P}}(\vec{U}_h^n, p_h^{n-1} - p^n) 
= -b_{\mathcal{P}}(\vec{U}_h^n, p^n - \pi_h p^n) 
= \sum_{e \in \Gamma_h \cup \partial{\Omega}}\int_e \avg{p^n - \pi_h p^n}\jump{\vec{U}_h^n}\cdot \normal_e.  
\end{align*}
Let $\Delta_e$ denote the union of the two elements sharing a face $e$. By a trace inequality, approximation property \eqref{eq:l2_proj_approximation}, \eqref{eq:error_dg_aux}, and the fact that $\jump{\vec{U}^n} = \vec{0}$ a.e. on any face $e \in \Gamma_h \cup \partial \Omega$ since $\vec{U}^n \in H_0^2(\Omega)^d$, we obtain  for any $\epsilon>0$.
\begin{align}\label{eq:bd_b_dual}
\abs{b_{\mathcal{P}}(\vec{U}_h^n, p_h^{n-1} - p^n)} 
\leq C h^{k+1} \vert p^n \vert_ {H^k(\Omega)} \| \vec{\chi}_{\vec{u}}^n \| 
\leq \epsilon \mu_\mathrm{s} \| \vec{\xi}_\vec{u}^n \|^2 + C h^{2k+2} \Big(1+\frac{1}{\epsilon\mu_\mathrm{s}}\Big).
\end{align}
In the above, we used that 
\begin{equation}\label{eq:boundchin}
\|\vec{\chi}_{\vec{u}}^n\| \leq C(h^{k+1}| \vec{u}^n |_{H^{k+1}(\Omega)} + \|\vec{\xi}_\vec{u}^n\|),
\end{equation}
which is obtained by applying the triangle inequality and approximation property \eqref{eq:approximation_prop_1}. This bound will be used repeatedly in this proof. For the second term on the right-hand side of \eqref{eq:first_error_eq_dual}, we simply have: 
\begin{equation}\label{eq:bounding_time_err_dual}
|(\tau (\partial_t \vec{u})^n - (\vec{u}^n - \vec{u}^{n-1}),  \vec{U}_h^n)| 
\leq C\tau^2 \int_{t^{n-1}}^{t^n} \|\partial_{tt} \vec{u}\|^2 + \tau \|\vec{U}_h^n\|_{\DG}^2. 
\end{equation} 
We now consider the terms on the left-hand side of \eqref{eq:first_error_eq_dual}. With \eqref{eq:fully_dis6} and \eqref{eq:aux_2_dg}, we have
\begin{align*}
(\vec{v}_h^n - \vec{u}^n - \vec{\chi}_{\vec{u}}^{n-1},\vec{U}_h^n) 
= (\vec{\chi}_{\vec{u}}^{n} - \vec{\chi}_{\vec{u}}^{n-1}, \vec{U}_h^n) - \tau b_{\mathcal{P}}(\vec{U}_h^n, \phi_h^n) 
= (\vec{\chi}_{\vec{u}}^{n} - \vec{\chi}_{\vec{u}}^{n-1}, \vec{U}_h^n).
\end{align*}
Note that from \eqref{eq:aux_1_dg}, \eqref{eq:aux_2_dg}, the above equality, and the symmetry of $a_{\mathcal{D}}(\cdot, \cdot)$, we have 
\begin{align}\label{eq:expanding_aepsi_dual}
&(\vec{v}_h^n - \vec{u}^n - \vec{\chi}^{n-1},\vec{U}_h^n) 
= a_{\mathcal{D}}(\vec{U}_h^n - \vec{U}_h^{n-1}, \vec{U}_h^n)\\
=&\, \frac{1}{2} \big(a_{\mathcal{D}}(\vec{U}_h^n, \vec{U}_h^n) - a_{\mathcal{D}}(\vec{U}_h^{n-1}, \vec{U}_h^{n-1}) +  a_{\mathcal{D}}(\vec{U}_h^{n} - \vec{U}_h^{n-1}, \vec{U}_h^{n} - \vec{U}_h^{n-1})\big). \nonumber
\end{align}
In addition, we write 
\begin{align}\label{eq:breaking_up_aD_first_err_dual}
a_\mathcal{D}(\vec{v}_h^n - \vec{u}^n , \vec{U}_h^n) 
= a_\mathcal{D}(\vec{\xi}_{\vec{v}}^n - \vec{\xi}_{\vec{u}}^n, \vec{U}_h^n) 
+ a_\mathcal{D}(\vec{\xi}_{\vec{u}}^n, \vec{U}_h^n) 
+ a_\mathcal{D}(\vec{\eta}_{\vec{u}}^n, \vec{U}_h^n).
\end{align}
To handle the last term in above equality, let $\mathcal{Q}_h \vec{u}^n$ be the elliptic projection of $\vec{u}^n$ onto the space $\mathbf{X}_h$. Since the domain is convex, this projection satisfies \cite{riviere2008}: 
\begin{equation}
\forall \vec{\theta}_h \in \mathbf{X}_h,~ a_{\mathcal{D}}(\vec{u}^n - \mathcal{Q}_h \vec{u}^n , \vec{\theta}_h) = 0 
~\mathrm{and}~ 
\| \vec{u}^n - \mathcal{Q}_h \vec{u}^n \| \leq Ch^{k+1}|\vec{u}^n|_{H^{k+1}(\Omega)}. \label{eq:elliptic_projection}
\end{equation}
Let $\vec{\theta}_h = \Pi_h \vec{u}^n - \mathcal{Q}_h \vec{u}^n$ in \eqref{eq:aux_1_dg}. We obtain 
\begin{align*}
a_{\mathcal{D}}(\vec{\eta}_{\vec{u}}^n, \vec{U}_h^n) 
= a_{\mathcal{D}} (\Pi_h \vec{u}^n - \mathcal{Q}_h \vec{u}^n, \vec{U}_h^n) 
= (\vec{\chi}_{\vec{u}}^n , \Pi_h \vec{u}^n - \mathcal{Q}_h\vec{u}^n ) + b_{\mathcal{P}}(\Pi_h \vec{u}^n - \mathcal{Q}_h\vec{u}^n, P_h^n).
\end{align*}
We have  
\begin{align*}
\abs{b_{\mathcal{P}}(\Pi_h \vec{u}^n - \mathcal{Q}_h\vec{u}^n, P_h^n)} 
\leq C\|\Pi_h \vec{u}^n - \mathcal{Q}_h \vec{u}^n \| |P_h^n|_{\DG}. 
\end{align*}
Further, with approximation properties and \eqref{eq:boundchin}, we obtain 
\begin{align}\label{eq:dg_bound_ph}
|P_h^n|_{\DG} &\leq |P_h^n - \pi_h P^n|_{\DG} + |\pi_h P^n|_{\DG} \leq Ch^{-1}\|P_h^n - \pi_h P^n\|+ C| P^n|_{H^1(\Omega)}\\ 
&\leq C (\|\vec{\xi}_{\vec{u}}^n\| + h^{k+1}|\vec{u}^n|_{H^{k+1}(\Omega)}).\nonumber
\end{align}
Hence, Cauchy--Schwarz's inequality, the above bounds, the regularity assumption that $\vec{u} \in L^{\infty}(0,T;H^{k+1}(\Omega)^d)$,  and Young's inequality yield for any $\epsilon>0$: 
\begin{align}\label{eq:pih_un_u_dual}
&| a_{\mathcal{D}} (\Pi_h \vec{u}^n - \vec{u}^n, \vec{U}_h^n) |  \leq C  \|\Pi_h \vec{u}^n  - \mathcal{Q}_h \vec{u}^n \|(\|\vec{\xi}_{\vec{u}}^n\| + h^{k+1}|\vec{u}^n|_{H^{k+1}(\Omega)})    
\\ \leq&\, C h^{k+1 } |\vec{u}^n|_{H^{k+1}(\Omega)}(\|\vec{\xi}_{\vec{u}}^n\| + h^{k+1}|\vec{u}^n|_{H^{k+1}(\Omega)}) \leq  \epsilon  \|\vec{\xi}_{\vec{u}}^n \|^2 + C\left(\frac{1}{\epsilon} + 1  \right)  h^{2k+2}. \nonumber
\end{align} 
Consider now the second term in \eqref{eq:breaking_up_aD_first_err_dual}.  Letting $\vec{\theta}_h = \vec{\xi}_{\vec{u}}^n$ in \eqref{eq:aux_1_dg},  we obtain
\begin{equation}
a_{\mathcal{D}}(\vec{\xi}_{\vec{u}}^n,\vec{U}_h^n) = \| \vec{\xi}_{\vec{u}}^n \|^2 + (\vec{\chi}_{\vec{u}}^n - \vec{\xi}_{\vec{u}}^n , \vec{\xi}_{\vec{u}}^n) +  b_\mathcal{P}(\vec{\xi}_{\vec{u}}^n, P_h^n). \label{eq:rewriting_a_ep_dual}
\end{equation}
With \eqref{eq:bd_b_dual}, \eqref{eq:bounding_time_err_dual}, \eqref{eq:expanding_aepsi_dual}, \eqref{eq:breaking_up_aD_first_err_dual}, \eqref{eq:pih_un_u_dual}, and \eqref{eq:rewriting_a_ep_dual}, the equality \eqref{eq:first_error_eq_dual} becomes 
\begin{align}\label{eq:second_err_eq_dual}
&\frac{1}{2} \big(a_{\mathcal{D}}(\vec{U}_h^n, \vec{U}_h^n) - a_{\mathcal{D}}(\vec{U}_h^{n-1}, \vec{U}_h^{n-1}) + a_{\mathcal{D}}(\vec{U}_h^{n} - \vec{U}_h^{n-1}\!, \vec{U}_h^{n} - \vec{U}_h^{n-1})\big)
+ \tau \mu_\mathrm{s} \| \vec{\xi}_{\vec{u}}^n \|^2 \\
\leq&\, 
C \Big(\mu_\mathrm{s}\Big(1+\frac{1}{\epsilon}\Big) + 1 + \frac{1}{\epsilon\mu_\mathrm{s}}\Big) \tau h^{2k+2} 
+ C\tau^2 \int_{t^{n-1}}^{t^n}  \| \partial_{tt} \vec{u} \|^2 
+ C \tau \|\vec{U}_h^n\|_{\DG}^2 \nonumber
\\ 
-&\, \tau \tilde{R}_{\mathcal{C}}(\vec{U}_h^n) 
+ \tau b_{\mathcal{I}}(c_h^{n-1},\mu_h^n, \vec{U}^n_h) - \tau b_{\mathcal{I}}(c^n, \mu^n, \vec{U}^n_h) 
+ 2\epsilon \tau \mu_\mathrm{s} \|\vec{\xi}_{\vec{u}}^n\|^2 \nonumber\\
-&\, \tau \mu_\mathrm{s} a_{\mathcal{D}}(\vec{\xi}_{\vec{v}}^n - \vec{\xi}_{\vec{u}}^n, \vec{U}_h^n) -\tau \mu_\mathrm{s} (\vec{\eta}_{\vec{u}}^n, \vec{\xi}_\vec{u}^n) - \tau \mu_\mathrm{s} b_{\mathcal{P}}(\vec{\xi}_{\vec{u}}^n, P_h^n). \nonumber 
\end{align}
We now handle the last three terms in the above bound. 
Let $\vec{\theta}_h = \vec{\xi}_{\vec{v}}^n - \vec{\xi}_{\vec{u}}^n$ in \eqref{eq:aux_1_dg}. 
\begin{align} 
a_{\mathcal{D}}( \vec{\xi}_{\vec{v}}^n - \vec{\xi}_{\vec{u}}^n,\vec{U}_h^n)
&= (\vec{\xi}_{\vec{u}}^n, \vec{\xi}_{\vec{v}}^n - \vec{\xi}_{\vec{u}}^n)  +  (\vec{\eta}_{\vec{u}}^n, \vec{\xi}_{\vec{v}}^n - \vec{\xi}_{\vec{u}}^n) + b_{\mathcal{P}}(\vec{\xi}_{\vec{v}}^n - \vec{\xi}_{\vec{u}}^n,  P_h^n). 
\end{align} 
Recall that by \eqref{eq:fully_dis6}, \eqref{eq:div_error_2}, \eqref{eq:lift_prop_g}, and the assumption that $\tilde{\sigma} \geq  \tilde{M}_k^2$,  we have 
\begin{equation}
(\vec{\xi}_{\vec{u}}^n,\vec{\xi}_{\vec{v}}^n - \vec{\xi}_{\vec{u}}^n) 
= -\tau b_{\mathcal{P}} ( \vec{\xi}_{\vec{u}}^n,\phi_h^n) 
= \tau^2 \sum_{e\in \Gamma_h} \frac{\tilde{\sigma}}{h_e} \|\jump{\phi_h^n}\|_{L^2(e)}^2  - \tau^2 \| \vec{G}_h([\phi_h^n])\|^2  \geq 0.
\end{equation}
Using Cauchy--Schwarz's inequality,  \eqref{eq:approximation_prop_1} and \eqref{eq:erruvint4}, we have the following bound. 
\begin{align}
\abs{(\vec{\eta}_{\vec{u}}^n , \vec{\xi}_{\vec{v}}^n - \vec{\xi}_{\vec{u}}^n)} 
+ \abs{b_{\mathcal{P}}(\vec{\xi}_{\vec{v}}^n - \vec{\xi}_{\vec{u}}^n , P_h^n)} 
\leq &C \|\vec{\xi}_{\vec{v}}^n - \vec{\xi}_{\vec{u}}^n \| (h^{k+1}|\vec{u}^n|_{H^{k+1} (\Omega)} + \vert P_h^n\vert_{\DG})\nonumber\\
\leq &C \tau \vert \phi_h^n \vert_{\mathrm{DG}}  (h^{k+1}|\vec{u}^n|_{H^{k+1} (\Omega)} + \vert P_h^n\vert_{\DG}).   \label{eq:bound105}
\end{align}
Therefore with \eqref{eq:dg_bound_ph}, the regularity assumption that $\vec{u} \in L^{\infty}(0,T; H^{k+1}(\Omega)^d)$,   and Young's inequality, the bound \eqref{eq:bound105} becomes 
\begin{align}
| (\vec{\eta}_{\vec{u}}^n, \vec{\xi}_{\vec{v}}^n - \vec{\xi}_{\vec{u}}^n) |  & +   |b_\mathcal{P}(\vec{\xi}_{\vec{v}}^n - \vec{\xi}_{\vec{u}}^n, P_h^n)|\leq  C \tau \vert \phi_h^n\vert_{\DG}( h^{k+1}|\vec{u}^n|_{H^{k+1} (\Omega)} + \|  \vec{\xi}_{\vec{u}}^n\|) \\ & \leq  \epsilon \| \vec{\xi}_{\vec{u}}^n\|^2+ C \left(1+ \frac{1}{\epsilon}\right) \tau^2 |\phi_h^n|^2_{\DG}  + Ch^{2k+2}.\nonumber
\end{align}
With \eqref{eq:approximation_prop_1}, we have 
\begin{align} 
|(\vec{\eta}_{\vec{u}}^n, \vec{\xi}_\vec{u}^n )| \leq \epsilon \|\vec{\xi}_\vec{u}^n\|^2 + \frac{C}{\epsilon} h^{2k+2} |\vec{u}^n|^2_{H^{k+1}(\Omega)}.
\end{align}
To handle the last term in \eqref{eq:second_err_eq_dual}, we use \eqref{eq:div_error_2}, \eqref{eq:lift_prop_g}, and  \eqref{eq:dg_bound_ph}. 
\begin{align}
&\abs{b_{\mathcal{P}}(\vec{\xi}_{\vec{u}}^n, P_h^n)}  
= \Big| -\tau  \sum_{e \in \Gamma_h} \frac{\tilde{\sigma}}{h_e} \int_e \jump{\phi_h^n}\jump{P_h^n} + \tau (\vec{G}_h(\jump{\phi_h^n}),\vec{G}_h(\jump{P_h^n}))\Big| \\
\leq&\, C\tau \vert \phi_h^n \vert_{\DG} \vert P_h^n \vert_{\DG} \leq  \epsilon \|\vec{\xi}_{\vec{u}}^n\|^2  + C \left(1+ \frac{1}{\epsilon}\right) \tau^2 |\phi_h^n|^2_{\DG}+ Ch^{2k+2}.\nonumber
\end{align} 
With the above bounds combined, \eqref{eq:second_err_eq_dual} becomes  
\begin{multline}\label{eq:improved_estimate_semi_final}
\frac{1}{2} \big(a_{\mathcal{D}}(\vec{U}_h^n, \vec{U}_h^n) - a_{\mathcal{D}}(\vec{U}_h^{n-1}, \vec{U}_h^{n-1}) + a_{\mathcal{D}}(\vec{U}_h^{n} - \vec{U}_h^{n-1}\!, \vec{U}_h^{n} - \vec{U}_h^{n-1})\big) 
+ \tau \mu_\mathrm{s} \| \vec{\xi}_{\vec{u}}^n \|^2\\ 
\leq C\Big(\mu_\mathrm{s}\Big(1+\frac{1}{\epsilon}\Big) + 1 + \frac{1}{\epsilon \mu_\mathrm{s}}\Big)\tau h^{2k+2}
+ C\tau^2 \int_{t^{n-1}}^{t^n} \|\partial_{tt} \vec{u}\|^2 
+ C \tau \|\vec{U}_h^n\|_{\DG}^2
- \tau \tilde{R}_{\mathcal{C}}(\vec{U}_h^n)
\\
+ \tau b_{\mathcal{I}}(c_h^{n-1},\mu_h^n, \vec{U}^n_h) - \tau  b_{\mathcal{I}}(c^n, \mu^n, \vec{U}^n_h)
+ 5\epsilon \tau \mu_\mathrm{s} \|\vec{\xi}_{\vec{u}}^n\|^2 + C \mu_\mathrm{s} \left(1+ \frac{1}{\epsilon}\right)\tau^3 \vert \phi_h^n \vert_{\DG}^2.
\end{multline}
The bound for the nonlinear term $\tilde{R}_{\mathcal{C}}(\vec{U}_h^n)$ is technical and can be found in \cite{masri2021improved}. 
Namely, we have 
\begin{align}\label{eq:bound_R_C_improved}
& |\tilde{R}_{\mathcal{C}}(\vec{U}_h^n)| \leq  7\epsilon \mu_\mathrm{s}\|\vec{\xi}_{\vec{u}}^n\|^2 +2\epsilon\|\vec{\xi}_{\vec{u}}^{n-1} - \vec{\xi}_{\vec{u}}^n\|^2  + C\Big(\frac{h^2}{\epsilon \mu_\mathrm{s}} + 1\Big)(\tau^2 |\phi_h^n|^2_{\DG} + h^{2k+2}) \\ 
& + C\Big( \frac{1}{\epsilon\mu_\mathrm{s}} + 1\Big)\big( h^{2}(\|\vec{\xi}_{\vec{u}}^n\|^2 + \|\vec{\xi}_{\vec{u}}^{n-1}\|^2)+(\|\vec{\xi}_{\vec{u}}^{n-1}\|^2+h^2)\| \vec{\xi}_{\vec{v}}^n \|_{\DG}^2 \big)\nonumber \\ 
& + C\tau \int_{t^{n-1}}^{t^n} \|\partial_t \vec{u}\|^2 + C\Big( \frac{1}{\epsilon\mu_\mathrm{s}} + 1\Big) \|\vec{U}_h^n\|_{\DG}^2. \nonumber 
\end{align}
%
%
We use \eqref{eq:bound_R_C_improved} in \eqref{eq:improved_estimate_semi_final}, use the coercivity property \eqref{eq:coercivity_astrain}, and choose $\epsilon = 1/24$. We sum the resulting equation,  from $n= 1$ to $n =m$, use the regularity assumptions, and obtain the following. 
\begin{align}\label{eq:error_eq_semi_final_improved_estimate}
&\frac{1}{2} a_{\mathcal{D}}(\vec{U}_h^m, \vec{U}_h^m)  - \frac{1}{2}a_{\mathcal{D}}( \vec{U}_h^{0}, \vec{U}_h^{0}) + \frac{K_\mathcal{D}}{4} \sum_{n = 1}^m  \|\vec{U}_h^n - \vec{U}_h^{n-1}\|_{\DG}^2 + \frac{\mu_\mathrm{s}\tau}{2} \sum_{n=1}^m \| \vec{\xi}_{\vec{u}}^n \|^2 \\
\leq&\, C\Big(1+ \frac{1}{\mu_\mathrm{s}} + \mu_\mathrm{s}\Big) h^{2k+2} + C\tau^2  + C \tau^3  \Big(\frac{h^2}{\mu_\mathrm{s}} +  1 + \mu_\mathrm{s} \Big)\sum_{n=1}^m \vert \phi_h^n \vert_{\DG}^2 \nonumber \\
+&\, \frac{\tau}{12} \sum_{n=1}^m \| \vec{\xi}_{\vec{u}}^n - \vec{\xi}_{\vec{u}}^{n-1} \|^2 + C \tau \Big( \frac{1}{\mu_\mathrm{s}} + 1 \Big) \sum_{n=1}^m \|\vec{U}_h^n\|_{\DG}^2 \nonumber \\ 
+&\, C\Big(\frac{1}{\mu_\mathrm{s}}+1\Big)\tau \sum_{n=1}^m \big( h^{2}(\|\vec{\xi}_{\vec{u}}^n\|^2  + \|\vec{\xi}_{\vec{u}}^{n-1}\|^2)+(\|\vec{\xi}_{\vec{u}}^{n-1}\|^2+h^2)\| \vec{\xi}_{\vec{v}}^n \|_{\DG}^2 \big) \nonumber \\ 
+&\, \tau \sum_{n=1}^m \big(b_{\mathcal{I}}(c_h^{n-1},\mu_h^n, \vec{U}^n_h) - b_{\mathcal{I}}(c^n, \mu^n, \vec{U}^n_h)\big).\nonumber
\end{align}
Note that by \eqref{eq:aux_1_dg}-\eqref{eq:aux_2_dg}, we easily see that $\vec{U}_h^0 = \vec{0}$. 
With \eqref{eq:error_estimate_theorem} and coercivity of $a_\mathcal{D}$ \eqref{eq:coercivity_astrain}, we have
\begin{align*}
& K_\mathcal{D} \|\vec{U}_h^m\|^2_{\DG} 
+  \mu_\mathrm{s} \tau \sum_{n=1}^m \| \vec{\xi}_{\vec{u}}^n\|^2  \leq 
C(C_\mathrm{err}) (\tau^2+h^{2k+2} + \tau h^2) 
\\
+&\, C \tau \Big( \frac{1}{\mu_\mathrm{s}} + 1 \Big) \sum_{n=1}^m \|\vec{U}_h^n\|_{\DG}^2
+ \tau \sum_{n=1}^m \big(b_{\mathcal{I}}(c_h^{n-1},\mu_h^n, \vec{U}^n_h) - b_{\mathcal{I}}(c^n, \mu^n, \vec{U}^n_h)\big). 
\end{align*}